\documentclass[english]{article}
\usepackage[T1]{fontenc}
\usepackage[latin9]{inputenc}
\usepackage{geometry}
\usepackage{verbatim}
\usepackage{prettyref}
\usepackage{amsthm}
\usepackage{amsmath}
\usepackage{amssymb}
\usepackage[numbers]{natbib}

\makeatletter

\providecommand{\tabularnewline}{\\}

\numberwithin{equation}{section}
\numberwithin{figure}{section}
\theoremstyle{plain}
\newtheorem{thm}{\protect\theoremname}[section]
  \theoremstyle{definition}
  \newtheorem{defn}[thm]{\protect\definitionname}
  \theoremstyle{definition}
  \newtheorem{example}[thm]{\protect\examplename}
  \theoremstyle{plain}
  \newtheorem{prop}[thm]{\protect\propositionname}
  \theoremstyle{remark}
  \newtheorem{rem}[thm]{\protect\remarkname}
  \theoremstyle{plain}
  \newtheorem{cor}[thm]{\protect\corollaryname}
  \theoremstyle{plain}
  \newtheorem{lem}[thm]{\protect\lemmaname}

\usepackage{amsthm}
\usepackage{fullpage}
\usepackage{import}
\usepackage{verbatim}
\usepackage{makeidx}
\usepackage{amssymb}
\usepackage{amsfonts}
\usepackage{latexsym}
\usepackage[all,cmtip]{xy}
\usepackage{stmaryrd}
\usepackage{color}
\usepackage[colorlinks=true]{hyperref}
\usepackage{tikz}
\usetikzlibrary{chains,fit,shapes,positioning,calc}
\makeindex

\newcommand{\Div}{\mathop{\mathrm{div}}\nolimits}

\newcommand{\Hom}{\mathop{\mathrm{Hom}}\nolimits}
\newcommand{\Id}{\mathop{\mathrm{Id}}\nolimits}

\newcommand{\pr}{\mathop{\mathrm{pr}}\nolimits}

\newcommand{\rank}{\mathop{\mathrm{rank}}\nolimits}

\newcommand{\tr}{\mathop{\mathrm{tr}}\nolimits}

\newcommand{\nablacirc}{\;\;\makebox[0pt]{$\nabla$}\makebox[0.4pt]{\raisebox{1.7pt}{$\circ$}}\;\;}
\newcommand{\nablashortmid}{\;\;\makebox[0pt]{$\nabla$}\makebox[0.4pt]{\raisebox{1.7pt}{$\shortmid$}}\;\;}
\def \nnabla {\nablacirc}
\def \lnabla {\nablashortmid}

\newcommand{\shortrightarrowtimes}{\;\;\makebox[0pt]{$\times$}\makebox[0.4pt]{$\shortrightarrow$}\;\;}
\newcommand{\shortleftarrowtimes}{\;\;\makebox[0pt]{$\times$}\makebox[0.4pt]{$\shortleftarrow$}\;\;}
\def \rtriv {\shortrightarrowtimes}
\def \ltriv {\shortleftarrowtimes}

\makeatother

\usepackage{babel}
  \providecommand{\corollaryname}{Corollary}
  \providecommand{\definitionname}{Definition}
  \providecommand{\examplename}{Example}
  \providecommand{\lemmaname}{Lemma}
  \providecommand{\propositionname}{Proposition}
  \providecommand{\remarkname}{Remark}
\providecommand{\theoremname}{Theorem}

\begin{document}

\title{Riemannian Calculus of Variations\\
using Strongly Typed Tensor Calculus}

\author{Victor Dods}

\date{2012.12.11}
\maketitle
\begin{abstract}
In this paper, the notion of strongly typed language will be borrowed
from the field of computer programming to introduce a calculational
framework for linear algebra and tensor calculus for the purpose of
detecting errors resulting from inherent misuse of objects and for
finding natural formulations of various objects. A tensor bundle formalism,
crucially relying on the notion of pullback bundle, will be used to
create a rich type system with which to distinguish objects. The type
system and relevant notation is designed to ``telescope'' to accomodate
a level of detail appropriate to a set of calculations. Various techniques
using this formalism will be developed and demonstrated with the goal
of providing a relatively complete and uniform method of coordinate-free
computation.

The calculus of variations pertaining to maps between Riemannian manifolds
will be formulated using the strongly typed tensor formalism and associated
techniques. Energy functionals defined in terms of first order Lagrangians
are the focus of the second half of this paper, in which the first
variation, the Euler-Lagrange equations, and the second variation
of such functionals will be derived.
\end{abstract}

\section*{Introduction}

Many important differential equations have a variational origin, being
derived as the Euler-Lagrange equations for a particular functional
on some space of functions. The variational approach lends itself
particularly to physics, in which conservation of energy or minimization
of action is a central concept. The naturality of such formulations
can't be understated, as solutions to such problems often depend critically
on the inherent geometry of the underlying objects. For example, solutions
to Laplace's equation for a real valued function (e.g. modeling steady-state
heat flow) on a Riemannian manifold depend qualitatively on the topology
of the manifold (e.g. harmonic functions on a closed Riemannian manifold
are necessarily constant, which makes sense geometrically because
there is no boundary through which heat can escape).

A central concept in the field of software design is that of \textbf{information
hiding} \citep{Parnas}, in which a computer program is organized
into modules, each presenting an abstract public interface. Other
parts of the program can interact only through the presented interface,
and the details of how each module works are hidden, thereby preventing
interference in the implementation details which are not required
by the inherent structure of the module. This concept has clear usefulness
in the field of mathematics as well. For example, there are several
formulations of the real numbers (e.g. equivalence classes of Cauchy
sequences of rational numbers, Dedekind cuts, decimal expansions,
etc), but their particulars are instances of what are known as \textbf{implementation
details}, and the details of each particular implementation are irrelevant
in most areas of mathematics, which only use the inherent properties
of the real numbers as a complete, totally ordered field. Of course,
at certain levels, it is useful or necessary to ``open up the box''
{[}go past the public interface{]} and work with a particular representation
of the real numbers.

Information hiding is characteristic of abstract mathematics, in which
general results are proved about abstract mathematical objects without
using any particular implementation of said objects. These results
can then be used modularly in other proofs, just as the functionality
of a computer program is organized into modularized objects and functions.
For example, a fixed point theorem for contractive mappings on closed
sets in Banach spaces, but a particular application of this theorem
renders an existence and uniqueness theorem for first order ODEs \citep[pgs. 59, 62]{WalterODEs}.

A loose conceptual analogy for modularity is that of diagonalizing
a linear operator. A basis of eigenvectors are chosen so that the
action of the linear operator on each eigenspace has a particularly
simple expression, and distinct eigenspaces do not interact with respect
to the operator's action. In this analogy, the eigenvectors then correspond
to individual lemmas, and the linear operator corresponds to a large
theorem which uses each lemma. Decomposing the proof of the main result
in terms of non-interacting lemmas simplifies the proof considerably,
just as it simplifies the quantification of the linear operator. The
term ``orthogonal'' has been borrowed by software design to describe
two program modules whose functionality is independent \citep[Chapter 4, Section 2]{Raymond}.
Orthogonality in software design is highly desirable as it generally
eases program implementation and program correctness verification,
as the human designers are only capable of keeping track of a certain
finite number of details simultaneously  \citep{Miller}. The scope
of each detail level of the design is limited in complexity, making
the overall design easier to comprehend.

This technique in software design carries over directly to proof design,
where it is desirable (elegant) to write proofs and do calculations
without introducing extraneous details, such as choice of bases in
vector spaces or local coordinates in manifolds. Because such choices
are generally non-unique, they can often obscure the inherent structure
of the relevant objects by introducing artifacts arising from properties
of the particular details used to implement said objects. For example,
the choice of a particular local coordinate chart on a manifold artificially
imposes an additive structure on a neighborhood of the manifold, but
such a structure has nothing to do with the inherent geometry of the
manifold. Furthermore, the descent to this ``lower level'' of calculation
discards some type information, representing points in a manifold
as Euclidean vectors, thereby losing the ability to distinguish points
from different manifolds, or even different localities in the same
manifold.

This paper makes a particular emphasis on natural formulations and
calculations in order to expose the underlying geometric structures
rather than relying on coordinate-based expressions. The construction
of the ``full'' direct sum and ``full'' tensor product bundles
are used in combination with induced covariant derivatives to this
end.

\section*{Notation and Conventions}

Let all vector spaces, manifolds and {[}fiber{]} bundles be real and
finite-dimensional unless otherwise noted (this allows the canonical
identification $V^{**}\cong V$ for a vector space or vector bundle
$V$), and let all tensor products be over $\mathbb{R}$. The unqualified
term ``bundle'' will mean ``fiber bundle''. The Einstein summation
convention will be assumed in situations when indexed tensors are
used for computation.

Unary operators are understood to have higher binding precedence than
binary operators, and super and subscripts are understood to have
the highest binding precedence. For example, the expression $\nabla X_{,M}\circ\phi$
would be parenthesized as $\left(\nabla\left(X_{,M}\right)\right)\circ\phi$.

Apart from the obvious purpose of providing a concise and central
reference for the notation in this paper, the following notation index
serves to illustrate the use of telescoping notation (see \prettyref{sec:Telescoping-Notation}).
The high-level (terse notation which requires the reader to do more
work in type inference but is more agile), mid-level, and low-level
(completely type-specified, requiring little work on the part of the
reader) notations are presented side-by-side with their definitions. 

Let $I\subseteq\mathbb{R}$ be a neighborhood of $0$, let $\epsilon,i$
each be coordinates on $I$, let $A,A_{1},\dots,A_{n},B$ be sets,
let $M,N$ be manifolds, let $\phi\in C^{\infty}\left(M,N\right)$,
let $\pi_{M}^{A}\colon A\to M$ and $\pi_{N}^{H}\colon H\to N$ and
be vector bundles, where $A=E,F,F_{1},\dots,F_{n},G$, let $U,V,V_{1},\dots,V_{n},W$
be vector spaces, and let $c_{i}\in\Gamma\left(F_{i}\otimes_{M}T^{*}M\right)$
such that 
\[
c_{1}\oplus_{M}\dots\oplus_{M}c_{n}\in\Gamma\left(\left(F_{1}\oplus_{M}\dots\oplus_{M}F_{n}\right)\otimes_{M}T^{*}M\right)
\]
is a vector bundle isomorphism.

\noindent \begin{center}
{\small }%
\begin{tabular*}{1\textwidth}{@{\extracolsep{\fill}}|c|c|c|c|}
\hline 
{\small High-} & {\small Mid-} & {\small Low-level} & {\small Description}\tabularnewline
\hline 
\hline 
\multicolumn{4}{|c|}{{\small Variations; variational derivatives; tangent vectors.}}\tabularnewline
\hline 
\multicolumn{2}{|c|}{{\small $m_{\epsilon}$}} & {\small $m$} & {\small Variation of a point in $M$; $I\ni\epsilon\mapsto m_{\epsilon}\in M$;
$m\colon I\to M$.}\tabularnewline
\multicolumn{2}{|c|}{{\small $\delta$}} & {\small $\delta_{\epsilon}$} & {\small Variational derivative; $\delta_{\epsilon}:=\frac{\partial}{\partial\epsilon}\mid_{\epsilon=0}$.}\tabularnewline
\multicolumn{2}{|c|}{{\small $\delta m_{\epsilon}$}} & {\small $\delta_{\epsilon}m$} & {\small Tangent vector; linearization of a variation;}\tabularnewline
\multicolumn{2}{|c|}{} &  & {\small $\delta m_{\epsilon}\in T_{m_{0}}M$; $\delta_{\epsilon}m\in T_{m\left(0\right)}M$.}\tabularnewline
\hline 
\hline 
\multicolumn{4}{|c|}{{\small Projection maps; canonical isomorphisms; bundle-related maps
and spaces.}}\tabularnewline
\hline 
{\small $\pr$} & {\small $\pr_{i}$} & {\small $\pr_{i}^{A_{1}\times\dots\times A_{n}}$} & {\small Set-theoretic projection onto $i$th factor or named factor;}\tabularnewline
 & {\small $\pr_{A_{i}}$} & {\small $\pr_{A_{i}}^{A_{1}\times\dots\times A_{n}}$} & {\small $\pr_{i}^{A_{1}\times\dots\times A_{n}}\colon A_{1}\times\dots\times A_{n}\to A_{i}$.}\tabularnewline
\hline 
{\small $\iota$} & {\small $\iota_{B}$, $\iota^{A}$} & {\small $\iota_{B}^{A}$} & {\small Canonical isomorphism; $\iota_{B}^{A}\colon A\to B$; $\iota_{A}^{B}:=\left(\iota_{B}^{A}\right)^{-1}$.}\tabularnewline
{\small $\pi$} & {\small $\pi_{M}$, $\pi^{F}$} & {\small $\pi_{M}^{F}$} & {\small Bundle projection map; $\pi_{M}^{F}\colon F\to M$.}\tabularnewline
{\small $\rho$} & {\small $\rho_{H}$, $\rho^{\phi^{*}H}$} & {\small $\rho_{H}^{\phi^{*}H}$} & {\small Pullback bundle fiber projection map; $\rho_{H}^{\phi^{*}H}\colon\phi^{*}H\to H$.}\tabularnewline
\hline 
\hline 
\multicolumn{4}{|c|}{{\small Trivial bundle constructions and projection maps.}}\tabularnewline
\hline 
\multicolumn{2}{|c|}{{\small $M\times N\to N$}} & {\small $M\rtriv N\to N$} & {\small Trivial bundle over $N$; $M\rtriv N:=M\times N$;}\tabularnewline
\multicolumn{2}{|c|}{} &  & {\small $\pi_{N}^{M\rtriv N}\colon M\rtriv N\to N,\,\left(m,n\right)\mapsto n$.}\tabularnewline
\multicolumn{2}{|c|}{{\small $M\times N\to M$}} & {\small $M\ltriv N\to M$} & {\small Trivial bundle over $M$; $M\ltriv N:=M\times N$;}\tabularnewline
\multicolumn{2}{|c|}{} &  & {\small $\pi_{M}^{M\ltriv N}\colon M\ltriv N\to M,\,\left(m,n\right)\mapsto m$.}\tabularnewline
\hline 
\hline 
\multicolumn{4}{|c|}{{\small Shared base-space bundle constructions and projection maps.}}\tabularnewline
\hline 
\multicolumn{2}{|c|}{{\small $E\times F\to M$}} & {\small $E\times_{M}F\to M$} & {\small Direct product; $E\times_{M}F:=\coprod_{m\in M}E_{m}\times F_{m}$;}\tabularnewline
\multicolumn{2}{|c|}{} &  & {\small $\pi_{M}^{E\times_{M}F}\left(e,f\right):=\pi_{M}^{E}\left(e\right)\equiv\pi_{M}^{F}\left(f\right)$.}\tabularnewline
\multicolumn{2}{|c|}{{\small $E\oplus F\to M$}} & {\small $E\oplus_{M}F\to M$} & {\small Whitney sum; $E\oplus_{M}F:=\coprod_{m\in M}E_{m}\oplus F_{m}$;}\tabularnewline
\multicolumn{2}{|c|}{} &  & {\small $\pi_{M}^{E\oplus_{M}F}\left(e\oplus f\right):=\pi_{M}^{E}\left(e\right)\equiv\pi_{M}^{F}\left(f\right)$.}\tabularnewline
\multicolumn{2}{|c|}{{\small $E\otimes F\to M$}} & {\small $E\otimes_{M}F\to M$} & {\small Tensor product; $E\otimes_{M}F:=\coprod_{m\in M}E_{m}\otimes F_{m}$;}\tabularnewline
\multicolumn{2}{|c|}{} &  & {\small $\pi_{M}^{E\otimes_{M}F}\left(c^{ij}e_{i}\otimes f_{j}\right):=\pi_{M}^{E}\left(e_{k}\right)\equiv\pi_{M}^{F}\left(f_{\ell}\right)$
(for any $k,\ell$).}\tabularnewline
\hline 
\hline 
\multicolumn{4}{|c|}{{\small Separate base-space bundle constructions and projection maps.}}\tabularnewline
\hline 
\multicolumn{2}{|c|}{{\small $E\times H\to M\times N$}} & {\small $E\times_{M\times N}H\to M\times N$} & {\small Direct product; $E\times_{M\times N}H:=\coprod_{\left(m,n\right)\in M\times N}E_{m}\times H_{n}$.}\tabularnewline
\multicolumn{2}{|c|}{} &  & {\small $\pi_{M\times N}^{E\times_{M\times N}H}\left(e,h\right):=\left(\pi_{M}^{E}\left(e\right),\pi_{N}^{H}\left(h\right)\right)$.}\tabularnewline
\multicolumn{2}{|c|}{{\small $E\oplus H\to M\times N$}} & {\small $E\oplus_{M\times N}H\to M\times N$} & {\small Whitney sum; $E\oplus_{M\times N}H:=\coprod_{\left(m,n\right)\in M\times N}E_{m}\oplus H_{n}$.}\tabularnewline
\multicolumn{2}{|c|}{} &  & {\small $\pi_{M\times N}^{E\oplus_{M\times N}H}\left(e\oplus h\right):=\left(\pi_{M}^{E}\left(e\right),\pi_{N}^{H}\left(h\right)\right)$.}\tabularnewline
\multicolumn{2}{|c|}{{\small $E\otimes H\to M\times N$}} & {\small $E\otimes_{M\times N}H\to M\times N$} & {\small Tensor product; $E\otimes_{M\times N}H:=\coprod_{\left(m,n\right)\in M\times N}E_{m}\otimes H_{n}$.}\tabularnewline
\multicolumn{2}{|c|}{} &  & {\small $\pi_{M\times N}^{E\otimes_{M\times N}H}\left(c^{ij}e_{i}\otimes h_{j}\right):=\left(\pi_{M}^{E}\left(e_{k}\right),\pi_{N}^{H}\left(h_{\ell}\right)\right)$
(for any $k,\ell$).}\tabularnewline
\hline 
\hline 
\multicolumn{4}{|c|}{{\small Trace; natural pairing; tensor/tensor field contraction. Simple
tensor expressions are extended linearly.}}\tabularnewline
\hline 
\multicolumn{2}{|c|}{{\small $\tr$}} & {\small $\tr_{V}$} & {\small Trace on $V$; $\tr_{V}\colon V^{*}\otimes V\to\mathbb{R},\,\alpha\otimes v\mapsto\alpha\left(v\right)$.}\tabularnewline
\multicolumn{2}{|c|}{{\small $\alpha\cdot v$}} & {\small $\alpha\cdot_{V}v$} & {\small Natural pairing; $\cdot_{V}\colon V^{*}\times V\to\mathbb{R},\,\left(\alpha,v\right)\mapsto\alpha\left(v\right)$.}\tabularnewline
\multicolumn{2}{|c|}{{\small $A\cdot B$}} & {\small $A\cdot_{V}B$} & {\small Tensor contraction; $\cdot_{V}\colon\left(U\otimes V^{*}\right)\times\left(V\otimes W\right)\to U\otimes W,$}\tabularnewline
\multicolumn{2}{|c|}{} &  & {\small $\left(u\otimes\alpha\right)\cdot_{V}\left(v\otimes w\right):=u\otimes\left(\alpha\cdot_{V}v\right)\otimes w\equiv\alpha\left(v\right)\, u\otimes w$.}\tabularnewline
\multicolumn{2}{|c|}{{\small $S\cdot^{n}T$}} & {\small $S\cdot_{V_{1}\otimes\dots\otimes V_{n}}T$} & {\small Alternate for $\cdot_{V}$, where $V=V_{1}\otimes\dots\otimes V_{n}$.}\tabularnewline
\multicolumn{2}{|c|}{{\small $S:T$, $S\cdot^{2}T$}} & {\small $S\cdot_{V_{1}\otimes V_{2}}T$} & {\small Special notation for $n=2$.}\tabularnewline
\hline 
\multicolumn{2}{|c|}{{\small $\tr$}} & {\small $\tr_{F}$} & {\small Trace on $F\to M$; $\tr_{F}\colon\Gamma\left(F^{*}\otimes_{M}F\right)\to C^{\infty}\left(M,\mathbb{R}\right),$}\tabularnewline
\multicolumn{2}{|c|}{} &  & {\small $\left[\tr_{F}\left(\sigma\otimes_{M}f\right)\right]\left(m\right):=\sigma\left(m\right)\cdot_{F_{m}}f\left(m\right)$
for $m\in M$.}\tabularnewline
\multicolumn{2}{|c|}{{\small $\sigma\cdot f$}} & {\small $\sigma\cdot_{F}f$} & {\small Natural pairing; $\cdot_{F}\colon\Gamma\left(F^{*}\right)\times\Gamma\left(F\right)\to C^{\infty}\left(M,\mathbb{R}\right)$,}\tabularnewline
\multicolumn{2}{|c|}{} &  & {\small $\left(\sigma\cdot_{F}f\right)\left(m\right):=\sigma\left(m\right)\cdot_{F_{m}}f\left(m\right)$
for $m\in M$.}\tabularnewline
\multicolumn{2}{|c|}{{\small $A\cdot f$}} & {\small $A\cdot_{F}f$} & {\small Natural pairing; $\cdot_{F}\colon\Gamma\left(E\otimes_{M}F^{*}\right)\times F\to E,$}\tabularnewline
\multicolumn{2}{|c|}{} &  & {\small $\left(e\otimes_{M}\sigma\right)\cdot_{F}f:=e\left(m\right)\,\left(\sigma\left(m\right)\cdot_{F_{m}}f\right)\in E_{m}$;
$m:=\pi_{M}^{F}\left(f\right)$.}\tabularnewline
\multicolumn{2}{|c|}{{\small $S\cdot T$}} & {\small $S\cdot_{F}T$} & {\small Tensor field contraction; pointwise tensor contraction;}\tabularnewline
\multicolumn{2}{|c|}{} &  & {\small $\cdot_{F}\colon\Gamma\left(E\otimes_{M}F^{*}\right)\times\Gamma\left(F\otimes_{M}G\right)\to\Gamma\left(E\otimes_{M}G\right),$}\tabularnewline
\multicolumn{2}{|c|}{} &  & {\small $\left[\left(e\otimes\sigma\right)\cdot_{F}\left(f\otimes g\right)\right]\left(m\right):=\left(\sigma\left(m\right)\cdot_{F_{m}}f\left(m\right)\right)\, e\left(m\right)\otimes g\left(m\right)$.}\tabularnewline
\multicolumn{2}{|c|}{{\small $S\cdot^{n}T$}} & {\small $S\cdot_{F_{1}\otimes_{M}\dots\otimes_{M}F_{n}}T$} & {\small Alternate for $\cdot_{F}$, where $F=F_{1}\otimes_{M}\dots\otimes_{M}F_{n}$.}\tabularnewline
\multicolumn{2}{|c|}{{\small $S:T$, $S\cdot^{2}T$}} & {\small $S\cdot_{F_{1}\otimes_{M}F_{2}}T$} & {\small Special notation for $n=2$.}\tabularnewline
\hline 
\end{tabular*}{\small \smallskip{}
}%
\begin{tabular*}{1\textwidth}{@{\extracolsep{\fill}}|c|c|c|c|}
\hline 
{\small High-} & {\small Mid-} & {\small Low-level} & {\small Description}\tabularnewline
\hline 
\hline 
\multicolumn{4}{|c|}{{\small Permutations of tensors and tensor fields.}}\tabularnewline
\hline 
\multicolumn{2}{|c|}{{\small $A^{\sigma}$, $A\cdot^{n}\sigma$}} & {\small $A\cdot_{V_{1}^{*}\otimes\dots\otimes V_{n}^{*}}\sigma$} & {\small Right-action of permutations on $n$-tensors/$n$-tensor fields;}\tabularnewline
\multicolumn{2}{|c|}{} &  & {\small{} $\left(v_{1}\otimes\dots\otimes v_{n}\right)^{\sigma}:=v_{\sigma^{-1}\left(1\right)}\otimes\dots\otimes v_{\sigma^{-1}\left(n\right)}$;
$\left(A^{\sigma}\right)^{\tau}=A^{\sigma\tau}$.}\tabularnewline
\hline 
\multicolumn{4}{|c|}{{\small Spaces of sections of bundles.}}\tabularnewline
\hline 
\multicolumn{3}{|c|}{{\small $\Gamma\left(H\right)$, $\Gamma\left(\pi_{N}^{H}\right)$}} & {\small Space of smooth sections of the bundle $\pi_{N}^{H}$;}\tabularnewline
\hline 
\multicolumn{3}{|c|}{} & {\small $\Gamma\left(H\right):=\left\{ h\in C^{\infty}\left(N,H\right)\mid\pi_{N}^{H}\circ h=\Id_{N}\right\} $.}\tabularnewline
\hline 
\multicolumn{3}{|c|}{{\small $\Gamma_{\phi}\left(H\right)$, $\Gamma_{\phi}\left(\pi_{N}^{H}\right)$}} & {\small Space of smooth sections of $\pi_{N}^{H}$ along $\phi$;}\tabularnewline
\hline 
\multicolumn{3}{|c|}{} & {\small $\Gamma_{\phi}\left(H\right):=\left\{ h\in C^{\infty}\left(M,H\right)\mid\pi_{N}^{H}\circ h=\phi\right\} $.}\tabularnewline
\hline 
\hline 
\multicolumn{4}{|c|}{{\small Vertical bundle, pullback bundle, projection maps, pullback
of sections.}}\tabularnewline
\hline 
\multicolumn{3}{|c|}{{\small $VE\to E$}} & {\small Vertical bundle over $E\to M$; $VE:=\ker T\pi_{M}^{E}\leq TE$.}\tabularnewline
\multicolumn{3}{|c|}{} & {\small projection map $\pi_{E}^{VE}:=\pi_{E}^{TE}\mid_{VE}$.}\tabularnewline
\hline 
\multicolumn{3}{|c|}{{\small $\phi^{*}H\to M$}} & {\small Pullback bundle; $\phi^{*}H:=\left\{ \left(m,h\right)\in M\times H\mid\phi\left(m\right)=\pi\left(h\right)\right\} $.}\tabularnewline
\multicolumn{3}{|c|}{} & {\small $\pi_{M}^{\phi^{*}H}\left(m,h\right):=m$; $\rho_{H}^{\phi^{*}H}\left(m,h\right):=h$.}\tabularnewline
\multicolumn{3}{|c|}{{\small $\phi^{*}h$}} & {\small Pullback of section $h\in\Gamma\left(H\right)$; $\phi^{*}h\in\Gamma\left(\phi^{*}H\right)$ }\tabularnewline
\multicolumn{3}{|c|}{} & {\small defined by $\rho_{H}^{\phi^{*}H}\circ\phi^{*}h=h$; $h\in\Gamma\left(H\right)$.}\tabularnewline
\hline 
\hline 
\multicolumn{4}{|c|}{{\small Covariant derivatives; partial covariant derivatives.}}\tabularnewline
\hline 
{\small $\nabla L$} & {\small $\lnabla L$} & {\small $\lnabla^{M\to\mathbb{R}}L$} & {\small Natural linear covariant derivative; differential of functions;}\tabularnewline
 & {\small $\nabla^{M\to\mathbb{R}}L$} &  & {\small $\lnabla^{M\to\mathbb{R}}L:=dL\in\Gamma\left(T^{*}M\right)$,
where $L\in C^{\infty}\left(M,\mathbb{R}\right)$.}\tabularnewline
\hline 
{\small $\nabla X$} & {\small $\lnabla X$} & {\small $\lnabla^{E}X$} & {\small Linear covariant derivative on vector bundle $E\to M$;}\tabularnewline
 & {\small $\nabla^{E}X$} &  & {\small $\nabla^{E}X\in\Gamma\left(E\otimes_{M}T^{*}M\right)$, where
$X\in\Gamma\left(E\right)$.}\tabularnewline
\hline 
{\small $\nabla\phi$} & {\small $\nnabla\phi$} & {\small $\nnabla^{M\to N}\phi$} & {\small Tangent map as tensor field;}\tabularnewline
 & {\small $\nabla^{M\to N}\phi$} &  & {\small $\nnabla^{M\to N}\phi\in\Gamma\left(\phi^{*}TN\otimes_{M}T^{*}M\right)$,
where $\phi\in C^{\infty}\left(M,N\right)$.}\tabularnewline
\hline 
{\small $\nabla\sigma$} & {\small $\nabla^{\phi}\sigma$} & {\small $\nabla^{\phi^{*}H}\sigma$} & {\small Pullback covariant derivative; $\sigma\in\Gamma\left(\phi^{*}H\right)$;}\tabularnewline
 &  &  & {\small defined by $\nabla^{\phi^{*}H}\phi^{*}h=\phi^{*}\nabla^{H}h\cdot_{\phi^{*}TN}\nnabla^{M\to N}\phi$;
$h\in\Gamma\left(H\right)$.}\tabularnewline
\hline 
\multicolumn{3}{|c|}{{\small $L_{,c_{1}},\dots,L_{,c_{n}}$}} & {\small Partial differential of functions;}\tabularnewline
\multicolumn{3}{|c|}{} & {\small $L_{,c_{i}}\in\Gamma\left(F_{i}^{*}\right)$, defined by $\lnabla^{M\to\mathbb{R}}L=\sum_{i=1}^{n}L_{,c_{i}}\cdot_{F_{i}}c_{i}$.}\tabularnewline
\multicolumn{3}{|c|}{{\small $X_{,c_{1}},\dots,X_{,c_{n}}$}} & {\small Partial linear covariant derivative;}\tabularnewline
\multicolumn{3}{|c|}{} & {\small $X_{,c_{i}}\in\Gamma\left(E\otimes_{M}F_{i}^{*}\right)$,
defined by $\lnabla^{E}X=\sum_{i=1}^{n}X_{,c_{i}}\cdot_{F_{i}}c_{i}$.}\tabularnewline
\multicolumn{3}{|c|}{{\small $\phi_{,M_{1}},\dots,\phi_{,M_{n}}$}} & {\small Partial derivative decomposition of tangent map;}\tabularnewline
\multicolumn{3}{|c|}{} & {\small $\phi_{,M_{i}}\in\Gamma\left(\phi^{*}TN\otimes_{M}\pr_{i}^{*}T^{*}M_{i}\right)$, }\tabularnewline
\multicolumn{3}{|c|}{} & {\small where $M=M_{1}\times\dots\times M_{n}$, $\pr_{i}:=\pr_{i}^{M}$,
and}\tabularnewline
\multicolumn{3}{|c|}{} & {\small $\nnabla^{M\to N}\phi=\sum_{i=1}^{n}\phi_{,M_{i}}\cdot_{\pr_{i}^{*}TM_{i}}\nnabla^{M\to M_{i}}\pr_{i}$.}\tabularnewline
\hline 
\hline 
\multicolumn{4}{|c|}{{\small Covariant Hessians.}}\tabularnewline
\hline 
{\small $\nabla^{2}L$} & \multicolumn{2}{c|}{{\small $\nabla^{T^{*}M}\nabla^{M\to\mathbb{R}}L$}} & {\small Covariant Hessian of functions;}\tabularnewline
{\small $\lnabla\lnabla L$} & \multicolumn{2}{c|}{{\small $\lnabla^{T^{*}M}\lnabla^{M\to\mathbb{R}}L$}} & {\small $\nabla^{2}L\in\Gamma\left(T^{*}M\otimes T^{*}M\right)$;
$L\in C^{\infty}\left(M,\mathbb{R}\right)$.}\tabularnewline
{\small $\nabla^{2}X$} & \multicolumn{2}{c|}{{\small $\nabla^{E\otimes T^{*}M}\nabla^{E}X$}} & {\small Covariant Hessian on vector bundle $E\to M$;}\tabularnewline
{\small $\lnabla\lnabla X$} & \multicolumn{2}{c|}{{\small $\lnabla^{E\otimes_{M}T^{*}M}\lnabla^{E}X$}} & {\small{} $\nabla^{2}X\in\Gamma\left(E\otimes T^{*}M\otimes T^{*}M\right)$;
$X\in\Gamma\left(E\right)$.}\tabularnewline
{\small $\nabla^{2}\phi$} & \multicolumn{2}{c|}{{\small $\nabla^{\phi^{*}TN\otimes T^{*}M}\nabla^{M\to N}\phi$}} & {\small Covariant Hessian of maps;}\tabularnewline
{\small $\lnabla\nnabla\phi$} & \multicolumn{2}{c|}{{\small $\lnabla^{\phi^{*}TN\otimes_{M}T^{*}M}\nnabla^{M\to N}\phi$}} & {\small $\nabla^{2}\phi\in\Gamma\left(\phi^{*}TN\otimes T^{*}M\otimes T^{*}M\right)$.
$\phi\in C^{\infty}\left(M,N\right)$.}\tabularnewline
\hline 
\hline 
\multicolumn{4}{|c|}{{\small Derivative conventions.}}\tabularnewline
\hline 
\multicolumn{3}{|c|}{{\small $\nabla_{X}e$}} & {\small Directional derivative notation; $\nabla_{X}e:=\nabla e\cdot_{TM}X$.}\tabularnewline
\multicolumn{3}{|c|}{{\small $\nabla^{n}e\cdot^{n}\left(X_{1}\otimes\dots\otimes X_{n-1}\otimes X_{n}\right)$}} & {\small Iterated covariant derivative convention;}\tabularnewline
\multicolumn{3}{|c|}{} & {\small defined by $\left(\nabla_{X_{n}}\nabla^{n-1}e\right)\cdot^{n-1}\left(X_{1}\otimes\dots\otimes X_{n-1}\right)$.}\tabularnewline
\multicolumn{3}{|c|}{{\small $R\left(X,Y\right):=-\nabla_{X}\nabla_{Y}+\nabla_{Y}\nabla_{X}+\nabla_{\left[X,Y\right]}$}} & {\small Curvature operator; $R\left(X,Y\right)e=\nabla^{2}e:\left(X\otimes Y-Y\otimes X\right)$.}\tabularnewline
\hline 
\multicolumn{3}{|c|}{{\small $z\colon M\to M\times I,\, m\mapsto\left(m,0\right)$}} & {\small Evaluation-at-zero map.}\tabularnewline
\multicolumn{3}{|c|}{{\small $z^{*}\partial_{i}=\delta_{i}$}} & {\small Pullback formulation of derivative-at-zero.}\tabularnewline
\hline 
\end{tabular*}
\par\end{center}{\small \par}

For more on relevant introductory theory on manifolds, bundles and
Riemannian geometry, see \citep{IntroLee}, \citep{JeffMLee}, \citep{Michor},
\citep{RiemannianLee}.

\part{Mathematical Setting \label{part:Mathematical-Setting}}

\section{Using Strong Typing to Error-Check Calculations}

Linear algebra is an excellent setting for discussion of the \textbf{strong
typing} \citep{Cardelli} of a language, a concept used in the design
of computer programming languages. The idea is that when the human-readable
source code of a program is compiled (translated into machine-readable
instructions), the compiler (the program which performs this translation)
or runtime (the software which executes the code) verifies that the
program objects are being used in a well-defined way, producing an
error for each operation that is not well-defined. For example, a
vector-type value would not be allowed to be added to a permutation-type
value, even though tuples of unsigned integers (i.e. bytes) are used
by the computer to represent both, and the computer's processing unit
could add together their byte-valued representations. However, such
an operation would be meaningless with respect to the types of the
operands. The result of the operation would depend on the non-canonical
choice of representation for each object. Strong type checking has
the advantage of catching many programming errors, including most
importantly those resulting from an inherent misuse of the program's
objects. Within this paper, certain type-explicit notations will be
used to provide forms of type awareness conducive to error-checking.\\

An important example of semi-strong typing in math is Penrose's abstract
index notation \citep{Penrose}, modeled on Einstein's summation convention,
in which linear algebra and tensor calculus are implemented using
indexed objects (tensors) having a certain number and order of ``up''
and ``down'' indices (an abstraction of the genuine basis/coordinate
expressions in which the indexed objects are arrays of scalars/functions).
A non-indexed tensor is a scalar value, a tensor having a single up
or down index is a vector or covector value respectively, a tensor
having an up and a down index is an endomorphism, and so forth. The
tensors are contracted by pairing a certain number of up indices with
the same number of down indices, resulting in an object having as
indices the uncontracted indices.

For example, given a finite-dimensional inner product space $\left(V,g\right)$,
where $g$ is a ${0 \choose 2}$-tensor (having the form $g_{ij}$,
i.e. two down indices), a vector $v\in V$ is a ${1 \choose 0}$-tensor,
and the length of $v$ is $\sqrt{v^{i}g_{ij}v^{j}}$. If $\dim V>1$,
then $\bigwedge^{2}V$ has positive dimension, its vectors each being
${2 \choose 0}$-tensors, and $G_{ijk\ell}:=g_{ik}g_{j\ell}-g_{i\ell}g_{jk}$
is an inner product on $\bigwedge^{2}V$ (which must be a ${0 \choose 4}$-tensor
in order to contract with two ${2 \choose 0}$-tensors).

Certain type errors are detected by use of abstract index notation
in the form of index mismatch. For example, with $\left(V,g\right)$
as above, if $\alpha\in V^{*}$, then $\alpha$ is a ${0 \choose 1}$-tensor.
Because of the repeated $j$ down indices, the expression $g_{ij}\alpha_{j}$
typically indicates a type error; $g_{ij}$ can't contract with $\alpha_{j}$
because of incompatible valence (valence being the number of up and
down indices). Furthermore, multiplying a ${0 \choose 2}$-tensor
with a ${0 \choose 1}$-tensor without contraction should result in
a ${0 \choose 3}$-tensor, which should be denoted using three indices,
as in $g_{ij}\alpha_{k}$.

The only explicit type information provided by abstract index notation
is that of valence. The ``semi'' qualifier mentioned earlier is
earned by the lack of distinction between the different spaces in
which the tensors reside. For example, if $U,V,W$ are finite-dimensional
vector spaces, then linear maps $A\colon U\to V$ and $B\colon V\to W$
can be written as ${1 \choose 1}$-tensors, and their composition
$B\circ A\colon U\to W$ is written as the tensor contraction $\left(B\circ A\right)_{j}^{i}=B_{k}^{i}A_{j}^{k}$.
However, while the expression $A_{k}^{i}B_{j}^{k}$ makes sense in
terms of valence compatibility (i.e. grammatically), the composition
``$A\circ B$'' that it should represent is not well-defined. Thus
this form of type error is not caught by abstract index notation,
since the domains/codomains of the linear maps must be checked separately.\\

The use of dimensional analysis (the abstract use of units such as
kilograms, seconds, etc) in Physics is an important precedent of strong
typing. Each quantity has an associated ``dimension'' (this is a
different meaning from the ``dimension'' of linear algebra) which
is expressed as a fraction of powers of formal symbols. The ordinary
algebraic rules for fractions and formal symbols are used for the
dimensional components, with the further requirement that addition
and equality may only occur between quantities having the same dimension.

For example, if $\mbox{E}$, $\mbox{M}$ and $\mbox{C}$ represent
the dimensions of energy, mass and cost, respectively, and if the
energy storage density $\rho\,\mbox{E}/\mbox{M}$ of a battery manufacturing
process is known (having dimensions energy per mass) and the manufacturing
weight yield $w\,\mbox{M}/\mbox{C}$ of the battery is known (having
dimensions mass per cost), then under the algebraic convensions of
dimensional analysis, calculating the energy storage per cost (which
should have dimensions energy per cost) is simple; 
\[
\left(\rho\,\frac{\mbox{E}}{\mbox{M}}\right)\left(w\,\frac{\mbox{M}}{\mbox{C}}\right)=\rho w\,\frac{\mbox{EM}}{\mbox{MC}}=\rho w\,\frac{\mbox{E}}{\mbox{C}}
\]
(the $\mbox{M}$ symbols cancel in the fraction). Here, both $\rho$
and $w$ are real numbers, and besides using the well-definedness
of real multiplication, no type-checking is done in the expression
$\rho w$. 

A contrasting example is the quantity $\rho/w$, having dimensions
$\mbox{EC}/\mbox{M}^{2}$. However, these dimensions may be considered
to be meaningless in the given context. The quantity's type adds meaning
to the real-valued quantity, and while the quantity is well-defined
as a real number, the uselessness of the type may indicate that an
error has been made in the calculations. For example, a type mismatch
between the two sides of an equation is a strong indication of error.

This is also a convenient way to think about the chain rule of calculus.
If $z\left(y\right)$, $y\left(x\right)$, and $x$ measure real-valued
quantities, then $z\left(y\left(x\right)\right)$ measures the quantity
$z$ with respect to quantity $x$. Using $\mbox{Z}$, $\mbox{Y}$
and $\mbox{X}$ for the dimensions of the quantities $z$, $y$ and
$x$ respectively, the derivative $\frac{dz}{dx}$ has units $\mbox{Z}/\mbox{X}$.
When worked out, the dimensions for the quantities on either side
of the equation $\frac{dz}{dx}=\frac{dz}{dy}\frac{dy}{dx}$ will match
exactly, having a non-coincidental similarity to the calculation in
the battery product example.

\section{Telescoping Notation (aka Don't Fear the Verbosity) \label{sec:Telescoping-Notation}}

Many of the computations developed in this paper will appear to be
overly pedantic, owing to the decoration-heavy notation that will
be introduced in \prettyref{sec:Strongly-Typed-Linear-Algebra}. This
decoration is largely for the purpose of tracking the myriad of types
in the type system and to assist the human reader or writer in making
sense of  and error-checking the expressions involved. The pedantry
in this paper plays the role of introducing the technique. The notation
is designed to \textbf{telescope}%
\footnote{Credit for the notion of telescoping notation is due in part to David
DeConde, during one of many enjoyable and insightful conversations.%
}, meaning that there is a spectrum of notational decoration; from
\begin{itemize}
\item pedantically type-specified, verbose, and decoration-heavy, where
{[}almost{]} no types must be inferred from context and there is little
work or expertise required on the part of the reader, to
\item somewhat decorated but more compact, where the reader must do a little
bit of thinking to infer some types, all the way to
\item tersely notated with minimal type decoration, where {[}almost{]} all
types must be inferred from context and the reader must either do
a lot of thinking or be relatively experienced.
\end{itemize}
Additionally, some of the chosen symbols are meant to obey the same
telescoping range of specifity. For example, compare $n$-fold tensor
contraction $\cdot^{n}$ with type-specified $\cdot_{V_{1}\otimes\dots\otimes V_{n}}$
as discussed in \prettyref{sec:Strongly-Typed-Linear-Algebra}, or
the symbols $\nabla$, $\nnabla$, and $\lnabla$ as discussed in
\prettyref{sec:Curvature-and-Commutation-of-Derivatives}. Tersely
notated computations can be seen in \prettyref{sec:Curvature-and-Commutation-of-Derivatives},
while fully-verbose computations abound in the careful exposition
of Part \prettyref{part:Riemannian-Calculus-of-Variations}.

\section{Strongly-Typed Linear Algebra via Tensor Products\label{sec:Strongly-Typed-Linear-Algebra}}

A fully strongly typed formulation of linear algebra will now be developed
which enjoys a level of abstraction and flexibility similar to that
of Penrose's abstract index notation. Emphasis will be placed on notational
and conceptual regularity via a tensor formalism, coupled with a notion
of ``untangled'' expression which exploits and notationally depicts
the associativity of linear composition.\\

If $V$ denotes a finite-dimensional vector space, then let
\[
\cdot_{V}\colon V^{*}\times V\to\mathbb{R},\,\left(\alpha,v\right)\mapsto\alpha\left(v\right)
\]
denote the \textbf{natural pairing} on $V$, and denote $\cdot_{V}\left(\alpha,v\right)$
using the infix notation $\alpha\cdot_{V}v$. The natural pairing
is a nondegenerate bilinear form and its bilinearity gives the expression
$\alpha\cdot_{V}v$ multiplicative semantics (distributivity and commutativity
with scalar multiplication), thereby justifying the use of the infix
$\cdot$ operator normally reserved for multiplication. The natural
pairing subscript $V$ is seemingly pedantic, but will prove to be
an invaluable tool for articulating and navigating the rich type system
of the linear algebraic and vector bundle constructions used in this
paper. When clear from context, the subscript $V$ may be omitted.

Because $V$ is finite-dimensional, it is reflexive (i.e. the canonical
injection $V\to V^{**},\, v\mapsto\left(\alpha\mapsto\alpha\left(v\right)\right)$
is a linear isomorphism). Thus the natural pairing $\cdot_{V^{*}}$
on $V^{*}$ can be written naturally as
\[
\cdot_{V^{*}}\colon V\times V^{*}\to\mathbb{R},\,\left(v,\alpha\right)\mapsto\alpha\left(v\right).
\]
Note that $\alpha\cdot_{V}v=v\cdot_{V^{*}}\alpha$. Though subtle,
the distinction between $\cdot_{V}$ and $\cdot_{V^{*}}$ is important
within the type system used in this paper.

Through a universal mapping property of multilinear maps, the bilinear
forms $\cdot_{V}$ and $\cdot_{V^{*}}$ descend to the \textbf{natural
trace} maps
\begin{align*}
\tr_{V}\colon V^{*}\otimes V & \to\mathbb{R},\,\alpha\otimes v\mapsto\alpha\left(v\right),\mbox{ and}\\
\tr_{V^{*}}\colon V\otimes V^{*} & \to\mathbb{R},\, v\otimes\alpha\mapsto\alpha\left(v\right),
\end{align*}
each extended linearly to non-simple tensors. These operations can
also be called \textbf{tensor contraction}. Noting that $\left(V^{*}\otimes V\right)^{*}$
and $\left(V\otimes V^{*}\right)^{*}$ are canonically isomorphic
to $V\otimes V^{*}$ and $V^{*}\otimes V$ respectively, then for
each $A\in V^{*}\otimes V$ and $B\in V\otimes V^{*}$, it follows
that $\tr_{V}\left(A\right)=\Id_{V^{*}}\cdot_{V^{*}\otimes V}A$ and
$\tr_{V^{*}}\left(B\right)=\Id_{V}\cdot_{V\otimes V^{*}}B$.
\begin{defn}[Linear maps as tensors]
 \label{def:linear_maps_as_tensors} Let $V$ and $W$ be finite-dimensional
vector spaces, and let $\Hom\left(V,W\right)$ denote the space of
vector space morphisms from $V$ to $W$ (i.e. linear maps). The linear
isomorphism 
\begin{eqnarray*}
W\otimes V^{*} & \to & \Hom\left(V,W\right),\\
w\otimes\alpha & \mapsto & \left(V\to W,\, v\mapsto w\left(\alpha\cdot_{V}v\right)\right)
\end{eqnarray*}
(extended linearly to general tensors) will play a central conceptual
role in the calculations employed in this paper, as it will facilitate
constructions which would otherwise be awkward or difficult to express.
Linear maps and appropriately typed tensor products will be identified
via this isomorphism.
\end{defn}
Given bases $v_{1},\dots,v_{m}\in V$ and $w_{1},\dots,w_{n}\in W$,
and dual bases $v^{1},\dots,v^{m}\in V^{*}$ and $w^{1},\dots,w^{n}\in W^{*}$,
a linear map $A\colon V\to W$ can be written under the identification
in (\ref{def:linear_maps_as_tensors}) as
\[
A=A_{j}^{i}\, w_{i}\otimes v^{j},
\]
where $A_{j}^{i}=w^{i}\cdot_{W}A\cdot_{V}v_{j}\in\mathbb{R}$, and
in fact $\left[A_{j}^{i}\right]\in M_{n\times m}\left(\mathbb{R}\right)$
is the matrix representation of $A$ with respect to the bases $v_{1},\dots,v_{m}\in V$
and $w_{1},\dots,w_{n}\in W$, noting that the $i$ and $j$ indices
denote the ``output'' and ``input'' components of $A$ respectively.
Tensors are therefore the strongly typed analog of matrices, where
the $W\otimes V^{*}$ type information is carried by the $w_{i}\otimes v^{j}$
component.

One clarifying example of the tensor formulation is the adjoint operation
of the natural pairing, also known as forming the dual of a linear
map. It is straightforward to show that 
\begin{eqnarray*}
*\colon W\otimes V^{*} & \to & V^{*}\otimes W,\\
w\otimes\alpha & \mapsto & \alpha\otimes w,
\end{eqnarray*}
(where the map is extended linearly to general tensors). This is literally
the tensor abstraction of the matrix transpose operation; if $A=A_{j}^{i}\, w_{i}\otimes\alpha^{j}$,
then the dual $A$ is $A^{*}=A_{i}^{j}\,\alpha^{i}\otimes w_{j}$.
The matrix of $A^{*}$ is precisely the transpose of the matrix of
$A$ with respect to the relevant bases. The map $*$ itself can be
written as a 4-tensor $*\in V^{*}\otimes W\otimes W^{*}\otimes V$,
where $A^{*}=*\cdot_{W\otimes V^{*}}A$. 

There is a notion of the natural pairing of tensor products, which
implements composition and evaluation of linear maps, and can be thought
of as a natural generalization of scalar multiplication in a field.
If $U,$ $V,$ and $W$ are each finite-dimensional vector spaces,
then the bilinear form
\begin{eqnarray*}
\left(U\otimes V^{*}\right)\times\left(V\otimes W\right) & \to & U\otimes\mathbb{R}\otimes W\cong U\otimes W,\\
\left(u\otimes\alpha,v\otimes w\right) & \mapsto & u\otimes\left(\alpha\cdot_{V}v\right)\otimes w=\left(\alpha\cdot_{V}v\right)u\otimes w
\end{eqnarray*}
will be denoted also by the infix notation $\cdot_{V}$ (i.e. $\left(u\otimes\alpha\right)\cdot_{V}\left(v\otimes w\right)=\left(\alpha\cdot_{V}v\right)u\otimes w$).
If $V$ itself is a tensor product of $n$ factors which are clear
from context, then $\cdot_{V}$ may be denoted by $\cdot^{n}$ (think
an $n$-fold tensor contraction). If $n=2$, then typically $:$ is
used in place of $\cdot^{2}$. For example, from above, $A^{*}=*\cdot_{W\otimes V^{*}}A=*:A$.\\

Given a permutation $\sigma\in S_{n}$, define a right-action by $\sigma\colon V_{1}\otimes\dots\otimes V_{n}\to V_{\sigma^{-1}\left(1\right)}\otimes\dots\otimes V_{\sigma^{-1}\left(n\right)}$,
mapping elements in the obvious way. For example, $\left(2\,3\,4\right)$
acting on $v_{1}\otimes v_{2}\otimes v_{3}\otimes v_{4}$ puts the
second factor in the third position, the third factor in the fourth
position, and the fourth factor in the second, giving $v_{1}\otimes v_{4}\otimes v_{2}\otimes v_{3}$.
This permutation is itself a linear map and of course can be written
as a tensor. However, because it is defined in terms of a right action,
the ``domain factors'' will come on the left. Thus $\sigma$ is
written as a tensor of the form $V_{1}^{*}\otimes\dots\otimes V_{n}^{*}\otimes V_{\sigma^{-1}\left(1\right)}\otimes\dots\otimes V_{\sigma^{-1}\left(n\right)}$
(i.e. as a $2n$-tensor). Certain tensor constructions are conducive
to using such permutations. In the above example, $*$ can be written
as $\left(1\,2\right)\in W^{*}\otimes V\otimes V^{*}\otimes W$. 

The permutation right-action also works naturally when notated using
superscripts. For example, if $B\in U\otimes V\otimes W$, then 
\[
B^{\left(1\,2\right)}:=B\cdot_{U^{*}\otimes V^{*}\otimes W^{*}}\left(1\,2\right)\in V\otimes U\otimes W
\]
and so
\begin{align*}
\left(B^{\left(1\,2\right)}\right)^{\left(2\,3\right)} & =\left(B\cdot_{U^{*}\otimes V^{*}\otimes W^{*}}\left(1\,2\right)\right)\cdot_{V^{*}\otimes U^{*}\otimes W^{*}}\left(2\,3\right)\\
 & =B\cdot_{U^{*}\otimes V^{*}\otimes W^{*}}\left(\left(1\,2\right)\cdot_{V^{*}\otimes U^{*}\otimes W^{*}}\left(2\,3\right)\right)\\
 & =B\cdot_{U^{*}\otimes V^{*}\otimes W^{*}}\left(1\,2\right)\left(2\,3\right)\\
 & =B\cdot_{U^{*}\otimes V^{*}\otimes W^{*}}\left(1\,3\,2\right)\in V\otimes W\otimes U.
\end{align*}
When multiplying the permutations $\left(1\,2\right)$ and $\left(2\,3\right)$
in the third line, it is important to note that they are read left-to-right,
since they are acting on $B$ on the right. 

The inline cycle notation is somewhat ambiguous in isolation because
the number of factors in the domain/codomain is not specified, let
alone their types. This information can sometimes be inferred from
context, such as from the natural pairing subscripts, as in the following
examples. 
\begin{example}[Linearizing the inversion map]
Let $i\colon GL\left(V\right)\to GL\left(V\right),\, A\mapsto A^{-1}$,
i.e. the linear map inversion operator, where $GL\left(V\right)$
is an open submanifold of $V\otimes V^{*}$ via the isomorphism $V\otimes V^{*}\cong\Hom\left(V,V\right)$.
Its linearization (derivative) $Di\colon GL\left(V\right)\to V\otimes V^{*}\otimes\left(V\otimes V^{*}\right)^{*}\cong V\otimes V^{*}\otimes V^{*}\otimes V$
at $A\in GL\left(V\right)$ in the direction $B\in T_{A}\left(GL\left(V\right)\right)\cong V\otimes V^{*}$
is 
\begin{align*}
Di\left(A\right)\cdot_{V\otimes V^{*}}B={} & Di\cdot_{V\otimes V^{*}}\delta\left(A+\epsilon B\right)\\
={} & \delta\left(i\left(A+\epsilon B\right)\right)\\
={} & \delta\left(\left(A+\epsilon B\right)^{-1}\right)\\
={} & \delta\left(\left(\left(1+\epsilon BA^{-1}\right)A\right)^{-1}\right)\\
={} & \delta\left(A^{-1}\left(1+\epsilon BA^{-1}\right)^{-1}\right)\\
={} & \delta\left(A^{-1}\sum_{n=0}^{\infty}\left(-\epsilon BA^{-1}\right)^{n}\right)\\
 & \mbox{(\ensuremath{\left|-\epsilon BA^{-1}\right|}is taken arbitrarily small due to the derivative \ensuremath{\delta:=\frac{d}{d\epsilon}\mid_{\epsilon=0}}}\\
 & \mbox{being evaluated in an arbitrarily small neighborhood of \ensuremath{\epsilon=0})}\\
={} & \delta\left(A^{-1}-\epsilon A^{-1}BA^{-1}+O\left(\epsilon^{2}\right)\right)\\
={} & -A^{-1}\cdot_{V}B\cdot_{V}A^{-1}.
\end{align*}
In order to ``move'' the $B$ parameter out so that it plays the
same syntactical role as in the original expression $Di\left(A\right)\cdot B$,
via adjacent natural pairing, some simple tensor manipulations can
be done. The process is easily and accurately expressed via diagram.
The following sequence of diagrams is a sequence of equalities. The
diagram should be self-explanatory, but for reference, the number
of boxes for a particular label denotes the rank of the tensor, with
each box labeled with its type. The lines connecting various boxes
are natural pairings, and the circles represent the unpaired ``slots'',
which comprise the type of the resulting expression.

\tikzstyle{labelnode}=[draw, on chain, minimum width=0.8cm, minimum height=0.8cm]
\tikzstyle{opennode}=[circle, on chain, draw=black, fill=white]
\begin{center}
\begin{tikzpicture}[start chain, node distance=-0.15mm, every join/.style=->]
	\node[opennode]  (out)  {};
	\node[on chain]  ()     {};
	\node[on chain]  ()     {};
	\node[opennode]  (in)   {};

	\node[on chain]  ()     {};
	\node[on chain]  ()     {};
	\node[labelnode] (box1) {$V$};
	\node[labelnode] (box2) {$V^*$};
	\node[labelnode] (box3) {$V^*$};
	\node[labelnode] (box4) {$V$};
	\node[above] at ($0.5*($(box2)+(box3)$)+(0,0.5)$) {$Di(A)$};

	\node[on chain]  ()     {};
	\node[on chain]  ()     {};
	\node[labelnode] (box5) {$V$};
	\node[labelnode] (box6) {$V^*$};
	\node[above] at ($0.5*($(box5)+(box6)$)+(0,0.5)$) {$B$};

	\path[draw] ($(out)-(0,0.17)$) edge [bend right] ($(box1)-(0,0.4)$);
	\path[draw] ($(in)-(0,0.17)$) edge [bend right] ($(box2)-(0,0.4)$);
	\path[draw] ($(box3)-(0,0.4)$) edge [bend right] ($(box5)-(0,0.4)$);
	\path[draw] ($(box4)-(0,0.4)$) edge [bend right] ($(box6)-(0,0.4)$);
\end{tikzpicture}
\end{center}

\begin{center}
\begin{tikzpicture}[start chain, node distance=-0.15mm, every join/.style=->]
	\node[opennode]  (out)  {};

	\node[on chain]  ()     {};
	\node[on chain]  ()     {};
	\node[labelnode] (box1) {$V$};
	\node[labelnode] (box2) {$V^*$};
	\node[above] at ($0.5*($(box1)+(box2)$)+(0,0.5)$) {$-A^{-1}$};

	\node[on chain]  ()     {};
	\node[on chain]  ()     {};
	\node[labelnode] (box3) {$V$};
	\node[labelnode] (box4) {$V^*$};
	\node[above] at ($0.5*($(box3)+(box4)$)+(0,0.5)$) {$B$};

	\node[on chain]  ()     {};
	\node[on chain]  ()     {};
	\node[labelnode] (box5) {$V$};
	\node[labelnode] (box6) {$V^*$};
	\node[above] at ($0.5*($(box5)+(box6)$)+(0,0.5)$) {$A^{-1}$};

	\node[on chain]  ()     {};
	\node[on chain]  ()     {};
	\node[opennode]  (in)   {};

	\path[draw] (out)--(box1);
	\path[draw] (box2)--(box3);
	\path[draw] (box4)--(box5);
	\path[draw] (box6)--(in);
\end{tikzpicture}
\end{center}

The following step is nothing but moving the boxes for $B$ out; the
natural pairings still apply to the same slots, hence the cables dangling
below.

\begin{center}
\begin{tikzpicture}[start chain, node distance=-0.15mm, every join/.style=->]
	\node[opennode]  (out)  {};

	\node[on chain]  ()     {};
	\node[on chain]  ()     {};
	\node[labelnode] (box1) {$V$};
	\node[labelnode] (box2) {$V^*$};
	\node[above] at ($0.5*($(box1)+(box2)$)+(0,0.5)$) {$-A^{-1}$};

	\node[on chain]  ()     {};
	\node[on chain]  ()     {};
	\node[labelnode] (box5) {$V$};
	\node[labelnode] (box6) {$V^*$};
	\node[above] at ($0.5*($(box5)+(box6)$)+(0,0.5)$) {$A^{-1}$};

	\node[on chain]  ()     {};
	\node[on chain]  ()     {};
	\node[opennode]  (in)   {};

	\node[on chain]  ()     {};
	\node[on chain]  ()     {};
	\node[labelnode] (box3) {$V$};
	\node[labelnode] (box4) {$V^*$};
	\node[above] at ($0.5*($(box3)+(box4)$)+(0,0.5)$) {$B$};

    \path[draw] (out)--(box1);
	\path[draw] ($(box2)-(0,0.4)$) edge [bend right] ($(box3)-(0,0.4)$);
	\path[draw] ($(box5)-(0,0.4)$) edge [bend right] ($(box4)-(0,0.4)$);
	\path[draw] (box6)--(in);
\end{tikzpicture}
\end{center}

In this setting, a tensor product amounts to flippantly gluing boxes
together.

\begin{center}
\begin{tikzpicture}[start chain, node distance=-0.15mm, every join/.style=->]
	\node[opennode]  (out)  {};

	\node[on chain]  ()     {};
	\node[on chain]  ()     {};
	\node[labelnode] (box1) {$V$};
	\node[labelnode] (box2) {$V^*$};
	\node[labelnode] (box5) {$V$};
	\node[labelnode] (box6) {$V^*$};
	\node[above] at ($0.5*($(box2)+(box5)$)+(0,0.5)$) {$-A^{-1} \otimes A^{-1}$};

	\node[on chain]  ()     {};
	\node[on chain]  ()     {};
	\node[opennode]  (in)   {};

	\node[on chain]  ()     {};
	\node[on chain]  ()     {};
	\node[labelnode] (box3) {$V$};
	\node[labelnode] (box4) {$V^*$};
	\node[above] at ($0.5*($(box3)+(box4)$)+(0,0.5)$) {$B$};

    \path[draw] (out)--(box1);
	\path[draw] ($(box2)-(0,0.4)$) edge [bend right] ($(box3)-(0,0.4)$);
	\path[draw] ($(box5)-(0,0.4)$) edge [bend right] ($(box4)-(0,0.4)$);
	\path[draw] (box6)--(in);
\end{tikzpicture}
\end{center}

In order for $B$ to be naturally paired in the same adjacent manner
as in the original expression $Di\left(A\right)\cdot B$, the slots
of $-A^{-1}\otimes A^{-1}$ must be permuted; the second moves to
the third, the third to the fourth, and the fourth to the second.

\begin{center}
\begin{tikzpicture}[start chain, node distance=-0.15mm, every join/.style=->]
	\node[opennode]  (out)  {};
	\node[on chain]  ()     {};
	\node[on chain]  ()     {};
	\node[opennode]  (in)   {};

	\node[on chain]  ()     {};
	\node[on chain]  ()     {};
	\node[labelnode] (box1) {$V$};
	\node[labelnode] (box2) {$V^*$};
	\node[labelnode] (box3) {$V^*$};
	\node[labelnode] (box4) {$V$};
	\node[above] at ($0.5*($(box2)+(box3)$)+(0,0.5)$) {$-\left(A^{-1} \otimes A^{-1}\right)^{(2\,3\,4)}$};

	\node[on chain]  ()     {};
	\node[on chain]  ()     {};
	\node[labelnode] (box5) {$V$};
	\node[labelnode] (box6) {$V^*$};
	\node[above] at ($0.5*($(box5)+(box6)$)+(0,0.5)$) {$B$};

	\path[draw] ($(out)-(0,0.17)$) edge [bend right] ($(box1)-(0,0.4)$);
	\path[draw] ($(in)-(0,0.17)$) edge [bend right] ($(box2)-(0,0.4)$);
	\path[draw] ($(box3)-(0,0.4)$) edge [bend right] ($(box5)-(0,0.4)$);
	\path[draw] ($(box4)-(0,0.4)$) edge [bend right] ($(box6)-(0,0.4)$);
\end{tikzpicture}
\end{center}

The first diagram equals the last one, thus $Di\left(A\right)\cdot_{V\otimes V^{*}}B=-\left(A^{-1}\otimes A^{-1}\right)^{\left(2\,3\,4\right)}\cdot_{V\otimes V^{*}}B$,
and by the nondegeneracy of the natural pairing on $V\otimes V^{*}$,
this implies that $Di\left(A\right)=-\left(A^{-1}\otimes A^{-1}\right)^{\left(2\,3\,4\right)}$,
noting that the statement of this expression does not require the
direction vector $B$. The permutation exponent $\left(2\,3\,4\right)$
can be calculated easily using simple tensors, if not by the above
diagrammatic manipulations;
\[
\left(a_{1}\otimes a_{2}\right)\cdot\left(b_{1}\otimes b_{2}\right)\cdot\left(a_{3}\otimes a_{4}\right)=\left(a_{1}\otimes a_{4}\otimes a_{2}\otimes a_{3}\right):\left(b_{1}\otimes b_{2}\right)=\left(a_{1}\otimes a_{2}\otimes a_{3}\otimes a_{4}\right)^{\left(2\,3\,4\right)}:\left(b_{1}\otimes b_{2}\right).
\]
Here, the expression $\left(a_{1}\otimes a_{2}\right)\cdot\left(b_{1}\otimes b_{2}\right)\cdot\left(a_{3}\otimes a_{4}\right)$
represents the expression $A^{-1}\cdot B\cdot A^{-1}$. 
\end{example}
The next example will later be extended to the setting of Riemannian
manifolds and their metric tensor fields, and put to use to formulate
what are known as harmonic maps (see (\ref{example_harmonic_maps})).
But first, a new tensor operation must be defined.
\begin{defn}[Parallel tensor product]
 \label{def:parallel_tensor_product} If $U,V,W,X$ are vector spaces
and $A\in U\otimes V$ and $B\in W\otimes X$, then define their \textbf{parallel
tensor product} $A\boxtimes B$ by
\[
A\boxtimes B:=\left(A\otimes B\right)^{\left(2\,3\right)}\in\left(U\otimes W\right)\otimes\left(V\otimes X\right).
\]
The parentheses in the type specification are unnecessary, but hint
at what the tensor decomposition for the quantity $A\boxtimes B$
should be, if used as an operand to $\boxtimes$ again (see below).
\end{defn}
If $A$ and $B$ represent linear maps, then $A\boxtimes B\in\left(U\otimes W\right)\otimes\left(V\otimes X\right)$
represents their tensor product as linear maps (the parentheses are
unnecessary but hint at what the domain and codomain are, and for
use of $A\boxtimes B$ as an operand in another parallel tensor product),
which is a ``parallel'' composition; if $\alpha\in V^{*}$ and $\beta\in X^{*}$,
then $\left(A\boxtimes B\right)\cdot_{V^{*}\otimes X^{*}}\left(\alpha\otimes\beta\right)=\left(A\cdot_{V^{*}}\alpha\right)\otimes\left(B\cdot_{X^{*}}\beta\right)$.

There is a slight ambiguity in the notation coming from a lack of
specification on how the tensor product of the operands is decomposed
in the case when there is more than one such decomposition. Notation
explicitly resolving this ambiguity will not be needed in this paper
as the relevant tensor product is usually clear from context.

The parallel tensor product is associative; if $Y$ and $Z$ are also
vector spaces and $C\in Y\otimes Z$, then 
\[
\left(A\boxtimes B\right)\boxtimes C=A\boxtimes\left(B\boxtimes C\right)\in\left(U\otimes W\otimes Y\right)\otimes\left(V\otimes X\otimes Z\right),
\]
allowing multiply-parallel tensor products. 
\begin{example}[Tensor product of inner product spaces]
\label{example_tensor_product_of_inner_product_spaces} If $\left(V,g\right)$
and $\left(W,h\right)$ are inner product spaces (noting that $g\in V^{*}\otimes V^{*}$
and $h\in W^{*}\otimes W^{*}$ are symmetric, i.e. literally invariant
under $\left(1\,2\right)$), then $W\otimes V^{*}$ is an inner product
space having induced inner product $k\left(A,B\right):=\tr_{V}\left(g^{-1}\cdot_{V^{*}}A^{*}\cdot_{W^{*}}h\cdot_{W}B\right)$.
Here, the ``inputs'' of $A$ and $B$ (the $V^{*}$ factors) are
being paired using $g^{-1}\in V\otimes V$, while the ``outputs''
(the $W$ factors) are being paired using $h\in W^{*}\otimes W^{*}$,
and the trace is used to ``complete the cycle'' by plugging the
output into the input, thereby producing a real number. The expression
$k\left(A,B\right)$ can be written in a more natural way, which takes
advantage of the linear composition, as $A:k:B$ (or, pedantically,
$A\cdot_{W^{*}\otimes V}k\cdot_{W\otimes V^{*}}B$)$ $, instead of
the more common but awkward trace expression mentioned earlier. In
the tensor formalism, the inner product $k$ should have type $W^{*}\otimes V\otimes W^{*}\otimes V$.
Permuting the middle two components of the 4-tensor $h\otimes g^{-1}\in W^{*}\otimes W^{*}\otimes V\otimes V$
gives the correct type. In fact, $k=h\boxtimes g^{-1}$. A further
advantage to this formulation is that if any or all of $A,k,B$ are
functions, there is a clear product rule for derivatives of the expression
$A:k:B$. This is something that is used critically in Riemannian
geometry in the form of covariant derivatives of tensor fields (see
(\ref{definition_linear_covariant_derivative})).
\end{example}
In this paper, the main use of the tensor formulation of linear maps
is twofold: to facilitate linear algebraic constructions which would
otherwise be difficult or awkward (this includes the ability to express
derivatives of {[}possibly vector or manifold-valued{]} maps without
needing to ``plug in'' the derivative's directional argument), and
to make clear the product-rule behavior of many important differentiable
constructions.

\section{Bundle Constructions \label{sec:Bundle-Constructions}}

In order to use the calculus of variations involving Lagrangians depending
tangent maps of maps between smooth manifolds, it suffices to consider
Lagrangians defined on smooth vector bundle morphisms. Continuing
in the style of the previous section, a ``full'' tensor product
of smooth vector bundles (\ref{proposition_tensor_product_bundle})
will be formulated which will then allow expression of smooth vector
bundle morphisms as tensor fields, sometimes called two-point tensor
fields \citep[pg. 70]{Marsden&Hughes}. The full arsenal of tensor
calculus can then be used to considerable advantage.\\

First, some definitions and simpler bundle constructions will be introduced.
A \textbf{smooth {[}fiber{]} bundle} (hereafter refered to simply
as a smooth bundle) is a 4-tuple $\left(\mathcal{E},E,\pi,N\right)$
where $\mathcal{E}$, $E$ and $N$ are smooth manifolds and $\pi\colon E\to N$
is locally trivial, i.e. $N$ is covered by open sets $\left\{ U_{\alpha}\right\} $
such that $\pi^{-1}\left(U_{\alpha}\right)\cong U_{\alpha}\times\mathcal{E}$
as smooth manifolds. The manifolds $\mathcal{E}$, $E$ and $N$ are
called the \textbf{typical fiber}, the \textbf{total space}, and the
\textbf{base space} respectively. The map $\pi$ is called the \textbf{bundle
projection}. The full 4-tuple specifying a bundle can be recovered
from the bundle projection map, so a locally trivial smooth map can
be said to define a smooth bundle. The dimension of the typical fiber
of a bundle will be called its \textbf{rank}, and will be denoted
by $\rank\pi$ or $\rank E$ when the bundle is understood from context.

The \textbf{space of smooth sections} of a smooth bundle defined by
$\pi\colon E\to N$ is
\[
\Gamma\left(\pi\right):=\left\{ \sigma\in C^{\infty}\left(N,E\right)\mid\pi\circ\sigma=\Id_{N}\right\} ,
\]
and may also be denoted by $\Gamma\left(E\right)$, if the bundle
is clear from context. If nonempty, $\Gamma\left(\pi\right)$ is generally
an infinite-dimensional manifold (the exception being when the base
space $N$ is finite).
\begin{prop}[Trivial bundle]
 \label{prop:trivial_bundle} Let $M$ and $N$ be smooth manifolds.
With $M\rtriv N:=M\times N$ and
\[
\pi^{M\rtriv N}:=\pr_{2}^{M\times N}\colon M\rtriv N\to N
\]
defines a smooth bundle $\left(M,M\rtriv N,\pi^{M\rtriv N},N\right)$,
called a \textbf{trivial bundle}. Similarly, with $M\ltriv N:=M\times N$
and $\pi^{M\ltriv N}:=\pr_{1}^{M\times N}\colon M\ltriv N\to M$,
$\left(N,M\ltriv N,\pi^{M\ltriv N},M\right)$ is a trivial bundle.
\end{prop}
No proof is deemed necessary for (\ref{prop:trivial_bundle}), as
each bundle projection trivializes globally in the obvious way. The
$\rtriv$ symbol is a composite of $\times$ (indicating direct product)
and $\shortrightarrow$ or $\shortleftarrow$ (indicating the base
space).

If $M$ and $N$ are smooth manifolds as in (\ref{prop:trivial_bundle}),
then there are two particularly useful natural identifications. 
\begin{align*}
C^{\infty}\left(M,N\right) & \cong\Gamma\left(M\ltriv N\right) & C^{\infty}\left(M,N\right) & \cong\Gamma\left(N\rtriv M\right)\\
\phi & \mapsto\Id_{M}\times_{M}\phi & \phi & \mapsto\phi\times_{M}\Id_{M}\\
\pr_{2}^{M\times N}\circ\Phi & \mapsfrom\Phi & \pr_{1}^{N\times M}\circ\Phi & \mapsfrom\Phi
\end{align*}
These identifications can be thought of identifying a map $\phi\in C^{\infty}\left(M,N\right)$
with its graph in $M\times N$ and $N\times M$ respectively. Furthermore,
this allows bundle theory to be applied to reasoning about spaces
of maps. The symbols $M\ltriv N$ and $N\rtriv M$ now carry a significant
amount of meaning. Generally $N\rtriv M$ will be used in this paper,
for consistency with the $\Hom\left(V,W\right)\cong W\otimes V^{*}$
convention discussed in Section \ref{sec:Strongly-Typed-Linear-Algebra}.
The symbols $\rtriv$ and $\ltriv$ are examples of telescoping notation,
as they are built notationally on $\times$, and conceptually on the
direct product, which is what is denoted by $\times$. The arrow portion
of the symbols can be discarded when type-specificity is not needed.
\begin{prop}[Direct product bundle]
\label{proposition_direct_product_bundle} Let $\left(\mathcal{E},E,\pi^{E},M\right)$
and $\left(\mathcal{F},F,\pi^{F},N\right)$ be smooth  bundles. Then
\[
\pi^{E}\times\pi^{F}\colon E\times F\to M\times N,\,\left(e,f\right)\mapsto\left(\pi^{E}\left(e\right),\pi^{F}\left(f\right)\right)
\]
defines a smooth bundle $\left(\mathcal{E}\times\mathcal{F},E\times F,\pi^{E}\times\pi^{F},M\times N\right)$.
This bundle is called the \textbf{direct product of $\pi^{E}$ and
$\pi^{F}$}, and is not necessarily a trivial bundle. \end{prop}
\begin{proof}
Let $\Psi^{E}\colon\left(\pi^{E}\right)^{-1}\left(U\right)\to U\times\mathcal{E}$
and $\Psi^{F}\colon\left(\pi^{F}\right)^{-1}\left(V\right)\to V\times\mathcal{F}$
trivialize $\pi^{E}$ and $\pi^{F}$ over open sets $U\subseteq M$
and $V\subseteq N$ respectively. Then 
\[
\Psi^{E}\times\Psi^{F}\colon\left(\pi^{E}\right)^{-1}\left(U\right)\times\left(\pi^{F}\right)^{-1}\left(V\right)\to U\times\mathcal{E}\times V\times\mathcal{F}
\]
has inverse $\left(\Psi^{E}\right)^{-1}\times\left(\Psi^{F}\right)^{-1}$.
Note that
\begin{align*}
\left(\pi^{E}\right)^{-1}\left(U\right)\times\left(\pi^{F}\right)^{-1}\left(V\right) & =\left\{ \left(e,f\right)\in E\times F\mid\pi^{E}\left(e\right)\in U,\,\pi^{F}\left(f\right)\in V\right\} \\
 & =\left\{ \left(e,f\right)\in E\times F\mid\left(\pi^{E}\times\pi^{F}\right)\left(e,f\right)\in U\times V\right\} \\
 & =\left(\pi^{E}\times\pi^{F}\right)^{-1}\left(U\times V\right),
\end{align*}
and that
\[
P\colon\left(U\times\mathcal{E}\right)\times\left(V\times\mathcal{F}\right)\to\left(U\times V\right)\times\left(\mathcal{E}\times\mathcal{F}\right),\,\left(\left(u,e\right),\left(v,f\right)\right)\mapsto\left(\left(u,v\right),\left(e,f\right)\right)
\]
defines a diffeomorphism. Then 
\[
\Psi^{E\times F}:=P\circ\left(\Psi^{E}\times\Psi^{F}\right)\colon\left(\pi^{E}\times\pi^{F}\right)^{-1}\left(U\times V\right)\to\left(U\times V\right)\times\left(\mathcal{E}\times\mathcal{F}\right)
\]
defines a diffeomorphism, and
\begin{align*}
 & \pr_{1}^{\left(U\times V\right)\times\left(\mathcal{E}\times\mathcal{F}\right)}\circ\Psi^{E\times F}\left(e,f\right)\\
={} & \pr_{1}^{\left(U\times V\right)\times\left(\mathcal{E}\times\mathcal{F}\right)}\circ P\circ\left(\Psi^{E}\times\Psi^{F}\right)\left(e,f\right)\\
={} & \pr_{1}^{\left(U\times V\right)\times\left(\mathcal{E}\times\mathcal{F}\right)}\circ P\left(\Psi^{E}\left(e\right),\Psi^{F}\left(f\right)\right)\\
={} & \pr_{1}^{\left(U\times V\right)\times\left(\mathcal{E}\times\mathcal{F}\right)}\circ P\left(\Psi^{E}\left(e\right),\Psi^{F}\left(f\right)\right)\\
={} & \pr_{1}^{\left(U\times V\right)\times\left(\mathcal{E}\times\mathcal{F}\right)}\circ P\left(\left(\pr_{1}^{U\times\mathcal{E}}\circ\Psi^{E}\left(e\right),\pr_{2}^{U\times\mathcal{E}}\circ\Psi^{E}\left(e\right)\right),\left(\pr_{1}^{V\times\mathcal{F}}\circ\Psi^{F}\left(f\right),\pr_{2}^{V\times\mathcal{F}}\circ\Psi^{F}\left(f\right)\right)\right)\\
={} & \pr_{1}^{\left(U\times V\right)\times\left(\mathcal{E}\times\mathcal{F}\right)}\left(\left(\pr_{1}^{U\times\mathcal{E}}\circ\Psi^{E}\left(e\right),\pr_{1}^{V\times\mathcal{F}}\circ\Psi^{F}\left(f\right)\right),\left(\pr_{2}^{U\times\mathcal{E}}\circ\Psi^{E}\left(e\right),\pr_{2}^{V\times\mathcal{F}}\circ\Psi^{F}\left(f\right)\right)\right)\\
={} & \left(\pr_{1}^{U\times\mathcal{E}}\circ\Psi^{E}\left(e\right),\pr_{1}^{V\times\mathcal{F}}\circ\Psi^{F}\left(f\right)\right)\\
={} & \left(\pi^{E}\left(e\right),\pi^{F}\left(f\right)\right)\\
={} & \left(\pi^{E}\times\pi^{F}\right)\left(e,f\right),
\end{align*}
showing that $\Psi^{E\times F}$ trivializes $\pi^{E}\times\pi^{F}$
over $U\times V\subseteq M\times N$. Since $M\times N$ can be covered
by such trivializing sets, this establishes that $\pi^{E}\times\pi^{F}$
defines a smooth bundle. The typical fiber of $\pi^{E}\times\pi^{F}$
is $\mathcal{E}\times\mathcal{F}$. 
\end{proof}
A \textbf{smooth vector bundle} is a fiber bundle whose typical fiber
is a vector space and whose local trivializations are linear isomorphisms
when restricted to each fiber. If $\left(\mathcal{E},E,\pi,M\right)$
is a smooth vector bundle, then its \textbf{dual vector bundle} $\left(\mathcal{E}^{*},E^{*},\pi^{*},M\right)$
is a smooth vector bundle defined in the following way.
\[
E^{*}:=\coprod_{p\in M}\left(E_{p}\right)^{*},\qquad\pi^{*}\colon E^{*}\to M,\,\eta_{p}\mapsto p.
\]
Because $\mathcal{E}$ is a vector space, the notation $\mathcal{E}^{*}$
is already defined. In analogy with Section \ref{sec:Strongly-Typed-Linear-Algebra},
there are natural pairings on a vector bundle and its dual, defined
simply by evaluation. If $p\in M$, $\eta\in E_{p}^{*}$ and $e\in E_{p}$,
then $\eta\cdot_{E}e:=\eta\cdot_{E_{p}}e$ and $e\cdot_{E}\eta:=e\cdot_{E_{p}}\eta$.
Both expressions evaluate to $\eta\left(e\right)$. Natural traces
and $n$-fold tensor contraction can be defined analogously. Again,
while seemingly pedantic, the subscripted natural pairing notation
will prove to be a valuable tool in articulating and error-checking
calculations involving vector bundles. To generalize the rest of Section
\ref{sec:Strongly-Typed-Linear-Algebra} will require the definition
of additional structures.\\

For the remainder of this section, let $\left(\mathcal{E},E,\pi^{E},M\right)$
and $\left(\mathcal{F},F,\pi^{F},N\right)$ now be smooth \emph{vector}
bundles. The following construction is essentially an alternate notation
for $\pi^{E}\times\pi^{F}\colon E\times F\to M\times N$, but is one
that takes advantage of the fact that $\pi^{E}$ and $\pi^{F}$ are
vector bundles, and encodes in the notation the fact that the resulting
construction is also a vector bundle. This is analogous to how $V\times W$
is a vector space with a natural structure if $V$ and $W$ are vector
spaces, except that this is usually denoted by $V\oplus W$.
\begin{prop}[``Full'' direct sum vector bundle]
 \label{proposition_direct_sum_vector_bundle} If
\[
E\oplus_{M\times N}F:=E\times F,
\]
Then
\[
\pi^{E}\oplus_{M\times N}\pi^{F}:=\pi^{E}\times\pi^{F}\colon E\oplus_{M\times N}F\to M\times N
\]
defines a smooth vector bundle $\left(\mathcal{E}\oplus\mathcal{F},E\oplus_{M\times N}F,\pi^{E}\oplus_{M\times N}\pi^{F},M\times N\right)$,
called the \textbf{full direct sum of $\pi^{E}$ and $\pi^{F}$}.

For each $\left(p,q\right)\in M\times N$, the vector space structure
on $\left(\pi^{E}\oplus_{M\times N}\pi^{F}\right)^{-1}\left(p,q\right)$
is given in the following way. Let $\alpha\in\mathbb{R}$ and $\left(e_{1},f_{1}\right),\left(e_{2},f_{2}\right)\in\left(\pi^{E}\oplus_{M\times N}\pi^{F}\right)^{-1}\left(p,q\right)$.
Then
\[
\alpha\left(e_{1},f_{1}\right)+\left(e_{2},f_{2}\right)=\left(\alpha e_{1}+e_{2},\alpha f_{1}+f_{2}\right).
\]

It is critical to see (\ref{remark_direct_sum_and_tensor_product_bundle_notation})
for remarks on notation.\end{prop}
\begin{proof}
Let $U$, $V$, $\mathcal{E}$, $\mathcal{F}$, $P$, $\Psi^{E}$,
$\Psi^{F}$ and $\Psi^{E\times F}$ be as in the proof of (\ref{proposition_direct_product_bundle}),
and define $\Psi^{E\oplus_{M\times N}F}:=\Psi^{E\times F}$. Noting
that $\Psi^{E\oplus_{M\times N}F}$ is a smooth bundle isomorphism
over $\Id_{U\times V}$, so to show that $\Psi^{E\oplus_{M\times N}F}$
is a linear isomorphism in each fiber, it suffices to show that it
is linear in each fiber. Let $\alpha\in\mathbb{R}$, $\left(p,q\right)\in U\times V$
and $\left(e_{1},f_{1}\right),\left(e_{2},f_{2}\right)\in\left(\pi^{E}\oplus_{M\times N}\pi^{F}\right)^{-1}\left(p,q\right)$.
Then 
\begin{align*}
\Psi^{E\oplus_{M\times N}F}\left(\alpha e_{1}+e_{2},\alpha f_{1}+f_{2}\right)={} & P\circ\left(\Psi^{E}\times\Psi^{F}\right)\left(\alpha e_{1}+e_{2},\alpha f_{1}+f_{2}\right)\\
={} & P\left(\Psi^{E}\left(\alpha e_{1}+e_{2}\right),\Psi^{F}\left(\alpha f_{1}+f_{2}\right)\right)\\
={} & P\left(\alpha\,\Psi^{E}\left(e_{1}\right)+\Psi^{E}\left(e_{2}\right),\alpha\,\Psi^{F}\left(f_{1}\right)+\Psi^{F}\left(f_{2}\right)\right)\\
 & \mbox{(by trivial vector bundle structures on \ensuremath{U\times\mathcal{E}}and \ensuremath{V\times\mathcal{F}})}\\
={} & \alpha\, P\left(\Psi^{E}\left(e_{1}\right),\Psi^{F}\left(f_{1}\right)\right)+P\left(\Psi^{E}\left(e_{2}\right),\Psi^{F}\left(f_{2}\right)\right)\\
 & \mbox{(by trivial vector bundle structure on \ensuremath{\left(U\times V\right)\times\left(\mathcal{E}\times\mathcal{F}\right)})}\\
={} & \alpha\, P\circ\left(\Psi^{E}\times\Psi^{F}\right)\left(e_{1},f_{1}\right)+P\circ\left(\Psi^{E}\times\Psi^{F}\right)\left(e_{2},f_{2}\right)\\
={} & \alpha\,\Psi^{E\oplus_{M\times N}F}\left(e_{1},f_{1}\right)+\Psi^{E\oplus_{M\times N}F}\left(e_{2},f_{2}\right).
\end{align*}
Thus $\Psi^{E\oplus_{M\times N}F}$ is linear in each fiber, and because
it is invertible, it is a linear isomorphism in each fiber. In particular,
$\Psi^{E\oplus_{M\times N}F}$ is a smooth vector bundle isomorphism
over $\Id_{U\times V}$. Applying $\left(\Psi^{E\oplus_{M\times N}F}\right)^{-1}$
to the above equation gives
\[
\left(\alpha e_{1}+e_{2},\alpha f_{1}+f_{2}\right)=\alpha\left(e_{1},f_{1}\right)+\left(e_{2},f_{2}\right),
\]
as desired. 
\end{proof}
This construction differs from the Whitney sum of two vector bundles,
as the base spaces of the bundles are kept separate, and aren't even
required to be the same. This allows the identification of $T\left(M\times N\right)\to M\times N$
as $TM\oplus_{M\times N}TN\to M\times N$, which may be done without
comment later in this paper. Some important related structures are
$\pr_{1}^{*}\pi_{M}^{TM}\colon\pr_{1}^{*}TM\to M\times N$ and $\pr_{2}^{*}\pi_{N}^{TN}\colon\pr_{2}^{*}TN\to M\times N$,
where $\pr_{i}:=\pr_{i}^{M\times N}$.

The next construction is what will be used in the implementation of
smooth vector bundle morphisms as tensor fields. 
\begin{prop}[``Full'' tensor product bundle]
 \label{proposition_tensor_product_bundle} If 
\[
E\otimes_{M\times N}F:=\coprod_{\left(p,q\right)\in M\times N}E_{p}\otimes F_{q}\mbox{ (disjoint union)},
\]
Then
\[
\pi^{E}\otimes_{M\times N}\pi^{F}\colon E\otimes_{M\times N}F\to M\times N,\,\alpha^{ij}e_{i}\otimes f_{j}\mapsto\left(\pi^{E}\left(e_{1}\right),\pi^{F}\left(f_{1}\right)\right)\mbox{ (here, \ensuremath{\alpha^{ij}\in\mathbb{R}})}
\]
defines a smooth vector bundle $\left(\mathcal{E}\otimes\mathcal{F},E\otimes_{M\times N}F,\pi^{E}\otimes_{M\times N}\pi^{F},M\times N\right)$,
called the \textbf{full tensor product}%
\footnote{This construction is alluded to in \citep[pg. 121]{Kolar&Michor&Slovak},
but is not defined or discussed.%
}\textbf{ of $\pi^{E}$ and $\pi^{F}$}. 

It is critical to see (\ref{remark_direct_sum_and_tensor_product_bundle_notation})
for remarks on notation.\end{prop}
\begin{proof}
Since the argument $\alpha^{ij}e_{i}\otimes f_{j}$ in the definition
of $\pi^{E}\otimes_{M\times N}\pi^{F}$ is not necessarily unique,
the well-definedness of $\pi^{E}\otimes_{M\times N}\pi^{F}$ must
be shown. Let $\alpha^{ij}e_{i}^{1}\otimes f_{j}^{1}=\beta^{ij}e_{i}^{2}\otimes f_{j}^{2}$.
Then in particular, $\alpha^{ij}e_{i}^{1}\otimes f_{j}^{1},\beta^{ij}e_{i}^{2}\otimes f_{j}^{2}\in E_{p}\otimes F_{q}$
for some $\left(p,q\right)\in M\times N$, and therefore $e_{i}^{1},e_{i}^{2}\in E_{p}$
and $f_{j}^{1},f_{j}^{2}\in F_{q}$ for each index $i$ and $j$.
Thus $\pi^{E}\left(e_{1}^{1}\right)=p=\pi^{E}\left(e_{1}^{2}\right)$
and $\pi^{F}\left(f_{1}^{1}\right)=q=\pi^{F}\left(f_{1}^{2}\right)$,
so the expression defining $\pi^{E}\otimes_{M\times N}\pi^{F}$ is
well-defined. 

The set $E\otimes_{M\times N}F$ does not have an a priori global
smooth manifold structure, as it is defined as the disjoint union
of vector spaces. A smooth manifold structure compatible with that
of the constituent vector spaces will now be defined. 

Let $\Psi^{E}\colon\left(\pi^{E}\right)^{-1}\left(U\right)\to U\times\mathcal{E}$
and $\Psi^{F}\colon\left(\pi^{F}\right)^{-1}\left(V\right)\to V\times\mathcal{F}$
trivialize $\pi^{E}$ and $\pi^{F}$ over open sets $U\subseteq M$
and $V\subseteq N$ respectively, such that $\Psi^{E}$ and $\Psi^{F}$
are each linear in each fiber. Define 
\begin{eqnarray*}
\Psi^{E\otimes_{M\times N}F}\colon\left(\pi^{E}\otimes_{M\times N}\pi^{F}\right)^{-1}\left(U\times V\right) & \to & \left(U\times V\right)\times\left(\mathcal{E}\otimes\mathcal{F}\right)\\
X & \mapsto & \left(\left(\pi^{E}\otimes_{M\times N}\pi^{F}\right)\left(X\right),\left(\left(\pr_{2}^{U\times\mathcal{E}}\circ\Psi^{E}\right)\otimes\left(\pr_{2}^{V\times\mathcal{F}}\circ\Psi^{F}\right)\right)\left(X\right)\right).
\end{eqnarray*}
The map $\Psi^{E\otimes_{M\times N}F}$ is well-defined and smooth
in each fiber by construction, since for each $\left(p,q\right)\in U\times V$,
\[
\left(\pr_{2}^{U\times\mathcal{E}}\circ\Psi^{E}\right)\otimes\left(\pr_{2}^{V\times\mathcal{F}}\circ\Psi^{F}\right)\mid_{E_{p}\otimes E_{q}}\colon E_{p}\otimes E_{q}\to\mathcal{E}\otimes\mathcal{F}
\]
is a linear isomorphism by construction. Additionally, $\Psi^{E\otimes_{M\times N}F}$
has been constructed so that 
\[
\pr_{1}^{\left(U\times V\right)\times\left(\mathcal{E}\otimes_{M\times N}\mathcal{F}\right)}\circ\Psi^{E\otimes_{M\times N}F}=\pi^{E}\otimes_{M\times N}\pi^{F}
\]
on $\left(\pi^{E}\otimes_{M\times N}\pi^{F}\right)^{-1}\left(U\times V\right)$.
Define the smooth structure on $\left(\pi^{E}\otimes_{M\times N}\pi^{F}\right)^{-1}\left(U\times V\right)\subseteq E\otimes_{M\times N}F$
by declaring $\Psi^{E\otimes_{M\times N}F}$ to be a diffeomorphism.
The map $\pi^{E}\otimes_{M\times N}\pi^{F}$ is trivialized over $U\times V$.
The set $E\otimes_{M\times N}F$ can be covered by such trivializing
open sets. Thus $E\otimes_{M\times N}F$ has been shown to be locally
diffeomorphic to the direct product of smooth manifolds, and therefore
it has been shown to be a smooth manifold. With respect to the smooth
structure on $E\otimes_{M\times N}F$, the map $\pi^{E}\otimes_{M\times N}\pi^{F}$
is smooth, and has therefore been shown to define a smooth vector
bundle.\end{proof}
\begin{rem}[Notation regarding base space]
 \label{remark_direct_sum_and_tensor_product_bundle_notation} The
``full'' direct sum (\ref{proposition_direct_sum_vector_bundle})
and ``full'' tensor product (\ref{proposition_tensor_product_bundle})
bundle constructions allow direct sums and tensor products to be taken
of vector bundles when the base spaces differ. If the base spaces
are the same, then the construction ``joins'' them, producing a
vector bundle over that shared base space. For example, if $E$ and
$F$ are vector bundles over $M$, then $E\otimes_{M\times M}F$ has
base space $M\times M$, while $E\otimes F$ has base space $M$.
The base space can be specified in either case as a notational aide;
the latter example would be written as $E\otimes_{M}F$. If no subscript
is provided on the $\otimes$ symbol, then the base spaces are ``joined''
if possible (if they are the same space), otherwise they are kept
separate, as in the ``full'' tensor product construction. This notational
convention conforms to the standard Whitney sum and tensor product
bundle notation, and uses the notion of telescoping notation to provide
more specificity when necessary. \\

Given a fiber bundle, a natural vector bundle can be constructed ``on
top'' of it, essentially quantifying the variations of bundle elements
along each fiber. This is known as the vertical bundle, and it plays
a critical role in the development of Ehresmann connections, which
provide the ``horizontal complement'' to the vertical bundle. \end{rem}
\begin{prop}[Vertical bundle]
 Let $\pi^{E}\colon E\to M$ define a smooth {[}fiber{]} bundle.
If $VE:=\ker T\pi^{E}\leq TE$, then $\pi^{VE}:=\pi_{E}^{TE}\mid_{VE}\colon VE\to E$
defines a smooth vector bundle subbundle of $\pi_{E}^{TE}\colon TE\to E$,
called the \textbf{vertical bundle} over $E$. Furthermore, the fiber
over $e\in E$ is $V_{e}E=T_{e}E_{\pi^{E}\left(e\right)}\leq T_{e}E$.\end{prop}
\begin{proof}
Because $\pi^{E}$ is a smooth surjective submersion, $VE\to E$ is
a subbundle of $TE\to E$ having corank $\dim M$ and therefore rank
equal to that of $E$. Furthermore, if $e\in E$ and $\epsilon\mapsto e_{\epsilon}\in E_{\pi^{E}\left(e\right)}$,
then $\delta e_{\epsilon}$ represents an arbitrary element of $T_{e}E_{\pi^{E}\left(e\right)}$,
and $T\pi^{E}\left(\delta e_{\epsilon}\right)=\delta\left(\pi^{E}\left(e_{\epsilon}\right)\right)=\delta\left(\pi\left(e\right)\right)=0$,
showing that $\delta e_{\epsilon}\in\ker T\pi^{E}$, and therefore
that $\delta e_{\epsilon}\in V_{e}E$. This shows that $T_{e}E_{\pi^{E}\left(e\right)}\subseteq V_{e}E$.
Because $\dim T_{e}E_{\pi^{E}\left(e\right)}=\rank E$, this shows
that $T_{e}E_{\pi^{E}\left(e\right)}=V_{e}E$.
\end{proof}
Given the extra structure that a vector bundle provides over a {[}fiber{]}
bundle, there is a canonical smooth vector bundle isomorphism which
adds significant value to the pullback bundle formalism used throughout
this paper. This can be seen put to greatest use in Part \ref{part:Riemannian-Calculus-of-Variations},
for example, in development of the first variation (see (\ref{thm:first_variation_of_L})).
\begin{prop}[Vertical bundle as pullback]
 \label{prop:vertical-bundle-as-pullback} If $\pi\colon E\to M$
defines a smooth vector bundle, then
\begin{eqnarray*}
\iota_{VE}^{\pi^{*}E}\colon\pi^{*}E & \to & VE,\\
\left(x,y\right) & \mapsto & \delta_{\epsilon}\left(x+\epsilon y\right)
\end{eqnarray*}
is a smooth vector bundle isomorphism over $\Id_{E}$, called the
\textbf{vertical lift}, having inverse
\begin{eqnarray*}
\iota_{\pi^{*}E}^{VE}\colon\delta e_{\epsilon} & \mapsto & \left(e_{0},\lim_{\epsilon\to0}\frac{e_{\epsilon}-e_{0}}{\epsilon}\right),
\end{eqnarray*}
where, without loss of generality, $e_{\epsilon}$ is an $E$-valued
variation which lies entirely in a single fiber.\end{prop}
\begin{proof}
It is clear that $\iota_{VE}^{\pi^{*}E}$ is linear and injective
on each fiber. By a dimension counting argument, it is therefore an
isomorphism on each fiber. Because it preserves the basepoint, it
is a vector bundle isomorphism over $\Id_{E}$. Because the map $\left(x,y,\epsilon\right)\mapsto x+\epsilon y$
is smooth, so is the defining expression for $\iota_{VE}^{\pi^{*}E}$,
thereby establishing smoothness. That $\iota_{\pi^{*}E}^{VE}$ inverts
$\iota_{VE}^{\pi^{*}E}$ is a trivial calculation.
\end{proof}

\section{Strongly-Typed Tensor Field Operations\label{sec:Strongly-Typed-Tensor-Field}}

Because vector bundles and the related operations can be thought of
conceptually as ``sheaves of linear algebra'', the constructions
in Section \ref{sec:Strongly-Typed-Linear-Algebra}, generalized earlier
in this section, can be further generalized to the setting of sections
of vector bundles.\\

If $E,F,G$ are smooth vector bundles over $M$, then define the natural
pairing of a tensor field with a vector:
\begin{eqnarray*}
\cdot_{F}\colon\Gamma\left(E\otimes_{M}F^{*}\right)\times F & \to & E,\\
\left(e\otimes_{M}\phi,f\right) & \mapsto & e\left(\pi^{F}\left(f\right)\right)\left[\phi\left(\pi^{F}\left(f\right)\right)\cdot_{F}f\right],
\end{eqnarray*}
extending linearly to general tensor fields. Further, define the natural
pairing of tensor fields:
\begin{eqnarray*}
\cdot_{F}\colon\Gamma\left(E\otimes_{M}F^{*}\right)\times\Gamma\left(F\otimes_{M}G\right) & \to & \Gamma\left(E\otimes_{M}G\right),\\
\left(e\otimes_{M}\phi,f\otimes_{M}g\right) & \mapsto & \left(p\mapsto e\left(p\right)\otimes_{M}\left(\phi\left(p\right)\cdot_{F_{p}}f\left(p\right)\right)\otimes_{M}g\left(p\right)\right)\\
 &  & =\left(p\mapsto\left(\phi\left(p\right)\cdot_{F_{p}}f\left(p\right)\right)\left(e\otimes_{M}g\right)\left(p\right)\right),
\end{eqnarray*}
extending linearly to general tensor fields. This multiple use of
the $\cdot_{F}$ symbol is a concept known as \textbf{operator overloading}
in computer programming. No ambiguity is caused by this overloading,
as the particular use can be inferred from the types of the operands.
As before, the subscript $F$ may be optionally omitted when clear
from context.\\

The permutations defined in Section \ref{sec:Strongly-Typed-Linear-Algebra}
are generalized as tensor fields. If $F_{1},\dots,F_{n}$ are smooth
vector bundles over $M$, and $\sigma\in S_{n}$ is a permutation,
then $\sigma$ can act on $F_{1}\otimes_{M}\dots\otimes_{M}F_{n}$
by permuting its factors, and therefore can be identified with a tensor
field
\[
\sigma\in\Gamma\left(F_{1}^{*}\otimes_{M}\dots\otimes_{M}F_{n}^{*}\otimes_{M}F_{\sigma^{-1}\left(1\right)}\otimes_{M}\dots\otimes_{M}F_{\sigma^{-1}\left(n\right)}\right)
\]
defined by
\[
\left(f_{1}\otimes_{M}\dots\otimes_{M}f_{n}\right)\cdot_{F_{1}^{*}\otimes_{M}\dots\otimes_{M}F_{n}^{*}}\sigma:=f_{\sigma^{-1}\left(1\right)}\otimes_{M}\dots\otimes_{M}f_{\sigma^{-1}\left(n\right)}.
\]
An important feature of such permutation tensor fields is that they
are parallel with respect to covariant derivatives on the factors
$F_{1},\dots,F_{n}$ (see (\ref{cor:permutation_tensor_fields_are_parallel})
for more on this).

\section{Pullback Bundles \label{sec:Pullback-Bundles}}

The pullback bundle, defined below, is a crucial building block for
many important bundle constructions, as it enriches the type system
dramatically, and allows the tensor formulation of linear algebra
to be extended to the vector bundle setting. In particular, the abstract,
global formulation of the space of smooth vector bundle morphisms
over a map $\phi\colon M\to N$ is achieved quite cleanly using a
pullback bundle. Furthermore, the use of pullback bundles and pullback
covariant derivatives simplifies what would otherwise be local coordinate
calculations, thereby giving more insight into the geometric structure
of the problem.\\

For the duration of this section, let $\left(\mathcal{F},F,\pi,N\right)$
be a smooth bundle having rank $r$.
\begin{prop}[Pullback bundle]
 \label{proposition_pullback_bundle} Let $M$ and $N$ be smooth
manifolds and let $\phi\colon M\to N$ be smooth. If 
\[
\phi^{*}F:=\left\{ \left(m,f\right)\in M\times F\mid\phi\left(m\right)=\pi\left(f\right)\right\} ,
\]
and 
\[
\pi^{\phi^{*}F}:=\pr_{1}^{M\times F}\mid_{\phi^{*}F}\colon\phi^{*}F\to M,\,\left(m,f\right)\mapsto m,
\]
then $\left(\mathcal{F},\phi^{*}F,\pi^{\phi^{*}F},M\right)$ defines
a smooth bundle. In particular, $\phi^{*}F$ is a smooth manifold
having dimension $\dim M+\rank\pi$. The bundle defined by $\pi^{\phi^{*}F}$
is called the \textbf{pullback of $\pi$ by $\phi$}. \end{prop}
\begin{proof}
Recalling that $\mathcal{F}$ denotes the typical fiber of $\pi$,
let $\Psi\colon\pi^{-1}\left(U\right)\to U\times\mathcal{F}$ trivialize
$\pi$ over open set $U\subseteq N$. Define

\[
\Psi_{\phi}\colon\phi^{*}\left(\pi^{-1}\left(U\right)\right)\to\phi^{-1}\left(U\right)\times\mathcal{F},\,\left(m,f\right)\mapsto\left(m,\pr_{2}^{U\times\mathcal{F}}\circ\Psi\left(f\right)\right)
\]
and
\[
\Psi_{\phi}^{-1}\colon\phi^{-1}\left(U\right)\times\mathcal{F}\to\phi^{*}\left(\pi^{-1}\left(U\right)\right),\,\left(m,f\right)\mapsto\left(m,\Psi^{-1}\left(\phi\left(m\right),f\right)\right).
\]

Claim (1): $\Psi_{\phi}$ and $\Psi_{\phi}^{-1}$ are smooth. Proof:
$\phi^{*}\left(\pi^{-1}\left(U\right)\right)\subseteq\phi^{-1}\left(U\right)\times\pi^{-1}\left(U\right)$,
and $\Psi_{\phi}$ is clearly smooth as a map defined on the larger
manifold. Therefore it restricts to a smooth map on $\phi^{*}\left(\pi^{-1}\left(U\right)\right)$.
An analogous argument shows that $\Psi_{\phi}^{-1}$ is smooth. Claim
(1) proved.

Claim (2): $\Psi_{\phi}^{-1}$ inverts $\Psi_{\phi}$. Proof: Let
$\left(m,f\right)\in\phi^{*}\left(\pi^{-1}\left(U\right)\right)$.
Then 
\begin{align*}
\Psi_{\phi}^{-1}\circ\Psi_{\phi}\left(m,f\right) & =\Psi_{\phi}^{-1}\left(m,\pr_{2}^{U\times\mathcal{F}}\circ\Psi\left(f\right)\right)\\
 & =\left(m,\Psi^{-1}\left(\phi\left(m\right),\pr_{2}^{U\times\mathcal{F}}\circ\Psi\left(f\right)\right)\right)\\
 & =\left(m,\Psi^{-1}\left(\pi\left(f\right),\pr_{2}^{U\times\mathcal{F}}\circ\Psi\left(f\right)\right)\right)\mbox{ (since \ensuremath{\phi\left(m\right)=\pi\left(f\right)})}\\
 & =\left(m,\Psi^{-1}\left(\pr_{1}^{U\times\mathcal{F}}\circ\Psi\left(f\right),\pr_{2}^{U\times\mathcal{F}}\circ\Psi\left(f\right)\right)\right)\\
 & =\left(m,\Psi^{-1}\circ\Psi\left(f\right)\right)\\
 & =\left(m,f\right).
\end{align*}
With $g\in\mathcal{F}$,
\begin{align*}
\Psi_{\phi}\circ\Psi_{\phi}^{-1}\left(m,g\right) & =\Psi_{\phi}\left(m,\Psi^{-1}\left(\phi\left(m\right),g\right)\right)\\
 & =\left(m,\pr_{2}^{U\times\mathcal{F}}\circ\Psi\circ\Psi^{-1}\left(\phi\left(m\right),g\right)\right)\\
 & =\left(m,\pr_{2}^{U\times\mathcal{F}}\left(\phi\left(m\right),g\right)\right)\\
 & =\left(m,g\right),
\end{align*}
proving Claim (2).

Claim (3): $\Psi_{\phi}$ trivializes $\pi^{\phi^{*}F}$ over $\phi^{-1}\left(U\right)\subseteq M$.
Proof: Let $\left(m,f\right)\in\phi^{*}\left(\pi^{-1}\left(U\right)\right)$.
Then
\[
\pr_{1}^{\phi^{-1}\left(U\right)\times\mathcal{F}}\circ\Psi_{\phi}\left(m,f\right)=\pr_{1}^{\phi^{-1}\left(U\right)\times\mathcal{F}}\circ\left(m,\pr_{2}^{U\times\mathcal{F}}\circ\Psi\left(f\right)\right)=m=\pi^{\phi^{*}F}\left(m,f\right),
\]
and by claims (1) and (2), $\Psi_{\phi}$ is a diffeomorphism, so
$\Psi_{\phi}$ trivializes $\pi^{\phi^{*}F}$ over $\phi^{-1}\left(U\right)\subseteq M$.
Claim (3) proved.

Since $M$ can be covered with sets as in claim (3) and since the
typical fiber of $\pi^{\phi^{*}F}$ is diffeomorphic to $\mathcal{F}$,
this shows that $\pi^{\phi^{*}F}$ defines a smooth bundle $\left(\mathcal{F},\phi^{*}F,\pi^{\phi^{*}F},M\right)$.
Because $\phi^{*}F$ is locally diffeomorphic to the product of an
open subset of $M$ with $\mathcal{F}$, $\phi^{*}F$ has been shown
to be a smooth manifold having dimension $\dim M+\dim\mathcal{F}=\dim M+\rank\pi$. 
\end{proof}
While the pullback bundle is constructed as a submanifold of a direct
product, there is a natural bundle morphism into the pulled-back bundle,
which serves as an interface to maps defined on the pulled-back bundle.
Usually this morphism is notationally suppressed, just as naturally
isomorphic spaces can be identified without explicit notation. 
\begin{cor}[Pullback fiber projection bundle morphism]
 \label{cor:pullback-fiber-projection} If $\phi\colon M\to N$ is
smooth, then
\begin{eqnarray*}
\rho_{F}^{\phi^{*}F}\colon\phi^{*}F & \to & F,\\
\left(m,f\right) & \mapsto & f
\end{eqnarray*}
is a smooth bundle morphism over $\phi$ which is an isomorphism when
restricted to any fiber of $\phi^{*}F$.
\end{cor}
Because $\rho_{F}^{\phi^{*}F}$ is the projection $\pr_{F}^{M\times F}\mid_{\phi^{*}F}$,
its tangent map is also just the projection $\pr_{TF}^{TM\oplus TF}\mid_{T\phi^{*}F}$. 
\begin{prop}[Bundle pullback is a contravariant functor]
 \label{prop:pullback_is_contravariant_functor} The map of categories
\begin{eqnarray*}
\mbox{Pullback}\colon\mbox{Manifold} & \to & \left\{ \mbox{Bundle}\left(M\right)\mid M\in\mbox{Manifold}\right\} ,\\
M & \mapsto & \mbox{Bundle}\left(M\right),\\
\left(\phi\colon M\to N\right) & \mapsto & \left(\mbox{Bundle}\left(N\right)\to\mbox{Bundle}\left(M\right),\,\left(\mathcal{F},F,\pi,N\right)\mapsto\left(\mathcal{F},\phi^{*}F,\pi^{\phi^{*}F},M\right)\right)
\end{eqnarray*}
is a contravariant functor. Here, naturally isomorphic bundles in
$\mbox{Bundle}\left(M\right)$, for each manifold $M$, are identified
(along with the corresponding morphisms).\end{prop}
\begin{proof}
Noting that 
\[
\Id_{N}^{*}F=\left\{ \left(n,f\right)\in N\times F\mid\Id_{N}\left(n\right)=\pi\left(f\right)\right\} \cong F
\]
and that
\begin{align*}
\left(\Id_{N}^{*}\pi\right)\left(n,f\right) & =\left(\pr_{1}^{N\times F}\mid_{\Id_{N}^{*}F}\right)\left(n,f\right)=n=\pi\left(f\right)\\
\implies\Id_{N}^{*}\pi & \cong\pi,
\end{align*}
it follows that $\mbox{Pullback}\left(\Id_{N}\right)=\Id_{\mbox{Bundle}\left(N\right)}=\Id_{\mbox{Pullback}\left(N\right)}$,
i.e. $\mbox{Pullback}$ satisfies the identity axiom of functoriality. 

For the contravariance axiom, let $\phi\colon M\to N$ and $\psi\colon L\to M$
be smooth manifold morphisms and let $\left(\mathcal{F},F,\pi,N\right)$
be a smooth bundle. Then
\begin{align*}
\psi^{*}\phi^{*}F & =\left\{ \left(\ell,p\right)\in L\times\phi^{*}F\mid\psi\left(\ell\right)=\pi^{\phi^{*}F}\left(p\right)\right\} \\
 & =\left\{ \left(\ell,\left(m,f\right)\right)\in L\times\left(M\times F\right)\mid\psi\left(\ell\right)=\pi^{\phi^{*}F}\left(m,f\right)\mbox{ and }\phi\left(m\right)=\pi\left(f\right)\right\} \\
 & =\left\{ \left(\ell,\left(m,f\right)\right)\in L\times\left(M\times F\right)\mid\psi\left(\ell\right)=m\mbox{ and }\phi\left(m\right)=\pi\left(f\right)\right\} \\
 & \cong\left\{ \left(\ell,f\right)\in L\times F\mid\phi\circ\psi\left(\ell\right)=\pi\left(f\right)\right\} \\
 & =\left(\phi\circ\psi\right)^{*}F
\end{align*}
and
\begin{align*}
\pi^{\psi^{*}\phi^{*}F}\left(\ell,\left(m,f\right)\right) & =\left(\pr_{1}^{L\times\phi^{*}F}\mid_{\psi^{*}\phi^{*}F}\right)\left(\ell,\left(m,f\right)\right)=\ell\mbox{ and}\\
\pi^{\left(\phi\circ\psi\right)^{*}F}\left(\ell,f\right) & =\left(\pr_{1}^{L\times F}\mid_{\left(\phi\circ\psi\right)^{*}F}\right)\left(\ell,f\right)=\ell,
\end{align*}
showing that $\pi^{\psi^{*}\phi^{*}F}\cong\pi^{\left(\phi\circ\psi\right)^{*}F}$,
and therefore
\[
\mbox{Pullback}\left(\psi\right)\circ\mbox{Pullback}\left(\phi\right)=\mbox{Pullback}\left(\phi\circ\psi\right),
\]
establishing $\mbox{Pullback}$ as a contravariant functor. 
\end{proof}
The space of sections of a pullback bundle is easily quantified.
\[
\Gamma\left(\phi^{*}F\right)=\left\{ \sigma\in C^{\infty}\left(M,\phi^{*}F\right)\mid\pi^{\phi^{*}F}\circ\sigma=\Id_{M}\right\} .
\]
This space will be central in the theory developed in the rest of
this paper. Furthermore, it is naturally identified with the space
of sections along the pullback map;
\[
\Gamma_{\phi}\left(F\right):=\left\{ \Sigma\in C^{\infty}\left(M,F\right)\mid\pi^{F}\circ\Sigma=\phi\right\} .
\]
These spaces are naturally isomorphic to one another, and therefore
an identification can be made when convenient. While the former space
is more correct from a strongly typed standpoint, the latter space
is a convenient and intuitive representational form. The particular
correspondence depends heavily on the fact that $\phi^{*}F$ is a
submanifold of $M\times F$.
\begin{eqnarray*}
\Gamma\left(\phi^{*}F\right) & \cong & \Gamma_{\phi}\left(F\right)\\
\sigma & \mapsto & \pr_{2}^{M\times F}\circ\sigma,\\
\Id_{M}\times_{M}\Sigma & \mapsfrom & \Sigma.
\end{eqnarray*}
Furthermore, if $f\in\Gamma\left(F\right)$, then $f\circ\phi\in\Gamma_{\phi}\left(F\right)$.
Note that it is \emph{not} true that any $\sigma\in\Gamma_{\phi}\left(F\right)$
can be written as $f\circ\phi$ for some $f\in\Gamma\left(F\right)$,
for example when there exists some distinct $p,q\in M$ such that
$\phi\left(p\right)=\phi\left(q\right)$ and $\sigma\left(p\right)\neq\sigma\left(q\right)$.
Furthermore, the representation $f\circ\phi$ is generally non-unique,
for example when $\phi$ is not surjective, sections $f_{1},f_{2}\in\Gamma\left(F\right)$
which differ only away from the image of $\phi$ will still give $f_{1}\circ\phi=f_{2}\circ\phi$.
Before developing the notion of a linear connection on a pullback
bundle, it will be necessary to address these features which, while
inconvenient, provide the strength of the pullback bundle and pullback
covariant derivative (see (\ref{remark:pullback_derivative_captures_fiber_variation})).
\begin{lem}[Local representation of $\Gamma_{\phi}\left(F\right)$ elements]
 \label{lemma_representation_of_sections_of_a_pullback_bundle} Recall
that $r$ denotes the rank of smooth bundle $F$. If $\sigma\in\Gamma_{\phi}\left(F\right)$
then each point $p\in M$ has some neighborhood $U$ in which $\sigma$
can be written locally as $\sigma\mid_{U}=\sigma^{i}\, f_{i}\circ\phi\mid_{U}$,
where $f_{1},\dots,f_{r}\in\Gamma\left(F\mid_{\phi\left(U\right)}\right)$
is a frame for $F\mid_{\phi\left(U\right)}$, and $\sigma^{1},\dots,\sigma^{r}\in C^{\infty}\left(U,\mathbb{R}\right)$
are defined by $\sigma^{i}=\left(f^{i}\circ\phi\mid_{U}\right)\cdot_{F}\sigma\mid_{U}$. \end{lem}
\begin{proof}
Let $p\in M$, let $V\subseteq N$ be a neighborhood of $\phi\left(p\right)$
over which $F\mid_{V}$ is trivial, and let $U=\phi^{-1}\left(V\right)$,
so that $U$ is a neighborhood of $p$. Let $f_{1},\dots,f_{r}\in\Gamma\left(F\mid_{V}\right)$
be a frame for $F\mid_{V}$ (i.e. $F\mid_{\phi\left(U\right)}$),
and let $f^{1},\dots,f^{r}\in\Gamma\left(\left(F\mid_{V}\right)^{*}\right)$
be the corresponding coframe (i.e. the unique $f^{1},\dots,f^{r}$
such that $f^{i}\cdot_{F}f_{j}=\delta_{j}^{i}$ for each $i,j$).
Define $\sigma^{i}\in C^{\infty}\left(M,\mathbb{R}\right)$ by $\sigma^{i}=\left(f^{i}\circ\phi\mid_{U}\right)\cdot_{F}\sigma\mid_{U}$.
Then 
\begin{align*}
\sigma^{i}\, f_{i}\circ\phi\mid_{U} & =\left(f^{i}\circ\phi\mid_{U}\right)\cdot_{F}\sigma\mid_{U}\, f_{i}\circ\phi\mid_{U}\\
 & =\left(\left(f_{i}\circ\phi\mid_{U}\right)\otimes_{U}\left(f^{i}\circ\phi\mid_{U}\right)\right)\cdot_{F}\sigma\mid_{U}\\
 & =\left(\left(f_{i}\otimes_{V}f^{i}\right)\circ\phi\mid_{U}\right)\cdot_{F}\sigma\mid_{U}\\
 & =\left(\Id_{F\mid_{V}}\circ\phi\mid_{U}\right)\cdot_{F}\sigma\mid_{U}\\
 & =\sigma\mid_{U},
\end{align*}
as desired. 
\end{proof}
Some literature uses expressions of the form $f\circ\phi\in\Gamma_{\phi}\left(F\right)$
along with an implicit use of the section-identifying isomorphism
to write down particular sections of pullback bundles. In most cases,
this tacit identification of spaces is harmless, but certain highly
involved calculations may suffer from it. The section that $f\circ\phi$
corresponds to under said isomorphism is $\Id_{M}\times_{M}\left(f\circ\phi\right)\in\Gamma\left(\phi^{*}F\right)$.
However, because this expression is unwieldy and therefore a more
compact and contextually meaningful expression is called for. 
\begin{defn}[Pullback section]
 If $f\in\Gamma\left(F\right)$ and $\phi\colon M\to N$ is smooth,
then define 
\[
\phi^{*}f:=\Id_{M}\times_{M}\left(f\circ\phi\right)\in\Gamma\left(\phi^{*}F\right).
\]
This is known as a \textbf{pullback section}.
\end{defn}
The pullback section is deservedly named. If $\phi\colon M\to N$
and $\psi\colon L\to M$ are smooth, then $\psi^{*}\phi^{*}f\cong\left(\phi\circ\psi\right)^{*}f$
in the sense of the proof of (\ref{prop:pullback_is_contravariant_functor}).
\begin{prop}[Bundle pullback commutes with tensor product]
 If $E$ and $F$ are smooth vector bundles over manifold $N$ and
$\phi\colon M\to N$ is smooth, then the map
\begin{eqnarray*}
\phi^{*}E\otimes_{M}\phi^{*}F & \to & \phi^{*}\left(E\otimes_{N}F\right),\\
\left(m,e\right)\otimes_{M}\left(m,f\right) & \mapsto & \left(m,e\otimes_{N}f\right)
\end{eqnarray*}
(extended linearly to general tensors) is a smooth vector bundle isomorphism.\end{prop}
\begin{proof}
Let $c$ denote the above map. The well-definedness of $c$ comes
from the universal mapping property on multilinear forms which induces
a linear map on a corresponding tensor product. If $c\left(\left(m,e\right)\otimes_{M}\left(m,f\right)\right)=0$,
then $e\otimes_{N}f=0,$ which implies that $e=0$ or $f=0$, and
therefore that $\left(m,e\right)\otimes_{M}\left(m,f\right)=0$. Because
there exists a basis for $\left(\phi^{*}E\otimes_{M}\phi^{*}F\right)_{m}$
consisting only of simple tensors, this implies that $c$ is injective,
and by a dimensionality argument, that $c$ is an isomorphism. The
map is clearly smooth and respects the fiber structures of its domain
and codomain. Thus $c$ is a smooth vector bundle isomorphism.
\end{proof}
The contravariance of pullback and its naturality with respect to
tensor product are two essential properties which provide some of
the flexibility and precision of the strongly typed tensor formalism
described in this paper. This will become quite apparent in Part \ref{part:Riemannian-Calculus-of-Variations}.\\

\begin{rem}[Tensor field formulation of smooth vector bundle morphisms]
 \label{remark:SVB-morphisms-as-tensor-fields} A particularly useful
application of pullback bundles is in forming a rich type system for
smooth vector bundle morphisms. This approach was inspired by \citep[pg. 11]{Xin}.
Let $\pi^{E}\colon E\to M$ and $\pi^{F}\colon F\to N$ be smooth
vector bundles, and let $\phi\colon M\to N$ be smooth. Consider $\Hom_{\phi}\left(E,F\right)$,
i.e. the space of smooth vector bundle morphisms over the map $\phi$.
There is a natural identification with another space which lets the
base map $\phi$ play a more direct role in the space's type. In particular,
\begin{eqnarray*}
\Hom_{\phi}\left(E,F\right) & \cong & \Hom_{\Id_{M}}\left(E,\phi^{*}F\right),\\
A & \mapsto & \pi^{E}\times_{E}A,\\
\pr_{2}^{M\times F}\circ B & \mapsfrom & B.
\end{eqnarray*}
This particular identification of smooth vector bundle morphisms over
$\phi$ can now be directly translated into the tensor field formalism,
analogously to (\ref{def:linear_maps_as_tensors}). 
\begin{eqnarray*}
\Gamma\left(\phi^{*}F\otimes_{M}E^{*}\right) & \to & \Hom_{\Id_{M}}\left(E,\phi^{*}F\right),\\
A & \mapsto & \left(e\mapsto A\cdot_{E}e\right).
\end{eqnarray*}
The inverse image of $B\in\Hom_{\Id_{M}}\left(E,\phi^{*}F\right)$
is given locally; let $\left(e_{i}\right)$ and $\left(f_{i}\right)$
denote local frames for $E$ and $F$ in neighborhoods $U\subseteq M$
and $V\subseteq N$ respectively, with $\phi\left(U\right)\subseteq V$,
and let $\left(e^{i}\right)$ and $\left(f^{i}\right)$ denote their
dual coframes. Then the tensor field corresponding to $B$ is given
locally in $U$ by $B_{j}^{i}\,\phi^{*}f_{i}\otimes_{M}e^{j}$, where
$B_{j}^{i}:=\phi^{*}df^{i}\circ B\circ e_{j}\in C^{\infty}\left(U,\mathbb{R}\right)$.

Quantifying smooth vector bundle morphisms as the tensor fields lends
itself naturally to doing calculus on vector and tensor bundles, as
the relevant derivatives (covariant derivatives) take the form of
tensor fields. The type information for a particular vector bundle
morphism is encoded in the relevant tensor bundle.
\end{rem}

\section{Tangent Map as a Tensor Field \label{sec:Tangent-Map-as-Tensor-Field}}

This section deals specifically with the tangent map operator by using
concepts from Section \ref{sec:Strongly-Typed-Tensor-Field} and Section
\ref{sec:Pullback-Bundles} to place it in a strongly typed setting
and to prepare to unify a few seemingly disparate concepts and notation
for some tangible benefit (in particular, see Section \ref{sec:Curvature-and-Commutation-of-Derivatives}).\\

Given a smooth map $\phi\colon M\to N$, its tangent map $T\phi\colon TM\to TN$
is a smooth vector bundle morphism over $\phi$, so by (\ref{remark:SVB-morphisms-as-tensor-fields}),
is naturally identified with a tensor field 
\[
\nnabla^{M\to N}\phi\in\Gamma\left(\phi^{*}TN\otimes_{M}T^{*}M\right),
\]
which may be denoted by $\nnabla\phi$ where type pedantry is deemed
unnecessary. This construction is known as a \textbf{two-point tensor
field} \citep[pg. 70]{Marsden&Hughes}. The inscribed $\circ$ symbol
in $\nnabla$ is used to denote that this is a nonlinear derivative,
thereby distinguishing it from a linear covariant derivative. 
\begin{rem}[Generalized covariant derivative]
 \label{rem:generalized-covariant-derivative} The well-known one-to-one
correspondence between linear connections and linear covariant derivatives
\citep[pg. 520]{JeffMLee} generalizes to a one-to-one correspondence
between Ehresmann connections and a generalized notion of covariant
derivative. To give a partial definition for the purposes of utility,
a \textbf{generalized covariant derivative} on a smooth {[}fiber{]}
bundle $F\to N$ is a map $\nabla$ on $\Gamma\left(F\right)$ such
that $\nabla\sigma\in\Gamma\left(\sigma^{*}VF\otimes_{N}T^{*}N\right)$
for each $\sigma\in\Gamma\left(F\right)$. The space of maps $C^{\infty}\left(M,N\right)$
is naturally identified as $\Gamma\left(N\rtriv M\right)$, and there
is a natural Ehresmann connection on the bundle $N\rtriv M$, whose
corresponding covariant derivative is the tangent map operator. This
is the subject of another of the author's papers and will not be discussed
here further. This is mentioned here to incorporate linear covariant
derivatives (to be introduced and discussed in Section \ref{sec:Linear-Covariant-Derivatives})
and the tangent map operator (a nonlinear covariant derivative) under
the single category ``covariant derivative''.\\

\end{rem}
There is a subtle issue regarding construction of the cotangent map
of $\phi$ which is handled easily by the tensor field construction.
In particular, while the cotangent map $T^{*}\phi$ is the pointwise
adjoint of the tangent map $T\phi$, i.e. for each $p\in M$, $T_{p}\phi\colon T_{p}M\to T_{\phi\left(p\right)}N$
is linear and $T_{p}^{*}\phi\colon T_{\phi\left(p\right)}^{*}N\to T_{p}^{*}M$
is the adjoint of $T_{p}\phi$, it does not follow that $T^{*}\phi\in\Hom\left(T^{*}N,T^{*}M\right)$,
being some sort of ``total adjoint'' of $T\phi\in\Hom\left(TM,TN\right)$.
The obstruction is due to the fact that $\phi$ may not be surjective,
so there may be some fiber $T_{q}^{*}N$ that is not of the form $T_{\phi\left(p\right)}^{*}N$,
and therefore the domain could not be all of $T^{*}N$. Furthermore,
even if $\phi$ were surjective, if it weren't also injective, say
$\phi\left(p_{0}\right)=\phi\left(p_{1}\right)$ for some distinct
$p_{0},p_{1}\in M$, then $T_{\phi\left(p_{0}\right)}^{*}N=T_{\phi\left(p_{1}\right)}^{*}N$,
and $T_{p_{0}}M\neq T_{p_{1}}M$, so the action on the fiber $T_{\phi\left(p_{0}\right)}^{*}N$
is not well-defined.

In the tensor field parlance, the cotangent map $T^{*}\phi$ simply
takes the form
\[
\left(\nnabla\phi\right)^{\left(1\,2\right)}\in\Gamma\left(T^{*}M\otimes_{M}\phi^{*}TN\right).
\]
The permutation superscript $\left(1\,2\right)$ is used here instead
of $*$ to distinguish it notationally from pullback notation, which
will be necessary in later calculations. The key concept is that the
tensor field $\left(\nnabla\phi\right)^{\left(1\,2\right)}$ encodes
the base map $\phi$; the basepoint $p\in M$ is part of the domain
$\phi^{*}T^{*}N$ itself.\\

The chain rule in the tensor field formalism makes use of the bundle
pullback. If $\psi\colon L\to M$ is smooth, then
\[
\nnabla^{L\to N}\left(\phi\circ\psi\right)=\psi^{*}\nnabla^{M\to N}\phi\cdot_{\psi^{*}TM}\nnabla^{L\to M}\psi.
\]
Because $\nnabla\psi\in\Gamma\left(\psi^{*}TM\otimes_{L}T^{*}L\right)$,
to form a well-defined natural pairing, the use of the pullback 
\[
\psi^{*}\nnabla\phi\in\Gamma\left(\psi^{*}\left(\phi^{*}TN\otimes_{M}T^{*}M\right)\right)=\Gamma\left(\psi^{*}\phi^{*}TN\otimes_{L}\psi^{*}T^{*}M\right)=\Gamma\left(\left(\phi\circ\psi\right)^{*}TN\otimes_{L}\psi^{*}T^{*}M\right)
\]
is necessary (instead of just $\nnabla\phi\in\Gamma\left(\phi^{*}TN\otimes_{M}T^{*}M\right)$).\\

Sometimes it is useful to discard some type information and write
$\nnabla\phi\in\Gamma_{\phi\times_{M}\Id_{M}}\left(TN\otimes_{N\times M}T^{*}M\right)$,
i.e. $\nnabla\phi\colon M\to TN\otimes_{N\times M}T^{*}M$ such that
$\left(\pi_{N}^{TN}\otimes_{N\times M}\pi_{M}^{T^{*}M}\right)\circ\nnabla\phi=\phi\times_{M}\Id_{M}$.
This is easily done by the canonical fiber projeciton available to
all pullback bundle constructions; $\phi^{*}TN\otimes_{M}T^{*}M\cong\left(\phi\times_{M}\Id_{M}\right)^{*}\left(TN\otimes_{N\times M}T^{*}M\right)$,
and the canonical fiber projection is
\[
\rho_{TN\otimes_{N\times M}T^{*}M}^{\left(\phi\times_{M}\Id_{M}\right)^{*}\left(TN\otimes_{N\times M}T^{*}M\right)}\colon\left(\phi\times_{M}\Id_{M}\right)^{*}\left(TN\otimes_{N\times M}T^{*}M\right)\to TN\otimes_{N\times M}T^{*}M,
\]
as defined in (\ref{cor:pullback-fiber-projection}). The granularity
of the type system should reflect the weight of the calculations being
performed. For demonstration of contrasting situations, see the discussion
at the beginning of Section \ref{sec:Linear-Covariant-Derivatives}
and the computation of the first variation in (\ref{thm:first_variation_of_L}).\\

It is important to have notation which makes the distinction between
the smooth vector bundle morphism formalism and the tensor field formalism,
because it may sometimes be necessary to mix the two, though this
paper will not need this. An added benefit to the tensor field formulation
of tangent maps is that certain notions regarding derivatives can
be conceptually and notationally combined, for example in Section
\ref{sec:Curvature-and-Commutation-of-Derivatives}.

\section{Linear Covariant Derivatives\label{sec:Linear-Covariant-Derivatives}}

As will be shown in the following discussion, a linear covariant derivative
(commonly referred to in the standard literature without the ``linear''
qualifier) provides a way to generalize the notion in elementary calculus
of the differential of a vector-valued function. The linear covariant
derivative interacts naturally with the notion of the pullback bundle,
and this interaction leads naturally to what could be called a covariant
derivative chain rule, which provides a crucial tool for the tensor
calculus computations seen later.\\

Let $V$ and $W$ be finite-dimensional vector spaces let $U\subseteq V$
be open, and let $\phi\colon U\to W$ be differentiable. Recall from
elementary calculus the differential $D\phi\colon U\to W\otimes V^{*}$
(essentially matrix-valued). There is no base map information encoded
in $D\phi$ (i.e. $\phi$ can't be recovered from $D\phi$ alone),
it contains only derivative information. The vector space structure
of $V$ and $W$ allows the trivializations $TU\cong V\rtriv U$ and
$TW\cong W\rtriv W$, where the first factors are the base spaces
and the second factors are the fibers (see (\ref{prop:trivial_bundle})).
The tangent map $\nnabla^{U\to W}\phi\colon U\to TW\otimes_{W\times U}T^{*}U$
(see Section \ref{sec:Tangent-Map-as-Tensor-Field}) has a codomain
that can be trivialized similarly;
\[
TW\otimes_{W\times U}T^{*}U\cong\left(W\rtriv W\right)\otimes_{W\times U}\left(V^{*}\rtriv U\right)\cong\left(W\otimes V^{*}\right)\rtriv\left(W\times U\right).
\]
Because $\left(W\otimes V^{*}\right)\rtriv\left(W\times U\right)$,
as a set, is a direct product, it can be decomposed into two factors.
Letting $\pr_{1}$ and $\pr_{2}$ be the projections onto the first
and second factors respectively,
\[
\pr_{1}\circ\nnabla\phi\colon U\to W\otimes V^{*}\mbox{ and }\pr_{2}\circ\nnabla\phi\colon U\to W\times U.
\]
The map $\pr_{2}\circ\nnabla\phi$ is the element of $\Gamma\left(W\rtriv U\right)$
identified with the base map $\phi$ itself; $\pr_{W}^{W\times U}\circ\pr_{2}\circ\nnabla\phi=\phi$.
This base map information is discarded in defining the differential
of $\phi$ as $D\phi:=\pr_{1}\circ\nnabla\phi$; the fiber portion
of $\nnabla\phi$. This construction relies critically on the natural
isomorphism $TW\cong W\rtriv W$ for a vector space $W$. 

An analogous construction shows that the differential $D\phi$ of
a map $\phi$ is well-defined even when its domain is a manifold.
However, when the codomain of a map $\phi$ is only a manifold, there
does not in general exist a natural trivialization of its tangent
bundle (in contrast to the vector space case), and therefore $D\phi$
can't be defined without additional structure. A linear covariant
derivative provides the missing structure.\\

For the remainder of this section, let $\pi\colon E\to N$ define
a smooth vector bundle having rank $r$.\\

A linear covariant derivative on $E$ provides a means of taking derivatives
of sections of $E$ (i.e. maps $\sigma\colon N\to E$ such that $\pi\circ\sigma=\Id_{N}$)
without passing to a higher tangent bundle as would happen under the
tangent map functor (i.e. if $\sigma\in\Gamma\left(E\right)$ then
$T\sigma\colon TN\to TE$ and $\nnabla^{N\to E}\sigma\colon N\to TE\otimes_{E\times N}T^{*}N$).
A linear covariant derivative provides an effective ``trivialization''
of $TE$ analogous to the trivialization $TW\cong W\rtriv W$ as discussed
above, discarding all but the ``fiber'' portion of the derivative,
allowing the construction of an object known as the total linear covariant
derivative analogous to the differential $D\phi$ as discussed above.

The notion of a linear covariant derivative on a vector bundle is
arguably the crucial element of differential geometry%
\footnote{The \emph{Fundamental Lemma of Riemannian Geometry} establishes the
existence of the Levi-Civita connection\citep[pg. 68]{RiemannianLee},
which is a linear covariant derivative satisfying certain naturality
properties.%
}. In particular, this operator implements the product rule property
common to anything that can be called a derivation -- a property which
is particularly conducive to the operation of tensor calculus. The
total linear covariant derivative of a vector field (i.e. section
of a vector bundle) allows the generalization of many constructions
in elementary calculus to the setting of smooth vector bundles equipped
with linear covariant derivatives. For example, the divergence $\Div X:=\tr DX$
of a vector field $X$ on $\mathbb{R}^{n}$ generalizes to the divergence
$\Div X:=\tr\nabla X$ of a vector field $X$ on $N$, which has an
analogous divergence theorem among other qualitative similarities.\\

\begin{rem}[Natural linear covariant derivative on trivial line bundle]
 \label{remark:covariant_derivative_on_trivial_line_bundle} Before
making the general definition for the linear covariant derivative,
a natural linear covariant derivative will be introduced. With $N$
denoting a smooth manifold as before, if $f\in C^{\infty}\left(N,\mathbb{R}\right)$,
then $df\in\Gamma\left(T^{*}N\right)$ is the \textbf{differential}
of $f$. Let 
\[
\lnabla^{N\to\mathbb{R}}f:=df.
\]
Because $C^{\infty}\left(N,\mathbb{R}\right)$ is naturally identified
with $\Gamma\left(\mathbb{R}\rtriv N\right)$, this is essentially
the natural linear covariant derivative on the trivial line bundle
$\mathbb{R}\rtriv N$. Note that there is an associated product rule;
if $f,g\in C^{\infty}\left(N,\mathbb{R}\right)$, then $fg\in C^{\infty}\left(N,\mathbb{R}\right)$,
and
\[
\lnabla^{N\to\mathbb{R}}\left(fg\right)=d\left(fg\right)=g\, df+f\, dg=g\lnabla^{N\to\mathbb{R}}f+f\lnabla^{N\to\mathbb{R}}g.
\]
When clear from context, the superscript decoration can be omitted
and the derivative denoted as $\lnabla f$. \end{rem}
\begin{defn}[Linear covariant derivative]
 \label{definition_linear_covariant_derivative} A \textbf{linear
covariant derivative} on a vector bundle defined by $\pi\colon E\to N$
is an $\mathbb{R}$-linear map $\lnabla^{E}\colon\Gamma\left(E\right)\to\Gamma\left(E\otimes_{N}T^{*}N\right)$
satisfying the product rule
\begin{equation}
\lnabla^{E}\left(f\otimes_{N}\sigma\right)=\sigma\otimes_{N}\lnabla^{N\to\mathbb{R}}f+f\otimes_{N}\lnabla^{E}\sigma,\label{eq:covariant_derivative_function_product_rule}
\end{equation}
where $f\in C^{\infty}\left(N,\mathbb{R}\right)$ and $\sigma\in\Gamma\left(E\right)$.
The switch in order in the first term of the expression is necessary
to form a tensor field of the correct type, $\Gamma\left(E\otimes_{N}T^{*}N\right)$.
If $\sigma\in\Gamma\left(E\right)$, then the expression $\lnabla^{E}\sigma$
is known as the \textbf{total {[}linear{]} covariant derivative }of
$\sigma$. If $\lnabla^{E}\sigma=0$ {[}in a subset $U\subseteq N${]},
then $\sigma$ is said to be \textbf{parallel} {[}on $U${]}. The
``linear'' qualifier is implied in standard literature and is therefore
often omitted.
\end{defn}
The inscribed $\shortmid$ in $\lnabla$ is to indicate that the covariant
derivative is linear, and can be omitted when clear from context,
or when it is unnecessary to distinguish it from the nonlinear tangent
map operator whose decorated symbol is $\nnabla$. For the remainder
of this section, this distinction will not be necessary, so an undecorated
$\nabla$ will be used.\\

For $V\in\Gamma\left(TN\right)$, it is customary to denote $\nabla^{E}\sigma\cdot V$
by $\nabla_{V}^{E}\sigma$, where $V$ indicates the ``directional''
component of the derivative. Following this convention, the product
rule can be written in a form where the product rule is more obvious;
\[
\nabla_{V}^{E}\left(f\otimes_{N}\sigma\right)=\nabla_{V}^{N\to\mathbb{R}}f\otimes_{N}\sigma+f\otimes_{N}\nabla_{V}^{E}\sigma.
\]

A covariant derivative is a local operator with respect to the base
space $N$; if $p\in N$, then $\left(\nabla^{E}\sigma\right)\left(p\right)$
depends only on the restriction of $\sigma$ to an arbitrarily small
neighborhood of $p$ \citep[pg. 50]{RiemannianLee}, and therefore
the restriction $\nabla^{E\mid_{U}}\colon\Gamma_{U}\left(E\right)\to\Gamma_{U}\left(E\otimes_{N}T^{*}N\right)$
makes sense, allowing calculations using local expressions. Furthermore,
a covariant derivative can be constructed locally and glued together
under certain conditions. See \citep[pg. 503]{JeffMLee} for more
on this, and as a reference for general theory on bundles, covariant
derivatives, and connections.\\

Linear covariant derivatives on several vector bundle constructions
will now be developed. In analogy to defining a linear map by its
action on a generating subset (e.g. a basis or a dense subspace) and
then extending using the linear structure, Lemma (\ref{lemma_covariant_derivative_construction})
allows a covariant derivative to be defined on a generating subset
(which can be chosen to make the defining expression particularly
natural) and then extending. In this case, the relevant space is the
space of sections of the vector bundle, which is a module over the
ring of smooth functions on a manifold, and the extension process
is done via linearity and the product rule (see (\ref{definition_linear_covariant_derivative})).
This approach will allow the local trivialization implementation details
to be hidden within the proof of Lemma (\ref{lemma_covariant_derivative_construction})
-- an example of information hiding -- so that constructions of covariant
derivatives can proceed clearly by focusing only on the natural properties
of the relevant objects and then invoking the lemma to do the ``dirty''
work (see (\ref{proposition_pullback_covariant_derivative}) and (\ref{proposition_induced_covariant_derivatives_on_sum_and_product_bundles})).\\

A bit of useful notation will be introduced to simplify the next definition.
If $G\subseteq\Gamma$ is a subset of a $C^{\infty}\left(N,\mathbb{R}\right)$-module
$\Gamma$ whose elements are functions on $N$ (and therefore have
a notion of restriction to a subset) and $U\subseteq N$ is open,
then let $G_{U}$ denote the set of restrictions of the elements of
$G$ to the set $U$. Note that $G_{U}\subseteq\Gamma_{U}$ by construction.
\begin{defn}[Finitely generating subset]
 Say that a subset of a module \textbf{finitely generates} the module
if the subset contains a finite set of generators for the module.
\end{defn}

\begin{defn}[Locally finitely generating subset]
 \label{definition_locally_finite_generating_subset} If $\Gamma$
is a $C^{\infty}\left(N,\mathbb{R}\right)$-module and $G\subseteq\Gamma$,
then $G$ is said to be a \textbf{locally finitely generating subset
of $\Gamma$} if each point $q\in N$ has a neighborhood $U\subseteq N$
for which $G_{U}$ finitely generates $\Gamma_{U}$.
\end{defn}
The space of sections of a vector bundle is the archetype for the
above definition. The locally trivial nature of $\pi\colon E\to N$
allows local frames to be chosen in a neighborhood of each point of
$N$, from which global smooth sections (though not necessarily a
global frame) can be made using a partition of unity subordinate to
the trivializing neighborhoods. The set of such global sections forms
a locally finite generating subset of $\Gamma\left(E\right)$.
\begin{lem}
\label{lemma_existence_of_local_frame} If $G$ is a locally finitely
generating subset of $\Gamma\left(E\right)$, then each point in $N$
has a neighborhood $U\subseteq N$ and $e_{1},\dots,e_{r}\in G_{U}$
such that $e_{1},\dots,e_{r}$ forms a frame for $\Gamma_{U}\left(E\right)$.
In other words, a local frame can be chosen out of $G$ near each
point in $N$. \end{lem}
\begin{proof}
Let $q\in N$ and let $V\subseteq N$ be a neighborhood of $q$ for
which $G_{V}=\left\{ g_{1},\dots,g_{\ell}\right\} $ finitely generates
$\Gamma_{V}\left(E\right)$ (here, $\ell\geq r$, recalling that $r=\rank E$).
Without loss of generality, let $g_{1}\left(q\right),\dots,g_{r}\left(q\right)$
be linearly independent (this is possible because $\left\{ g_{1}\left(q\right),\dots,g_{\ell}\left(q\right)\right\} $
spans the vector space $E_{q}$). Because $g_{i}$ is continuous for
each $i$ and the linear independence of the sections $g_{1},\dots,g_{r}$
is an open condition (defined by $L^{-1}\left(\mathbb{R}\backslash\left\{ 0\right\} \right)$
where $L\colon N\to\bigwedge^{r}E_{q},\, p\mapsto g_{1}\left(p\right)\wedge\dots\wedge g_{r}\left(p\right)$),
there is a neighborhood $U\subseteq V$ of $q$ for which $\left\{ g_{1}\left(p\right),\dots,g_{r}\left(p\right)\right\} $
is a linearly independent set for each $p\in U$. Finally, letting
$e_{i}:=g_{i}\mid_{U}$ for $i\in\left\{ 1,\dots,r\right\} $, the
sections $e_{1},\dots,e_{r}\in G_{U}$ form a frame for $\Gamma_{U}\left(E\right)$. 
\end{proof}
The following lemma shows that defining a covariant derivative on
a locally finitely generating subset of the space of sections of a
vector bundle is sufficient to uniquely define a covariant derivative
on the whole space. The particular generating subset can be chosen
so the covariant derivative has a particularly natural expression
within that subset.
\begin{lem}[Linear covariant derivative construction]
 \label{lemma_covariant_derivative_construction} Let $G$ be a locally
finite generating subset of $\Gamma\left(E\right)$. If $\nabla^{G}\colon G\to\Gamma\left(E\otimes_{N}T^{*}N\right)$
satisfies the linear covariant derivative axioms%
\footnote{What is meant by this is that the product rule must only be satisfied
on $\lambda\otimes_{N}g$ if $\lambda g\in G$, where $\lambda\in C^{\infty}\left(N,\mathbb{R}\right)$
and $g\in G$.%
}, then there is a unique linear covariant derivative $\nabla^{E}\colon\Gamma\left(E\right)\to\Gamma\left(E\otimes_{N}T^{*}N\right)$
whose restriction to $G$ is $\nabla^{G}$. \end{lem}
\begin{proof}
If $q\in N$, then by (\ref{lemma_existence_of_local_frame}) there
exists a neighborhood $U\subseteq N$ of $q$ for which there are
$e_{1},\dots,e_{r}\in G_{U}$ forming a frame for $E\mid_{U}$. If
$\sigma\in\Gamma\left(E\right)$, then $\sigma\mid_{U}=\sigma^{i}e_{i}$
for some $\sigma^{1},\dots,\sigma^{r}\in C^{\infty}\left(U,\mathbb{R}\right)$
(specifically, $\sigma^{i}=e^{i}\cdot_{E}\sigma\mid_{U}$, where $e^{1},\dots,e^{r}\in\Gamma_{U}\left(E^{*}\right)$
denotes the dual coframe of $e_{1},\dots,e_{r}$). Define $\nabla^{E}\colon\Gamma\left(E\right)\to\Gamma\left(E\otimes_{N}T^{*}N\right)$
locally on $\Gamma_{U}\left(E\right)$ so as to satisfy the product
rule
\[
\nabla^{E}\left(\sigma\mid_{U}\right):=e_{i}\otimes_{N}\nabla^{N\to\mathbb{R}}\sigma^{i}+\sigma^{i}\otimes_{N}\nabla^{G}e_{i}.
\]
To show well-definedness, let $f_{1},\dots,f_{r}\in G_{U}$ be another
frame for $E\mid_{U}$. Then $\sigma=\tau^{i}f_{i}$ for some $\tau^{1},\dots,\tau^{r}\in C^{\infty}\left(U,\mathbb{R}\right)$.
Let $\Psi\colon\Gamma_{U}\left(E\right)\to\Gamma_{U}\left(E\right)$
be the unique smooth vector bundle isomorphism such that $f_{i}=\Psi\cdot_{E}e_{i}$.
Writing $\Psi$ and $\Psi^{-1}$ with respect to the frame $\left(e_{i}\right)$
as $\Psi_{j}^{i}e_{i}\otimes e^{j}$ and $\left(\Psi^{-1}\right)_{j}^{i}e_{i}\otimes e^{j}$
respectively, it follows that $f_{i}=\Psi_{i}^{j}e_{j}$ and $\tau^{i}=\sigma^{j}\left(\Psi^{-1}\right)_{j}^{i}$.
Then
\begin{align*}
\nabla^{E}\left(\tau^{i}f_{i}\right)={} & f_{i}\otimes_{N}\nabla^{N\to\mathbb{R}}\tau^{i}+\tau^{i}\otimes_{N}\nabla^{G}f_{i}\\
={} & \Psi_{i}^{j}e_{j}\otimes_{N}\nabla^{N\to\mathbb{R}}\left(\sigma^{k}\left(\Psi^{-1}\right)_{k}^{i}\right)+\sigma^{j}\left(\Psi^{-1}\right)_{j}^{i}\otimes_{N}\nabla^{G}\left(\Psi_{i}^{k}e_{k}\right)\\
={} & \Psi_{i}^{j}e_{j}\left(\Psi^{-1}\right)_{k}^{i}\otimes_{N}\nabla^{N\to\mathbb{R}}\sigma^{k}+\Psi_{i}^{j}e_{j}\sigma^{k}\otimes_{N}\nabla^{N\times\mathbb{R}}\left(\Psi^{-1}\right)_{k}^{i}\\
{} & +\sigma^{j}\left(\Psi^{-1}\right)_{j}^{i}e_{k}\otimes_{N}\nabla^{N\to\mathbb{R}}\Psi_{i}^{k}+\sigma^{j}\left(\Psi^{-1}\right)_{j}^{i}\Psi_{i}^{k}\otimes\nabla^{G}e_{k}\\
={} & \delta_{k}^{j}e_{j}\otimes_{N}\nabla^{N\to\mathbb{R}}\sigma^{k}+\sigma^{j}\delta_{j}^{k}\otimes\nabla^{G}e_{k}+\sigma^{\ell}e_{k}\otimes_{N}\nabla^{N\to\mathbb{R}}\left(\Psi_{i}^{k}\left(\Psi^{-1}\right)_{\ell}^{i}\right)\\
={} & \nabla^{E}\left(\sigma^{i}e_{i}\right).
\end{align*}
The last equality follows because $\Psi_{i}^{k}\left(\Psi^{-1}\right)_{\ell}^{i}=\delta_{\ell}^{k}$,
which is a constant function, so $\nabla^{N\to\mathbb{R}}\left(\Psi_{i}^{k}\left(\Psi^{-1}\right)_{\ell}^{i}\right)=0$.
Thus the expression defining $\nabla^{E}$ doesn't depend on the choice
of local frame. This establishes the well-definedness of $\nabla^{E}$.

Clearly the restriction of $\nabla^{E}$ to $G$ is $\nabla^{G}$.
This establishes the claim of existence. Uniqueness follows from the
fact that $\nabla^{E}$ is defined in terms of the maps $\nabla^{N\to\mathbb{R}}$
and $\nabla^{G}$. 
\end{proof}
Lemma (\ref{lemma_covariant_derivative_construction}) is used in
the proof of the following proposition to allow a natural formulation
of the pullback covariant derivative with respect to a natural locally
finite generating subset of $\Gamma\left(\phi^{*}E\right)$, in which
the relevant derivative has a natural chain rule.
\begin{prop}[Pullback covariant derivative]
 \label{proposition_pullback_covariant_derivative} If $\phi\colon M\to N$
is smooth and $\nabla^{E}$ is a covariant derivative on $E$, then
there is a unique covariant derivative $\nabla^{\phi^{*}E}$ on $\phi^{*}E$
satisfying the chain rule
\[
\nabla^{\phi^{*}E}\phi^{*}e=\phi^{*}\nabla^{E}e\cdot_{\phi^{*}TN}\nnabla^{M\to N}\phi
\]
for all $e\in\Gamma\left(E\right)$. \end{prop}
\begin{proof}
Let $G:=\left\{ \sigma\in\Gamma\left(\phi^{*}E\right)\mid\sigma=\phi^{*}e\mbox{ for some }e\in\Gamma\left(E\right)\right\} $,
noting that a local frame $e_{1},\dots,e_{\rank E}\in\Gamma_{U}\left(E\right)$
over open set $U\subseteq N$ induces a local frame $\phi^{*}e_{1},\dots,\phi^{*}e_{\rank E}\in\Gamma_{\phi^{-1}\left(U\right)}\left(\phi^{*}E\right)$,
so $G$ is a locally finite generating subset of $\Gamma\left(\phi^{*}E\right)$.
Define
\begin{eqnarray*}
\nabla^{G}\colon G & \to & \Gamma\left(\phi^{*}E\otimes_{M}T^{*}N\right),\\
\phi^{*}e & \mapsto & \phi^{*}\nabla^{E}e\cdot_{\phi^{*}TN}\nnabla^{M\to N}\phi.
\end{eqnarray*}
The well-definedness and $\mathbb{R}$-linearity of $\nabla^{G}$
comes from that of $\nabla^{E}$. For the product rule, if $\lambda\in C^{\infty}\left(M,\mathbb{R}\right)$
and $e\in\Gamma\left(E\right)$, then the product $\lambda\otimes_{M}\phi^{*}e$
is an element of $G$ if and only if $\lambda=\phi^{*}\mu$ for some
$\mu\in C^{\infty}\left(N,\mathbb{R}\right)$, in which case, $\lambda\otimes_{M}\phi^{*}e$
$=\phi^{*}\mu\otimes_{M}\phi^{*}e$ $=\phi^{*}\left(\mu\otimes_{N}e\right)$.
Then it follows that
\begin{align*}
\nabla^{G}\left(\lambda\otimes_{M}\phi^{*}e\right) & =\nabla^{G}\phi^{*}\left(\mu\otimes_{N}e\right)\\
 & =\phi^{*}\nabla^{E}\left(\mu\otimes_{N}e\right)\cdot_{\phi^{*}TN}\nnabla^{M\to N}\phi\\
 & =\phi^{*}\left(e\otimes_{N}\nabla^{N\to\mathbb{R}}\mu+\mu\otimes_{N}\nabla^{E}e\right)\cdot_{\phi^{*}TN}\nnabla^{M\to N}\phi\\
 & =\phi^{*}\left(e\otimes_{N}\nabla^{N\to\mathbb{R}}\mu\right)\cdot_{\phi^{*}TN}\nnabla^{M\to N}\phi+\phi^{*}\left(\mu\otimes_{N}\nabla^{E}e\right)\cdot_{\phi^{*}TN}\nnabla^{M\to N}\phi\\
 & =\phi^{*}e\otimes_{M}\left(\phi^{*}\nabla^{N\to\mathbb{R}}\mu\cdot_{\phi^{*}TN}\nnabla^{M\to N}\phi\right)+\phi^{*}\mu\otimes_{M}\left(\phi^{*}\nabla^{E}e\cdot_{\phi^{*}TN}\nnabla^{M\to N}\phi\right)\\
 & =\phi^{*}e\otimes_{M}\nabla^{M\to\mathbb{R}}\phi^{*}\mu+\phi^{*}\mu\otimes_{M}\nabla^{G}\phi^{*}e\\
 & =\phi^{*}e\otimes_{M}\nabla^{M\to\mathbb{R}}\lambda+\lambda\otimes_{M}\nabla^{G}\phi^{*}e,
\end{align*}
which is exactly the required product rule. By (\ref{lemma_covariant_derivative_construction}),
there exists a unique covariant derivative $\nabla^{\phi^{*}E}$ on
$\phi^{*}E$ whose restriction to $G$ is $\nabla^{G}$. 
\end{proof}
The full notation $\nabla^{\phi^{*}E}$ is often cumbersome, so it
may be denoted by $\nabla^{\phi}$ when the pulled-back bundle is
clear from context.
\begin{rem}
\label{remark:pullback_derivative_captures_fiber_variation}There
is an important feature of a pullback covariant derivative in the
case that pullback map is not an immersion; the pullback covariant
derivative may be nonzero even where the pullback map is singular.
This fact can be obscured by a certain abuse of notation which often
comes in the expression of the geodesic equations in differential
geometry (see (\ref{example:geodesic_equation})). An example will
illustrate this point. 

\begin{figure}
\centering
\def\svgwidth{\columnwidth}
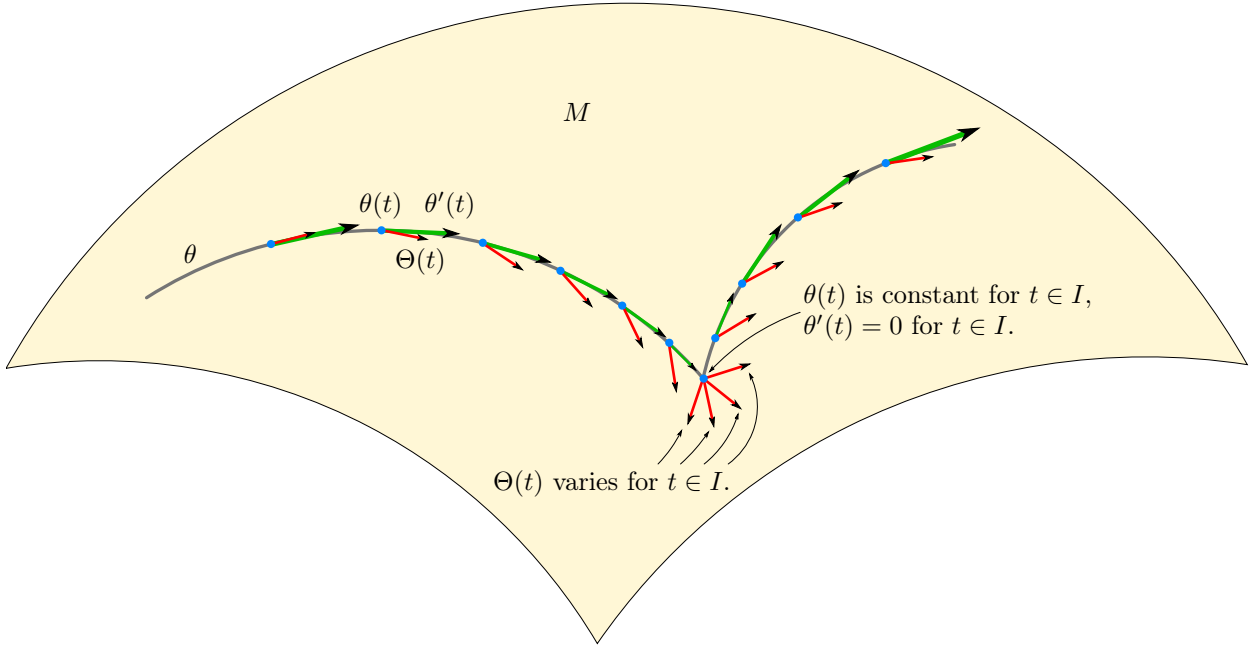

\caption{\label{fig:covariant-deriv-along-singular-path} A picture of the
manifold $M$, path $\theta$, and vector fields $\theta^{\prime}$
and $\Theta$. The blue dots represent $\theta\left(t\right)$ at
certain points $t\in\mathbb{R}$, while the green and red arrows represent
$\theta^{\prime}\left(t\right)$ and $\Theta\left(t\right)$ at at
these points respectively. Note that $\Theta$ is a unit-length vector
field along $\theta$ and varies within $I$, whereas $\theta^{\prime}$
is a vector field along $\theta$ that vanishes within $I$.}

\end{figure}

Let $\nabla^{TM}$ be a covariant derivative on $\pi_{M}^{TM}\colon TM\to M$.
Let $\Theta\colon\mathbb{R}\to TM$ be a unit-length vector field
which describes the location of a person (the basepoint) and direction
s/he is looking (the fiber portion) with respect to time (let $\mathbb{R}$
have standard coordinate $t$). Define $\theta\colon\mathbb{R}\to M$
by $\theta:=\pi_{M}^{TM}\circ\Theta$, so that $\theta$ is the base
map of $\Theta$, i.e. $\theta$ has discarded the direction information
and only encodes the location information. Say that for some closed
interval $I\subseteq\mathbb{R}$, $\frac{d\theta}{dt}\mid_{I}$ is
identically zero (and so is not an immersion), but that $\frac{d\Theta}{dt}\mid_{I}$
is nonvanishing; see Figure \ref{fig:covariant-deriv-along-singular-path}.
Mathematically, this means that during this time, $\Theta$ is varying
only within a single fiber of $TM$. Physically, this means that during
this time, the person is standing still but the direction s/he is
looking is changing. Passing to a higher tangent space is often undesirable
(note that $\frac{d\Theta}{dt}$ takes values in $TTM$), so to avoid
this, a covariant derivative is used. In order to be meaningful, the
covariant derivative must capture this fiber-only variation.

Because $\Theta$ is a vector field along $\theta$, it can be written
as $\Theta\in\Gamma\left(\theta^{*}TM\right)$, and the covariant
derivative on $TM$ induces a pullback covariant derivative on $\theta^{*}TM$,
which has base space $\mathbb{R}$. In other words, $\theta^{*}TM$
is parameterized by time. Then $\nabla_{\frac{d}{dt}}^{\theta^{*}TM}\Theta\in\Gamma\left(\theta^{*}TM\right)$
is the desired covariant derivative of $\Theta$ with respect to time.
A coordinate-based calculation will be made to make completely obvious
why this pullback covariant derivative captures the desired information.
Let $\left(x^{i}\right)$ be local coordinates on $M$ and, for simplicity,
assume that the image of $\theta$ lies entirely within this coordinate
chart. Because $\left(\partial_{i}\right)$ is a local frame for $TM$,
$\left(\theta^{*}\partial_{i}\right)$ is a local frame for $\theta^{*}TM$,
by (\ref{lemma_representation_of_sections_of_a_pullback_bundle})
and $\Theta\in\Gamma\left(\theta^{*}TM\right)$ can be written locally
as $\Theta\left(t\right)=\Theta^{i}\left(t\right)\,\left(\theta^{*}\partial_{i}\right)\left(t\right)$
for some functions $\left(\Theta^{i}\colon\mathbb{R}\to\mathbb{R}\right)$.
Then
\begin{align*}
\nabla_{\frac{d}{dt}}^{\theta^{*}TM}\Theta & =\nabla_{\frac{d}{dt}}^{\theta^{*}TM}\left(\Theta^{i}\,\theta^{*}\partial_{i}\right)\\
 & =\left(\nabla_{\frac{d}{dt}}\Theta^{i}\right)\theta^{*}\partial_{i}+\Theta^{i}\nabla_{\frac{d}{dt}}^{\theta^{*}TM}\theta^{*}\partial_{i}\\
 & =\frac{d\Theta^{i}}{dt}\,\theta^{*}\partial_{i}+\Theta^{i}\,\theta^{*}\nabla^{TM}\partial_{i}\cdot_{\theta^{*}TM}\nnabla^{\mathbb{R}\to M}\theta\cdot_{T\mathbb{R}}\frac{d}{dt}\\
 & =\frac{d\Theta^{i}}{dt}\,\theta^{*}\partial_{i}+\Theta^{i}\,\theta^{*}\nabla^{TM}\partial_{i}\cdot_{\theta^{*}TM}\frac{d\theta}{dt}.
\end{align*}
Note that $\nnabla^{\mathbb{R}\to M}\theta\in\Gamma\left(\theta^{*}TM\right)$.
Within the interval $I$, $\frac{d\theta}{dt}$ vanishes, so the second
term vanishes on $I$. However, because $\Theta$ is varying in a
fiber-only direction within $I$, the basepoint is not changing and
$\frac{d\Theta^{i}}{dt}\theta^{*}\partial_{i}$ can be identified
with an elementary vector space derivative (the fiber is a vector
space and so an elementary derivative is well-defined there). This
fiber-direction derivative is nonvanishing by assumption, so $\nabla_{\frac{d}{dt}}^{\theta^{*}TM}\Theta$
is nonvanishing on $I$ as desired.\\

\end{rem}
Introducing a bit of natural notation which will be helpful for the
next result, if $X\in\Gamma\left(E\right)$ and $Y\in\Gamma\left(F\right)$,
then define $X\oplus Y\equiv X\oplus_{M\times N}Y\in\Gamma\left(E\oplus_{M\times N}F\right)$
and $X\otimes Y\equiv X\otimes_{M\times N}Y\in\Gamma\left(E\otimes_{M\times N}F\right)$
by
\[
\left(X\oplus_{M\times N}Y\right)\left(p,q\right):=X\left(p\right)\oplus Y\left(q\right)\qquad\mbox{and}\qquad\left(X\otimes_{M\times N}Y\right)\left(p,q\right):=X\left(p\right)\otimes Y\left(q\right)
\]
for each $\left(p,q\right)\in M\times N$.
\begin{prop}[Induced covariant derivatives on $E\oplus_{M\times N}F$ and $E\otimes_{M\times N}F$]
 \label{proposition_induced_covariant_derivatives_on_sum_and_product_bundles}
If $\nabla^{E}$ and $\nabla^{F}$ are covariant derivatives on $E$
and $F$ respectively, then there are unique covariant derivatives
\[
\nabla^{E\oplus_{M\times N}F}\colon\Gamma\left(E\oplus_{M\times N}F\right)\to\Gamma\left(\left(E\oplus_{M\times N}F\right)\otimes_{M\times N}\left(T^{*}M\oplus_{M\times N}T^{*}N\right)\right)
\]
and 
\[
\nabla^{E\otimes_{M\times N}F}\colon\Gamma\left(E\otimes_{M\times N}F\right)\to\Gamma\left(\left(E\otimes_{M\times N}F\right)\otimes_{M\times N}\left(T^{*}M\oplus_{M\times N}T^{*}N\right)\right)
\]
on $E\oplus F$ and $E\otimes F$ respectively, satisfying the sum
rule
\[
\nabla_{u\oplus v}^{E\oplus F}\left(X\oplus Y\right)=\nabla_{u}^{E}X\oplus\nabla_{v}^{F}Y
\]
and the product rule 
\[
\nabla_{u\oplus v}^{E\otimes F}\left(X\otimes Y\right)=\nabla_{u}^{E}X\otimes Y+X\otimes\nabla_{v}^{F}Y,
\]
respectively, where $X\in\Gamma\left(E\right)$, $Y\in\Gamma\left(F\right)$,
and $u\oplus v\in TM\oplus TN$. Here, $TM\oplus TN\to M\times N$
(and its dual) is used instead of the isomorphic vector bundle $T\left(M\times N\right)\to M\times N$
(and its dual). \end{prop}
\begin{proof}
Suppressing the pedantic use of the $M\times N$ subscript to avoid
unnecessary notational overload, the set $G:=\left\{ e\oplus f\mid e\in\Gamma\left(E\right),\, f\in\Gamma\left(F\right)\right\} $
is a locally finite generator of $\Gamma\left(E\oplus F\right)$,
since local frames for $E\oplus F$ take the form $\left\{ e_{i}\oplus0,0\oplus f_{j}\right\} $,
where $\left\{ e_{i}\right\} $ and $\left\{ f_{j}\right\} $ are
local frames for $E$ and $F$ respectively. Define 
\begin{eqnarray*}
\nabla^{G}\colon G & \to & \Gamma\left(\left(E\oplus F\right)\otimes\left(T^{*}M\oplus T^{*}N\right)\right),\\
X\oplus Y & \mapsto & \left(u\oplus v\mapsto\nabla_{u}^{E}X\oplus\nabla_{v}^{F}Y\right),\mbox{ where }u\oplus v\in TM\oplus TN.
\end{eqnarray*}
This map is well-defined and $\mathbb{R}$-linear by construction,
since the connections $\nabla^{E}$ and $\nabla^{F}$ are well-defined
and $\mathbb{R}$-linear. If $\lambda\in C^{\infty}\left(M\times N,\mathbb{R}\right)$,
$X\in\Gamma\left(E\right)$, and $Y\in\Gamma\left(F\right)$, then
the product $\lambda\otimes\left(X\oplus Y\right)$ is in $G$ (i.e.
has the form $\overline{X}\oplus\overline{Y}$ for some $\overline{X}\in\Gamma\left(E\right)$
and $\overline{Y}\in\Gamma\left(F\right)$) if and only if $\lambda$
is constant. Thus the product rule (restricted to elements of $G$)
reduces to $\mathbb{R}$-linearity, which is already satisfied. By
(\ref{lemma_covariant_derivative_construction}), there exists a unique
connection $\nabla^{E\oplus F}$ on $E\oplus F$ whose restriction
to $G$ is $\nabla^{G}$.

Similarly, the set $H:=\left\{ e\otimes f\mid e\in\Gamma\left(E\right),\, f\in\Gamma\left(F\right)\right\} $
is a locally finite generator of $\Gamma\left(E\otimes F\right)$,
since local frames for $E\otimes F$ take the form $\left\{ e_{i}\otimes f_{j}\right\} $,
where $\left\{ e_{i}\right\} $ and $\left\{ f_{j}\right\} $ are
local frames for $E$ and $F$ respectively. Define
\begin{eqnarray*}
\nabla^{H}\colon H & \to & \Gamma\left(\left(E\otimes F\right)\otimes\left(T^{*}M\oplus T^{*}N\right)\right),\\
X\otimes Y & \mapsto & \left(u\oplus v\mapsto\nabla_{u}^{E}X\otimes Y+X\otimes\nabla_{v}^{F}Y\right),\mbox{ where }u\oplus v\in TM\oplus TN.
\end{eqnarray*}
This map is well-defined and $\mathbb{R}$-linear by construction,
since the connections $\nabla^{E}$ and $\nabla^{F}$ are well-defined
and $\mathbb{R}$-linear. For the product rule, with $\lambda\in C^{\infty}\left(M\times N,\mathbb{R}\right)$,
$X\in\Gamma\left(E\right)$, and $Y\in\Gamma\left(F\right)$, the
product $\lambda\otimes\left(X\otimes Y\right)$ is in $H$ if and
only if there exist $\mu\in C^{\infty}\left(M,\mathbb{R}\right)$
and $\nu\in C^{\infty}\left(N,\mathbb{R}\right)$ such that $\lambda=\mu\otimes\nu\in\left(\mathbb{R}\rtriv M\right)\otimes\left(\mathbb{R}\rtriv N\right)$
(noting that then $\lambda\otimes_{M\times N}\left(X\otimes Y\right)$
$=\left(\mu\otimes\nu\right)\otimes_{M\times N}\left(X\otimes Y\right)$
$=\left(\mu\otimes_{M}X\right)\otimes\left(\nu\otimes_{N}Y\right)$).
In this case, with $u\oplus v\in TM\oplus TN$,
\begin{align*}
 & \nabla_{u\oplus v}^{H}\left(\lambda\otimes_{M\times N}\left(X\otimes Y\right)\right)\\
={} & \nabla_{u\oplus v}^{H}\left(\left(\mu\otimes\nu\right)\otimes_{M\times N}\left(X\otimes Y\right)\right)\\
={} & \nabla_{u\oplus v}^{H}\left(\left(\mu\otimes_{M}X\right)\otimes\left(\nu\otimes_{N}Y\right)\right)\\
={} & \nabla_{u}^{E}\left(\mu\otimes_{M}X\right)\otimes\left(\nu\otimes_{N}Y\right)+\left(\mu\otimes_{M}X\right)\otimes\nabla_{v}^{F}\left(\nu\otimes_{N}Y\right)\\
={} & \left(\nabla_{u}^{M\to\mathbb{R}}\mu\otimes_{M}X\right)\otimes\left(\nu\otimes_{N}Y\right)+\left(\mu\otimes_{M}\nabla_{u}^{E}X\right)\otimes\left(\nu\otimes_{N}Y\right)\\
 & +\left(\mu\otimes_{M}X\right)\otimes\left(\nabla_{v}^{N\to\mathbb{R}}\nu\otimes_{N}Y\right)+\left(\mu\otimes_{M}X\right)\otimes\left(\nu\otimes_{N}\nabla_{v}^{F}Y\right)\\
={} & \left(\nabla_{u}^{M\to\mathbb{R}}\mu\otimes\nu+\mu\otimes\nabla_{v}^{N\to\mathbb{R}}\nu\right)\otimes_{M\times N}\left(X\otimes Y\right)+\lambda\otimes_{M\times N}\left(\nabla_{u}^{E}X\otimes Y+X\otimes\nabla_{v}^{F}Y\right)\\
={} & \nabla_{u\oplus v}^{M\times N\to\mathbb{R}}\lambda\otimes_{M\times N}\left(X\otimes Y\right)+\lambda\otimes_{M\times N}\nabla_{u\oplus v}^{H}\left(X\otimes Y\right),
\end{align*}
which is exactly the required product rule. By (\ref{lemma_covariant_derivative_construction}),
there exists a unique connection $\nabla^{E\otimes F}$ on $E\otimes F$
whose restriction to $H$ is $\nabla^{H}$. \end{proof}
\begin{rem}[Naturality of the covariant derivatives on $E\oplus_{M\times N}F$
and $E\otimes_{M\times N}F$]
 Letting $\pr_{i}:=\pr_{i}^{M\times N}$ ($i\in\left\{ 1,2\right\} $)
for brevity, the maps
\begin{eqnarray*}
\xi\colon E\oplus_{M\times N}F & \to & \pr_{1}^{*}E\oplus_{M\times N}\pr_{2}^{*}F,\\
e\oplus f & \mapsto & \left(\left(\pi^{E}\oplus\pi^{F}\right)\left(e\oplus f\right),e\right)\oplus_{M\times N}\left(\left(\pi^{E}\oplus\pi^{F}\right)\left(e\oplus f\right),f\right)
\end{eqnarray*}
and 
\begin{eqnarray*}
\psi\colon E\otimes_{M\times N}F & \to & \pr_{1}^{*}E\otimes_{M\times N}\pr_{2}^{*}F,\\
e\otimes f & \mapsto & \left(\left(\pi^{E}\otimes\pi^{F}\right)\left(e\otimes f\right),e\right)\otimes_{M\times N}\left(\left(\pi^{E}\otimes\pi^{F}\right)\left(e\otimes f\right),f\right),
\end{eqnarray*}
each extended linearly to the rest of their domains, are easily shown
to be smooth vector bundle isomorphisms over $\Id_{M\times N}$. Then
\[
\nabla_{z}^{E\oplus F}\left(X\oplus Y\right)=\xi^{-1}\left(\nabla_{z}^{\pr_{1}^{*}E\oplus_{M\times N}\pr_{2}^{*}F}\xi\left(X\oplus Y\right)\right)
\]
and 
\[
\nabla_{z}^{E\otimes F}\left(X\otimes Y\right)=\psi^{-1}\left(\nabla_{z}^{\pr_{1}^{*}E\otimes_{M\times N}\pr_{2}^{*}F}\psi\left(X\otimes Y\right)\right)
\]
for all $X\in\Gamma\left(E\right)$, $Y\in\Gamma\left(F\right)$,
and $z\in T\left(M\times N\right)$, showing that the connections
on $E\oplus F$ and $E\otimes F$ are $\xi$ and $\psi$-related to
the naturally induced connections on $\pr_{1}^{*}E\oplus\pr_{2}^{*}F$
and $\pr_{1}^{*}E\otimes_{M\times N}\pr_{2}^{*}F$ respectively, and
are therefore in this sense natural. The sum $X\oplus Y\in\Gamma\left(E\oplus F\right)$
and product $X\otimes Y\in\Gamma\left(E\otimes F\right)$ correspond
to $\pr_{1}^{*}X\oplus_{M\times N}\pr_{2}^{*}Y$ and $\pr_{1}^{*}X\otimes_{M\times N}\pr_{2}^{*}Y\in\Gamma\left(\pr_{1}^{*}E\otimes_{M\times N}\pr_{2}^{*}F\right)$
under $\xi$mathn$\psi$ respectively. \\

Many important tensor constructions involve permutations. An extremely
useful property of these permutations is that they commute with the
covariant derivatives induced by the covariant derivatives on the
tensor bundle factors, making them natural operators in the setting
of covariant tensor calculus.\end{rem}
\begin{prop}[Transposition tensor fields are parallel]
 \label{prop:transposition_tensor_fields_are_parallel} Let $E_{1},E_{2},E_{3},E_{4}$
be smooth vector bundles over $M$ having covariant derivatives $\nabla^{E_{1}},\nabla^{E_{2}},\nabla^{E_{3}},\nabla^{E_{4}}$
respectively, let $A:=E_{1}\otimes_{M}E_{2}\otimes_{M}E_{3}\otimes_{M}E_{4}$
and $B:=E_{1}\otimes_{M}E_{3}\otimes_{M}E_{2}\otimes_{M}E_{4}$, and
let $\nabla^{A}$ and $\nabla^{B}$ denote the induced covariant derivatives. 

If $\left(2\,3\right)\in\Gamma\left(A^{*}\otimes_{M}B\right)$ denotes
the tensor field which maps $e_{1}\otimes_{M}\otimes e_{2}\otimes_{M}e_{3}\otimes_{M}e_{4}\in A$
to $e_{1}\otimes_{M}e_{3}\otimes_{M}e_{2}\otimes_{M}e_{4}\in B$ (i.e.
$\left(2\,3\right)$ transposes the second and third factors), then
$\left(2\,3\right)$ is a parallel tensor field with respect to the
covariant derivative induced on the vector bundle $A^{*}\otimes_{M}B\to M$,
i.e. $\nabla^{A^{*}\otimes_{M}B}\left(2\,3\right)=0$. \end{prop}
\begin{proof}
Let $X\in\Gamma\left(TM\right)$. Then
\begin{align*}
 & \left(e_{1}\otimes_{M}e_{2}\otimes_{M}e_{3}\otimes_{M}e_{4}\right)\cdot_{A^{*}}\nabla_{X}^{A^{*}\otimes_{M}B}\left(2\,3\right)\\
={} & \nabla_{X}^{B}\left(\left(e_{1}\otimes_{M}e_{2}\otimes_{M}e_{3}\otimes_{M}e_{4}\right)\cdot_{A^{*}}\left(2\,3\right)\right)-\nabla_{X}^{A}\left(e_{1}\otimes_{M}e_{2}\otimes_{M}e_{3}\otimes_{M}e_{4}\right)\cdot_{A^{*}}\left(2\,3\right)\\
={} & \nabla_{X}^{B}\left(e_{1}\otimes_{M}e_{3}\otimes_{M}e_{3}\otimes_{M}e_{4}\right)\\
 & -\nabla_{X}^{E_{1}}e_{1}\otimes_{M}e_{3}\otimes_{M}e_{2}\otimes_{M}e_{4}\\
 & -e_{1}\otimes_{M}\nabla_{X}^{E_{3}}e_{3}\otimes_{M}e_{2}\otimes_{M}e_{4}\\
 & -e_{1}\otimes_{M}e_{3}\otimes_{M}\nabla_{X}^{E_{2}}e_{2}\otimes_{M}e_{4}\\
 & -e_{1}\otimes_{M}e_{3}\otimes_{M}e_{2}\otimes_{M}\nabla_{X}^{E_{4}}e_{4}\\
={} & \nabla_{X}^{B}\left(e_{1}\otimes_{M}e_{3}\otimes_{M}e_{3}\otimes_{M}e_{4}\right)-\nabla_{X}^{B}\left(e_{1}\otimes_{M}e_{3}\otimes_{M}e_{3}\otimes_{M}e_{4}\right)\\
={} & 0.
\end{align*}
Because $X$ is arbitrary, this shows that $\left(e_{1}\otimes_{M}e_{2}\otimes_{M}e_{3}\otimes_{M}e_{4}\right)\cdot_{A^{*}}\nabla^{A^{*}\otimes_{M}B}\left(2\,3\right)=0$.
This extends linearly to general tensors, so $\nabla^{A^{*}\otimes_{M}B}\left(2\,3\right)=0$,
as desired.
\end{proof}
The fact that all transposition tensor fields are parallel implies
that all permutation tensor fields are parallel, since every permutation
is just the product of transpositions. This gives as an easy corollary
that a covariant derivative operation commutes with a permutation
operation, which has quite a succinct statement using the permutation
superscript notation.
\begin{cor}[Permutation tensor fields are parallel]
 \label{cor:permutation_tensor_fields_are_parallel} Let $E_{1},\dots,E_{k}$
be smooth vector bundles over $M$ each having a covariant derivative,
and let $A:=E_{1}\otimes_{M}\dots\otimes_{M}E_{k}$ and $B:=E_{\sigma^{-1}\left(1\right)}\otimes_{M}\dots\otimes_{M}E_{\sigma^{-1}\left(k\right)}$.
If $\sigma\in S_{k}$ is interpreted as the tensor field in $\Gamma\left(A^{*}\otimes_{M}B\right)$
which maps $e_{1}\otimes_{M}\dots\otimes_{M}e_{k}$ to $e_{\sigma^{-1}\left(1\right)}\otimes_{M}\dots\otimes_{M}e_{\sigma^{-1}\left(k\right)}$,
then $\sigma$ is a parallel tensor field. Stated using the superscript
notation, with $X\in\Gamma\left(TM\right)$ and $a\in\Gamma\left(A\right)$,
\[
\nabla_{X}^{B}a^{\sigma}=\left(\nabla_{X}^{A}a\right)^{\sigma}.
\]
\end{cor}
\begin{proof}
This follows from the fact that $\sigma$ can be written as the product
of transpositions; $\nabla_{X}\sigma=0$ because of the product rule
and because each transposition is parallel. The claim regarding commutation
with the superscript permutation follows easily from its definition.
\[
\nabla_{X}^{B}a^{\sigma}=\nabla_{X}^{B}\left(a\cdot_{A^{*}}\sigma\right)=a\cdot_{A^{*}}\nabla_{X}^{A^{*}\otimes_{M}B}\sigma+\nabla_{X}^{A}a\cdot_{A^{*}}\sigma=\left(\nabla_{X}^{A}a\right)^{\sigma},
\]
using the fact that $\nabla_{X}^{A^{*}\otimes_{M}B}\sigma=0$, since
$\sigma$ is a parallel tensor field.
\end{proof}

\section{Decomposition of $\pi_{E}^{TE}\colon TE\to E$\label{sec:Decomposition-of-TE}}

In using the calculus of variations on a manifold $M$ where the Lagrangian
is a function of $TM$ (this form of Lagrangian is ubiquitous in mechanics),
taking the first variation involves passing to $TTM$. Without a way
to decompose variations into more tractable components, the standard
integration-by-parts trick \citep[pg. 16]{Giaquinta&Hildebrandt}
can't be applied. The notion of a local trivialization of $TTM$ via
choice of coordinates on $M$ is one way to provide such a decomposition.
A coordinate chart $\left(U,\phi\colon U\to\mathbb{R}^{n}\right)$
on $M$ establishes a locally trivializing diffeomorphism $TTU\cong\phi\left(U\right)\times\mathbb{R}^{n}\times\mathbb{R}^{n}\times\mathbb{R}^{n}$.
However such a trivialization imposes an artificial additive structure
on $TTU$ depending on the {[}non-canonical{]} choice of coordinates,
only gives a local formulation of the relevant objects, and the ensuing
coordinate calculations don't give clear insight into the geometric
structure of the problem. The notion of the linear connection remedies
this.\\

A \textbf{linear connection} on the vector bundle $\pi\colon E\to M$
is a subbundle $H\to E$ of $\pi_{E}^{TE}\colon TE\to E$ such that
$TE=H\oplus_{E}VE$ and $T\lambda_{a}\cdot H_{x}=H_{ax}$ for all
$a\in\mathbb{R}\backslash\left\{ 0\right\} $ and $x\in E$, where
$\lambda_{a}\colon E\to E,\, e\mapsto ae$ is the scalar multiplication
action of $a$ on $E$ \citep[pg. 512]{JeffMLee}. The bundle $H\to E$
may also be called a \textbf{horizontal space} of the vector bundle
$\pi_{E}^{TE}\colon TE\to E$ (``a'' is used instead of ``the''
because a choice of $H\to E$ is generally non-unique). For convenience,
define $h:=\nnabla\pi\in\Gamma\left(\pi^{*}TM\otimes_{E}T^{*}E\right),$
noting then that $VE=\ker h$.

A linear connection can equivalently be specified by what is known
as a connection map; essentially a projection onto the vertical bundle.
This is a slightly more active formulation than just the specification
of a horizontal space, as a covariant derivative can be defined directly
in terms of the connection map -- see \citep[pg. 518]{JeffMLee},
\citep[pg. 128]{Ebin&Marsden}, \citep[pg. 173]{Eliasson}, and \citep[pg. 208]{Michor}.
\begin{prop}[Connection map formulation of a linear connection]
 If $v\in\Gamma\left(\pi^{*}E\otimes_{E}T^{*}E\right)$ (i.e. $v\colon TE\to E$
is a smooth vector bundle morphism over $\pi$) is a left-inverse
for $\iota_{VE}^{\pi^{*}E}\in\Gamma\left(VE\otimes_{E}\pi^{*}E^{*}\right)$
that is equivariant with respect to $T\lambda_{a}$ and $\lambda_{a}$
(i.e. $v\cdot T\lambda_{a}=\pi^{*}\lambda_{a}\cdot v$) \citep[pg. 245]{Michor},
then $H:=\ker v\leq TE$ defines a linear connection on the vector
bundle $\pi\colon E\to M$. Such a map $v$ is called the \textbf{connection
map} associated to $H$. Conversely, given a linear connection $H$,
there is exactly one connection map defining $H$ in the stated sense. \end{prop}
\begin{proof}
That $v$ is a left-inverse for $\iota_{VE}^{\pi^{*}E}$ implies that
$v$ has full rank, so $H:=\ker v$ defines a subbundle of $\pi_{E}^{TE}\colon TE\to E$
having the same rank as $TM$. Because $v$ is smooth, $H$ is a smooth
subbundle. Furthermore, the condition implies that $V_{e}E\cap H_{e}=\left\{ 0\right\} $
for each $e\in E$, and therefore $TE=H\oplus_{E}VE$ by a rank-counting
argument.

If $x\in TE$ and $a\in\mathbb{R}\backslash\left\{ 0\right\} $, then
$v\cdot T\lambda_{a}\cdot x=\pi^{*}\lambda_{a}\cdot v\cdot x$, which
equals zero if and only if $v\cdot x=0$, i.e. if and only if $x\in H$.
Thus $T\lambda_{a}\cdot H=H$. This establishes $H\to E$ as a linear
connection. 

Conversely, if $H$ is a linear connection and $v_{1}$ and $v_{2}$
are connection maps for $H$, then $v_{1}\cdot_{TE}\iota_{VE}^{\pi^{*}E}=\Id_{\pi^{*}E}=v_{2}\cdot_{TE}\iota_{VE}^{\pi^{*}E}$.
Then because the image of $\iota_{VE}^{\pi^{*}E}$ is all of $VE$,
it follows that $v_{1}\mid_{VE}=v_{2}\mid_{VE}$. Since $v_{1}\mid_{H}=0=v_{2}\mid_{H}$
by definition, and since $TE=H\oplus_{E}VE$, this shows that $v_{1}=v_{2}$.
Uniqueness of connection maps has been established. To show existence,
define $v:=\iota_{\pi^{*}E}^{VE}\cdot_{VE}\pr_{VE}\in\Gamma\left(\pi^{*}E\otimes_{E}T^{*}E\right)$,
where $\pr_{VE}\colon H\oplus_{E}VE\to VE$ be the canonical projection,
recalling that $H\oplus_{E}VE=TE$. It is easily shown that $v$ is
a connection map for $H$.\end{proof}
\begin{prop}[Decomposing $\pi_{E}^{TE}\colon TE\to E$]
 \label{proposition_decomposing_TE} If $v\in\Gamma\left(\pi^{*}E\otimes_{E}T^{*}E\right)$
is a connection map, then 
\begin{eqnarray}
h\oplus_{E}v\colon TE & \to & \pi^{*}TM\oplus_{E}\pi^{*}E\label{eq:TE_decomposition}
\end{eqnarray}
is a smooth vector bundle isomorphism over $\Id_{E}$. See Figure
\ref{fig:horiz-vert-decomp}.\end{prop}
\begin{proof}
Because $TE=H\oplus_{E}VE$, and $H=\ker v$ and $VE=\ker h$, the
fiber-wise restriction 
\begin{eqnarray*}
h\oplus_{E}v\mid_{T_{e}E}\colon T_{e}E & \to & \left(\pi^{*}TM\oplus_{E}\pi^{*}E\right)_{e}\cong T_{\pi\left(e\right)}M\oplus E_{\pi\left(e\right)}
\end{eqnarray*}
is a linear isomorphism for each $e\in E$. The map is a smooth vector
bundle morphism over $\Id_{E}$ by construction. It is therefore a
smooth vector bundle isomorphism over $\Id_{E}$.
\end{proof}
\begin{figure}
\centering
\def\svgwidth{\columnwidth}
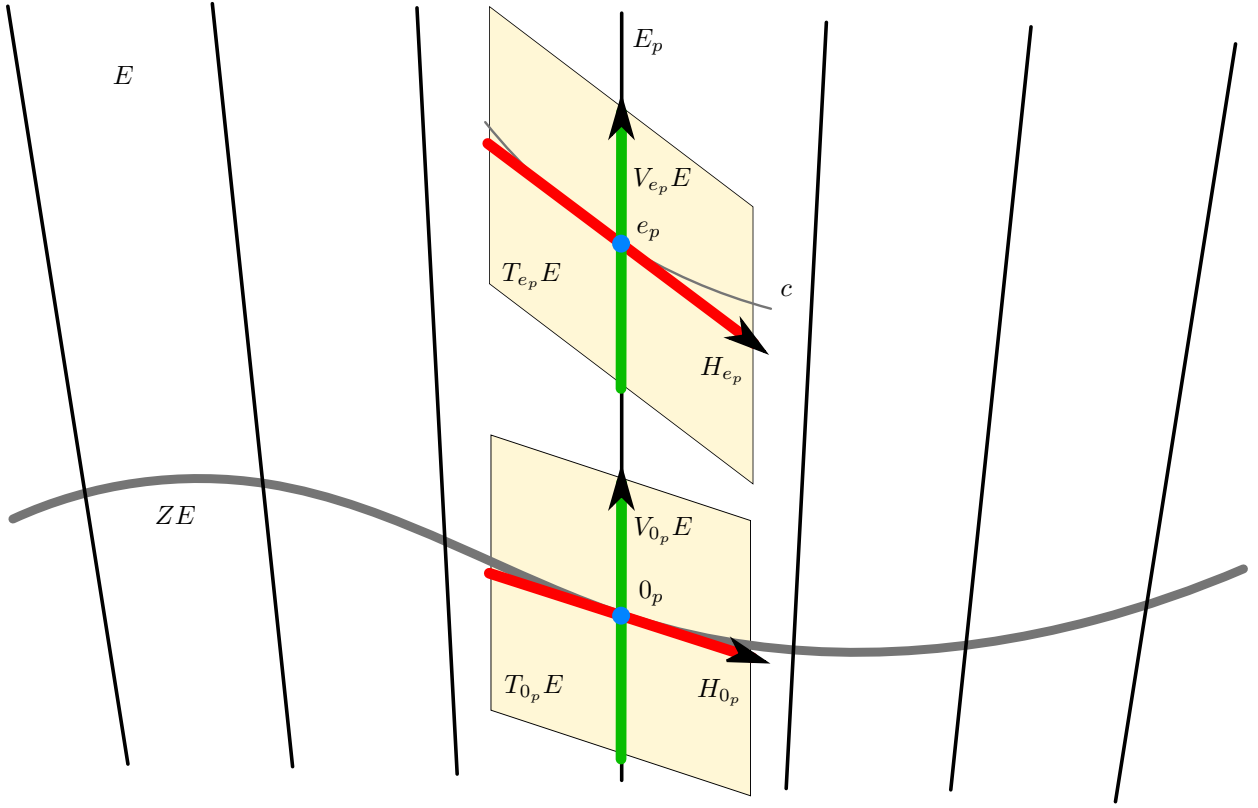

\caption{\label{fig:horiz-vert-decomp} A diagram representing the decomposition
of $TE\to E$ into horizontal and vertical subbundles. The vertical
lines represent individual fibers of $E$, while $p\in M$, $e_{p}\in E_{p}$,
$0_{p}\in E_{p}$ denotes the zero vector of $E_{p}$, and $ZE$ denotes
the zero subbundle of $E$; $ZE\cong M$. By the equivariance property
of the linear connection, $ZE$ is a submanifold of $E$ which is
entirely horizontal (its tangent space is entirely composed of horizontal
vectors). The tangent spaces $T_{0_{p}}E$ and $T_{e_{p}}E$ are drawn;
green arrows representing the vertical subspaces (``along'' the
fibers), red arrows representing the horizontal subspaces. Finally,
$c$ is a horizontal curve passing through $e_{p}$. }
\end{figure}

\begin{rem}[Linear connection/covariant derivative correspondence]
 \label{remark:connection_derivative_correspondence} Given a covariant
derivative $\nabla^{E}$ on a smooth vector bundle $\pi\colon E\to M$,
there is a naturally induced linear connection, defined via the connection
map 
\begin{eqnarray}
v\colon TE & \to & E,\label{eq:vertical_projection_from_covariant_derivative}\\
\delta_{\epsilon}\Theta & \mapsto & \nabla_{\delta_{\epsilon}}^{\left(\pi\circ\Theta\right)^{*}E}\Theta,\nonumber 
\end{eqnarray}
where $\Theta\colon I\to E$ is a variation of $\theta\in E$. Here,
$\nabla^{\left(\pi\circ\Theta\right)^{*}E}$ denotes the pullback
of the covariant derivative $\nabla^{E}$ through the map $\pi\circ\Theta$
(see (\ref{proposition_pullback_covariant_derivative})). Conceptually,
all $v$ does is replace an ordinary derivative ($\delta_{\epsilon}$)
with the corresponding covariant one ($\nabla_{\delta_{\epsilon}}^{\left(\pi\circ\Theta\right)^{*}E}$).

Conversely, given a connection map $v\in\Gamma\left(\pi^{*}E\otimes_{E}T^{*}E\right)$
for a linear connection $H\to E$, there is a naturally induced covariant
derivative $\nabla^{E}$ on the smooth vector bundle $\pi\colon E\to M$,
defined by
\begin{eqnarray*}
\nabla^{E}\colon\Gamma\left(E\right) & \to & \Gamma\left(E\otimes_{M}T^{*}M\right),\\
\sigma & \mapsto & \sigma^{*}v\cdot_{\sigma^{*}TE}\nnabla^{M\to E}\sigma.
\end{eqnarray*}
The scaling equivariance of $v$ is critical for showing that this
map actually defines a covariant derivative. Full type safety should
be observed here; by the contravariance of the pullback of bundles
(see (\ref{prop:pullback_is_contravariant_functor})), $\sigma^{*}\pi^{*}E\cong\left(\pi\circ\sigma\right)^{*}E=\Id_{M}^{*}E\cong E$,
so
\[
\sigma^{*}v\in\Gamma\left(\sigma^{*}\left(\pi^{*}E\otimes_{E}T^{*}E\right)\right)\cong\Gamma\left(\sigma^{*}\pi^{*}E\otimes_{M}\sigma^{*}T^{*}E\right)\cong\Gamma\left(E\otimes_{M}\sigma^{*}T^{*}E\right),
\]
and therefore $\sigma^{*}v\cdot\nnabla\sigma\in\Gamma\left(E\otimes_{M}T^{*}M\right)$
as desired. This connection map construction of a covariant derivative
gives (\ref{proposition_pullback_covariant_derivative}) as an immediate
consequence via the chain rule for the tangent map.\\

\end{rem}
The following construction is an abstraction of taking partial derivatives
of a function, inspired by \citep[pg. 277]{Marsden&Hughes}. Instead
of taking partial derivatives with respect to individual coordinates,
partial covariant derivatives along distributions over the base manifold
are formed, where the distributions (subbundles) decompose the base
manifold's tangent bundle into a direct sum. Such a construction conveniently
captures the geometry of maps with respect to the geometry of its
domain.
\begin{prop}[Partial covariant derivatives]
 \label{prop:partial-covariant-derivatives} Let $L\in C^{\infty}\left(M,\mathbb{R}\right)$,
and for each $i\in\left\{ 1,\dots,n\right\} $ let $F_{i}\to M$ be
a smooth vector bundle. If, for each $i\in\left\{ 1,\dots,n\right\} $,
$c_{i}\in\Gamma\left(F_{i}\otimes_{M}T^{*}M\right)$ such that $c_{1}\oplus_{M}\dots\oplus_{M}c_{n}\in\Gamma\left(\left(F_{1}\oplus_{M}\dots\oplus_{M}F_{n}\right)\otimes_{M}T^{*}M\right)$
is a smooth vector bundle isomorphism over $\Id_{E}$, then there
exist unique sections $L_{,c_{i}}\in\Gamma\left(F_{i}^{*}\right)$
for each $i\in\left\{ 1,\dots,n\right\} $ such that
\[
\nabla^{M\to\mathbb{R}}L=L_{,c_{1}}\cdot_{F_{1}}c_{1}+\dots+L_{,c_{n}}\cdot_{F_{n}}c_{n}.
\]
This decomposition of $\nabla L$ provides what will be called \textbf{partial
covariant derivatives} of $L$ (with respect to the given decomposition). \end{prop}
\begin{proof}
The following equivalences provide a formula for directly defining
$L_{,c_{1}},\dots,L_{,c_{n}}$.
\begin{align*}
\nabla L & =L_{,c_{1}}\cdot_{F_{1}}c_{1}+\dots+L_{,c_{n}}\cdot_{F_{n}}c_{n}\\
\iff\nabla L & =\left(L_{,c_{1}}\oplus_{M}\dots\oplus_{M}L_{,c_{n}}\right)\cdot_{F_{1}\oplus_{M}\dots\oplus_{M}F_{n}}\left(c_{1}\oplus_{M}\dots\oplus_{M}c_{n}\right)\\
\iff\nabla L\cdot_{TM}\left(c_{1}\oplus_{M}\dots\oplus_{M}c_{n}\right)^{-1} & =L_{,c_{1}}\oplus_{M}\dots\oplus_{M}L_{,c_{n}}.
\end{align*}
Existence and uniqueness is therefore proven.\end{proof}
\begin{cor}[Horizontal/vertical derivatives]
 \label{cor:horizontal_vertical_derivatives} Let $h:=\nnabla\pi\in\Gamma\left(\pi^{*}TM\otimes_{E}T^{*}E\right)$
as before. If $v\in\Gamma\left(\pi^{*}E\otimes_{E}T^{*}E\right)$
is a connection map, and if $L\colon E\to\mathbb{R}$ is smooth, then
there exist unique $L_{,h}\in\Gamma\left(\pi^{*}T^{*}M\right)$ and
$L_{,v}\in\Gamma\left(\pi^{*}E^{*}\right)$ such that $\nabla L=L_{,h}\cdot_{\pi^{*}TM}h+L_{,v}\cdot_{\pi^{*}E}v$.
\end{cor}
It should be noted that the basepoint-preserving issue discussed in
Section \ref{sec:Tangent-Map-as-Tensor-Field} plays a role in choosing
to use the tensor field formulation of $h\colon TE\to TM$ and $v\colon TE\to E$.
In particular, without preserving the basepoint (via the $\pi$-pullback
of $TM$ and $E$ to form $h\in\Gamma\left(\pi^{*}TM\otimes_{E}T^{*}E\right)$
and $v\in\Gamma\left(\pi^{*}E\otimes_{E}T^{*}E\right)$), the map
$h\oplus_{E}v$ would not be a smooth bundle isomorphism, and the
horizontal and vertical derivatives would be maps of the form $L_{,h}\colon E\to T^{*}M$
and $L_{,v}\colon E\to E^{*}$, but that, critically, are \emph{not}
sections of smooth vector bundles, and can only claim to be smooth
{[}fiber{]} bundle morphisms. Derivative trivializations will be central
in calculating the first and second variations of an energy functional
having Lagrangian $L$ (see (\ref{thm:first_variation_of_L}) and
(\ref{thm:second_variation_of_L})).

\section{Curvature and Commutation of Derivatives \label{sec:Curvature-and-Commutation-of-Derivatives}}

A ubiquitous consideration in mathematics is to determine when two
operations commute. In the setting of tensor calculus, this often
manifests itself in determining the commutativity (or lack thereof)
of two covariant derivatives. Here, ``covariant derivatives'' may
refer to both linear covariant derivatives and the tangent map operator
(see (\ref{rem:generalized-covariant-derivative})). This unified
categorization of derivatives will now be leveraged to show that certain
fiber bundles are flat (in a sense analogous to the vanishing of a
curvature endomorphism) with respect to particular covariant derivatives.
This reduces the work often done showing commutativity of derivatives
in the derivation of the first variation of a function in the calculus
of variations to the simple statement that a particular tensor field
is symmetric, which is comes as a corollary to the aforementioned
flatness.

In this section, the symbol $\nabla$ may denote $\nnabla$ or $\lnabla$,
depending on context. This eases the expression of repeated covariant
derivatives, such as the covariant Hessian of a section (see below),
and is an example of telescoping notation as discussed in Section
\ref{sec:Strongly-Typed-Linear-Algebra}.\\

If $\pi\colon E\to M$ defines a smooth {[}fiber{]} bundle whose space
of sections $\Gamma\left(E\right)$ has two repeated covariant derivatives
defined and if $\nabla^{TM}$ is a symmetric linear covariant derivative
(meaning $\nabla_{X}Y-\nabla_{Y}X=\left[X,Y\right]$ for $X,Y\in\Gamma\left(TM\right)$),
then the tensor contraction
\[
\nabla^{2}\sigma:\left(X\otimes_{M}Y-Y\otimes_{M}X\right)
\]
is an expression measuring the non-commutativity of the $X$ and $Y$
derivatives of $\sigma$. The quantity $\nabla^{2}\sigma$ will be
called the \textbf{covariant Hessian} of $\sigma$, because it generalizes
the Hessian of elementary calculus; it contains only second-derivative
information, and in the special case seen below, it is symmetric in
the argument components. It should be noted that if $F\to M$ is the
vector bundle such that $\nabla\sigma\in\Gamma\left(F\otimes_{M}T^{*}M\right)$,
then $\nabla^{2}\sigma\in\Gamma\left(F\otimes_{M}T^{*}M\otimes_{M}T^{*}M\right)$.
Intentionally leaving the $\nabla$ and $\cdot$ symbols undecorated
in preference of contextual interpretation, unwinding the expression
above gives
\begin{align*}
\nabla^{2}\sigma:\left(X\otimes_{M}Y-Y\otimes_{M}X\right) & =\nabla_{Y}\nabla\sigma\cdot X-\nabla_{X}\nabla\sigma\cdot Y\\
 & =\nabla_{Y}\nabla_{X}\sigma-\nabla\sigma\cdot\nabla_{Y}X-\nabla_{X}\nabla_{Y}\sigma+\nabla\sigma\cdot\nabla_{X}Y\\
 & =-\nabla_{X}\nabla_{Y}\sigma+\nabla_{Y}\nabla_{X}\sigma+\nabla\sigma\cdot\left[X,Y\right]\\
 & =-\nabla_{X}\nabla_{Y}\sigma+\nabla_{Y}\nabla_{X}\sigma+\nabla_{\left[X,Y\right]}\sigma,
\end{align*}
which is syntactically identical to the common definition for the
{[}Riemannian{]} curvature endomorphism $R\left(X,Y\right)\sigma$.
In the traditional setting, where $\nabla^{E}$ is a linear covariant
derivative on vector bundle $E$, the curvature endomorphism takes
the form of a tensor field $R^{E}\in\Gamma\left(E\otimes_{M}E^{*}\otimes_{M}T^{*}M\otimes_{M}T^{*}M\right)$.
In this setting however, because $\nabla^{E}$ may be nonlinear, such
a tensorial formulation doesn't generally exist. Instead, 
\[
R^{E}\left(X,Y\right):=-\nabla_{X}\nabla_{Y}^{E}+\nabla_{Y}\nabla_{X}^{E}+\nabla_{\left[X,Y\right]}^{E}
\]
defines a second-order covariant differential operator (``covariant''
meaning tensorial in the $X$ and $Y$ components). Put differently,
\[
R^{E}\left(X,Y\right)\sigma=\nabla^{2}\sigma:\left(X\otimes_{M}Y-Y\otimes_{M}X\right),
\]
which will be called the (possibly nonlinear)\textbf{ curvature operator},
measures the non-commutativity of the $X$ and $Y$ derivatives of
$\sigma$. If $R^{E}$ is identically zero, then the bundle $E$ is
said to be \textbf{flat} with respect to the relevant connections/covariant
derivatives.

There are two particularly important instances of flat bundles. The
first is the trivial line bundle defined by $\pi^{\mathbb{R}\rtriv S}$
(whose space of smooth sections, as discussed in Section \ref{sec:Bundle-Constructions},
is naturally identified with $C^{\infty}\left(S,\mathbb{R}\right)$).
In this case, $\nabla^{S\to\mathbb{R}}f\in\Gamma\left(T^{*}S\right)$,
and $\nabla^{2}f\equiv\lnabla^{T^{*}S}\lnabla^{S\to\mathbb{R}}f\in\Gamma\left(T^{*}S\otimes_{S}T^{*}S\right)$
is the object referred to in most literature as the covariant Hessian
of $f$. Here, $R^{S\to\mathbb{R}}\left(X,Y\right)f$ is a real-valued
function on $S$.
\begin{prop}[Symmetry of covariant Hessian on functions]
 \label{prop:symmetry_of_covariant_hessian_on_functions} Let $S$
be a smooth manifold and let $\nabla^{TS}$ be a symmetric covariant
derivative. If $f\in C^{\infty}\left(S,\mathbb{R}\right)$, then $\nabla^{2}f\in\Gamma\left(T^{*}S\otimes_{S}T^{*}S\right)$
is a symmetric tensor field (i.e. it has a $\left(1\,2\right)$ symmetry).
Here, the covariant derivative on $C^{\infty}\left(S,\mathbb{R}\right)$
is $\nabla^{S\to\mathbb{R}}$ as defined above.\end{prop}
\begin{proof}
Let $X,Y\in\Gamma\left(TS\right)$. Recall that $\nabla f\equiv df\in\Gamma\left(T^{*}S\right)$.
Then
\begin{align*}
 & \nabla^{2}f:\left(X\otimes_{S}Y-Y\otimes_{S}X\right)\\
={} & \nabla_{Y}\nabla f\cdot X-\nabla_{X}\nabla f\cdot Y\\
={} & \nabla_{Y}\left(\nabla f\cdot X\right)-\nabla f\cdot\nabla_{Y}X-\nabla_{X}\left(\nabla f\cdot Y\right)+\nabla f\cdot\nabla_{X}Y\\
={} & \nabla\left(\nabla f\cdot X\right)\cdot Y-\nabla\left(\nabla f\cdot Y\right)\cdot X+\nabla f\cdot\left[X,Y\right] & \mbox{(by symmetry of \ensuremath{\nabla^{TS}})}\\
={} & -\nabla f\cdot\left[X,Y\right]+\nabla f\cdot\left[X,Y\right] & \mbox{(by definition of \ensuremath{\left[X,Y\right]})}\\
={} & 0.
\end{align*}
Because $X\otimes_{S}Y$ is pointwise-arbitrary in $TS\otimes_{S}TS$,
this shows that $\nabla^{2}f$ is symmetric. Equivalently stated,
$R^{S\to\mathbb{R}}$ is identically zero, and therefore the relevant
bundle is flat.
\end{proof}
The second important case involves the nonlinear covariant derivative
$\nabla^{M\to S}$ on $C^{\infty}\left(M,S\right)$. Here, if $\phi\in C^{\infty}\left(M,S\right)$,
then 
\[
\nabla^{2}\phi\equiv\lnabla^{\phi^{*}TS\otimes_{M}T^{*}M}\nnabla^{M\to S}\phi\in\Gamma\left(\phi^{*}TS\otimes_{M}T^{*}M\otimes_{M}T^{*}M\right),
\]
so $R^{M\to S}\left(X,Y\right)\phi\in\Gamma\left(\phi^{*}TS\right)$. 
\begin{prop}[Symmetry of covariant Hessian on maps]
 \label{prop:symmetry_of_covariant_hessian_on_maps} Let $M$ and
$S$ be smooth manifolds and let $\nabla^{TM}$ and $\nabla^{TS}$
be symmetric covariant derivatives. If $\phi\in C^{\infty}\left(M,S\right)$,
then $\nabla^{2}\phi\in\Gamma\left(\phi^{*}TS\otimes_{M}T^{*}M\otimes_{M}T^{*}M\right)$
is a tensor field which is symmetric in the two $T^{*}M$ components
(i.e. it has a $\left(2\,3\right)$ symmetry). Here, the covariant
derivative on $C^{\infty}\left(M,S\right)$ is $\nnabla^{M\to S}$
as defined above.\end{prop}
\begin{proof}
Let $X,Y\in\Gamma\left(TM\right)$ and $f\in C^{\infty}\left(S,\mathbb{R}\right)$,
so that $\phi^{*}\nabla f\in\Gamma\left(\phi^{*}TS\right)$. Then
\begin{align*}
\phi^{*}\nabla f\cdot_{\phi^{*}TS}R^{M\to S}\left(X,Y\right)\phi={} & \phi^{*}\nabla f\cdot\left(-\nabla_{X}\nabla_{Y}\phi+\nabla_{Y}\nabla_{X}\phi+\nabla_{\left[X,Y\right]}\phi\right)\\
={} & -\nabla_{X}\left(\phi^{*}\nabla f\cdot\nabla\phi\cdot Y\right)+\nabla_{X}\phi^{*}\nabla f\cdot\nabla\phi\cdot Y\\
 & +\nabla_{Y}\left(\phi^{*}\nabla f\cdot\nabla\phi\cdot X\right)-\nabla_{Y}\phi^{*}\nabla f\cdot\nabla\phi\cdot X\\
 & +\phi^{*}\nabla f\cdot\nabla\phi\cdot\left[X,Y\right]\\
={} & -\nabla_{X}\left(\nabla\phi^{*}f\cdot Y\right)+\nabla_{Y}\left(\nabla\phi^{*}f\cdot X\right)+\nabla\phi^{*}f\cdot\left[X,Y\right]\\
 & +\left(\phi^{*}\nabla^{2}f\cdot\nabla\phi\cdot X\right)\cdot\nabla\phi\cdot Y-\left(\phi^{*}\nabla^{2}f\cdot\nabla\phi\cdot Y\right)\cdot\nabla\phi\cdot X\\
={} & -\nabla\left(\nabla\phi^{*}f\cdot Y\right)\cdot X+\nabla\left(\nabla\phi^{*}f\cdot X\right)\cdot Y+\nabla\phi^{*}f\cdot\left[X,Y\right]\\
 & -\phi^{*}\nabla^{2}f:\left(\left(\nabla\phi\cdot X\right)\otimes_{M}\left(\nabla\phi\cdot Y\right)-\left(\nabla\phi\cdot Y\right)\otimes_{M}\left(\nabla\phi\cdot X\right)\right).
\end{align*}
By definition, $-\nabla\left(\nabla\phi^{*}f\cdot Y\right)\cdot X+\nabla\left(\nabla\phi^{*}f\cdot X\right)\cdot Y=-\nabla\phi^{*}f\cdot\left[X,Y\right]$,
which cancels out the other term. By (\ref{prop:symmetry_of_covariant_hessian_on_functions}),
$\nabla^{2}f$ is symmetric, so the final term is zero. Because $\phi^{*}\nabla f$
is pointwise-arbitrary in $\phi^{*}T^{*}S$ and $X$ and $Y$ are
pointwise-arbitrary in $TM$, this shows that $R^{M\to S}$ is identically
zero, so the bundle defined by $\pi_{M}^{S\times M}\colon S\times M\to M$,
whose space of sections is identified with $C^{\infty}\left(M,S\right)$,
is flat, and therefore $\nabla^{2}\phi$ is symmetric in its two $T^{*}M$
components.
\end{proof}
The construction used in (\ref{prop:partial-covariant-derivatives})
can be applied to nonlinear as well as linear covariant derivatives
to considerable advantage. For example, if $\psi\colon M\times N\to L$,
where $M,N,L$ are smooth manifolds and $p_{M}:=\pr_{1}^{M\times N}$
and $p_{N}:=\pr_{2}^{M\times N}$, then define $\psi_{,M}\in\Gamma\left(\psi^{*}TL\otimes_{M\times N}p_{M}^{*}T^{*}M\right)$
and $\psi_{,N}\in\Gamma\left(\psi^{*}TL\otimes_{M\times N}p_{N}^{*}T^{*}N\right)$
by 
\[
\nnabla\psi=\nnabla^{M\times N\to L}\psi=\psi_{,M}\cdot_{p_{M}^{*}TM}\nnabla p_{M}+\psi_{,N}\cdot_{p_{N}^{*}TN}\nnabla p_{N}.
\]
This gives a convenient way to express partial covariant derivatives,
which will be used heavily in Part \ref{part:Riemannian-Calculus-of-Variations}
in calculating the first and second variations of an energy functional.
Note that in this parlance, $\psi_{,\left(M\times N\right)}$ is the
full tangent map $\nnabla\psi$.

Defining second partial covariant derivatives $\psi_{,MM}$, $\psi_{,MN}$,
$\psi_{,NM}$ and $\psi_{,NN}$ by
\begin{align*}
\nabla\psi_{,M} & =\psi_{,MM}\cdot\nnabla p_{M}+\psi_{,MN}\cdot\nnabla p_{N}\mbox{ and}\\
\nabla\psi_{,N} & =\psi_{,NM}\cdot\nnabla p_{M}+\psi_{,NN}\cdot\nnabla p_{N},
\end{align*}
the symmetry of the covariant Hessian of $\psi$ can be used to show
various symmetries these second derivatives. 
\begin{prop}[Symmetries of partial covariant derivatives]
  \label{prop:symmetries-of-partial-covariant-derivatives} With
$\psi$ and its second partial covariant derivatives as above, 
\[
\psi_{,MM}\in\Gamma\left(\psi^{*}TL\otimes_{M\times N}p_{M}^{*}T^{*}M\otimes_{M\times N}p_{M}^{*}T^{*}M\right)
\]
and $\psi_{,NN}$ (having analogous type) are $\left(2\,3\right)$-symmetric
(i.e. $\left(\psi_{,MM}\right)^{\left(2\,3\right)}=\psi_{,MM}$ and
$\left(\psi_{,NN}\right)^{\left(2\,3\right)}=\psi_{,NN}$) and the
mixed, second partial covariant derivatives 
\begin{align*}
\psi_{,MN} & \in\Gamma\left(\psi^{*}TL\otimes_{M\times N}p_{M}^{*}T^{*}N\otimes_{M\times N}p_{N}^{*}T^{*}N\right)\mbox{ and}\\
\psi_{,NM} & \in\Gamma\left(\psi^{*}TL\otimes_{M\times N}p_{N}^{*}T^{*}N\otimes_{M\times N}p_{M}^{*}T^{*}M\right)
\end{align*}
are mutually $\left(2\,3\right)$-symmetric (i.e. $\psi_{,MN}=\left(\psi_{,NM}\right)^{\left(2\,3\right)}$).\end{prop}
\begin{proof}
Let $X,Y\in\Gamma\left(TM\oplus TN\right)$. If $Tp_{N}\cdot X=0$
and $Tp_{M}\cdot Y=0$, then
\begin{align*}
0={} & \nabla^{2}\psi:\left(X\otimes_{M\times N}Y-Y\otimes_{M\times N}X\right)\mbox{ (by (\ref{prop:symmetry_of_covariant_hessian_on_maps}))}\\
={} & \psi_{,MM}:\left(\nnabla p_{M}\cdot X\otimes_{M\times N}\nnabla p_{M}\cdot Y-\nnabla p_{M}\cdot Y\otimes_{M\times N}\nnabla p_{M}\cdot X\right)\\
 & +\psi_{,MN}:\left(\nnabla p_{M}\cdot X\otimes_{M\times N}\nnabla p_{N}\cdot Y-\nnabla p_{M}\cdot Y\otimes_{M\times N}\nnabla p_{N}\cdot X\right)\\
 & +\psi_{,NM}:\left(\nnabla p_{N}\cdot X\otimes_{M\times N}\nnabla p_{M}\cdot Y-\nnabla p_{N}\cdot Y\otimes_{M\times N}\nnabla p_{M}\cdot X\right)\\
 & +\psi_{,NN}:\left(\nnabla p_{N}\cdot X\otimes_{M\times N}\nnabla p_{N}\cdot Y-\nnabla p_{N}\cdot Y\otimes_{M\times N}\nnabla p_{N}\cdot X\right)\\
={} & \psi_{,MN}:\left(\nnabla p_{M}\cdot X\otimes_{M\times N}\nnabla p_{N}\cdot Y-\nnabla p_{N}\cdot Y\otimes_{M\times N}\nnabla p_{M}\cdot X\right)\\
={} & \left(\psi_{,MN}-\left(\psi_{,NM}\right)^{\left(2\,3\right)}\right):\left(\nnabla p_{M}\cdot X\otimes_{M\times N}\nnabla p_{N}\cdot Y\right).
\end{align*}
Because $\nnabla p_{M}\cdot X$ and $\nnabla p_{N}\cdot Y$ are pointwise-arbitrary
in $p_{M}^{*}TM$ and $p_{N}^{*}TN$ respectively, this implies that
$\psi_{,MN}=\left(\psi_{,NM}\right)^{\left(2\,3\right)}$. Analogous
calculations (setting $\nnabla p_{M}\cdot X=0$ and $\nnabla p_{M}\cdot Y=0$
and then separately setting $\nnabla p_{N}\cdot X=0$ and $\nnabla p_{N}\cdot Y=0$)
show that $\psi_{,MM}=\left(\psi_{,MM}\right)^{\left(2\,3\right)}$
and $\psi_{,NN}=\left(\psi_{,NN}\right)^{\left(2\,3\right)}$.
\end{proof}
There are two final results regarding the second covariant derivative
that will be especially useful in the calculation of the first and
second variations of an energy functional (see (\ref{thm:first_variation_of_L})
and (\ref{thm:second_variation_of_L_alternate})).
\begin{prop}[Chain rule for covariant Hessian]
 \label{prop:chain_rule_for_covariant_hessian} Let $\pi\colon E\to N$
define a bundle having a first and second covariant derivative (i.e.
a section of $E$ can be covariantly differentiated twice). If $\phi\colon M\to N$
and $e\in\Gamma\left(E\right)$, then
\[
\nabla^{2}\phi^{*}e=\phi^{*}\nabla^{2}e:_{\phi^{*}TN}\left(\nnabla\phi\boxtimes_{M}\nnabla\phi\right)+\phi^{*}\nabla e\cdot_{\phi^{*}TN}\nabla\nnabla\phi.
\]
\end{prop}
\begin{proof}
Let $X\in\Gamma\left(TM\right)$. Then
\begin{align*}
\nabla^{2}\phi^{*}e\cdot X & =\nabla_{X}\nabla^{\phi^{*}E}\phi^{*}e\\
 & =\nabla_{X}\left(\phi^{*}\nabla^{E}e\cdot_{\phi^{*}TN}\nnabla\phi\right)\\
 & =\nabla_{X}\left(\phi^{*}\nabla e\right)\cdot_{\phi^{*}TN}\nnabla\phi+\phi^{*}\nabla e\cdot_{\phi^{*}TN}\nabla_{X}\nnabla\phi\\
 & =\left(\phi^{*}\nabla^{2}e\cdot_{\phi^{*}TN}\nnabla\phi\cdot X\right)\cdot_{\phi^{*}TN}\nnabla\phi+\phi^{*}\nabla e\cdot_{\phi^{*}TN}\nabla_{X}\nnabla\phi\\
 & =\left[\phi^{*}\nabla^{2}e:_{\phi^{*}TN}\left(\nnabla\phi\boxtimes_{M}\nnabla\phi\right)+\phi^{*}\nabla e\cdot_{\phi^{*}TN}\nabla\nnabla\phi\right]\cdot X.
\end{align*}
Because $X$ is pointwise-arbitrary in $TM$, this establishes the
desired equality.\end{proof}
\begin{prop}[Pullback curvature endomorphism]
 \label{prop:pullback_curvature_endomorphism} Let $\pi\colon E\to N$
define a vector bundle having first and second covariant derivatives.
If $\phi\colon M\to N$, then $R^{\phi^{*}TN}=\phi^{*}R^{TN}:_{\phi^{*}TN}\left(\nnabla\phi\boxtimes_{M}\nnabla\phi\right)$.\end{prop}
\begin{proof}
Note that $R^{\phi^{*}TN}\in\Gamma\left(\phi^{*}TN\otimes_{M}\phi^{*}T^{*}N\otimes_{M}T^{*}M\otimes_{M}T^{*}M\right)$.
Let $X,Y\in\Gamma\left(TM\right)$ and let $Z\in\Gamma\left(TN\right)$,
so that $\phi^{*}Z\in\Gamma\left(\phi^{*}TN\right)$. Then
\begin{align*}
 & \left(\Id_{\phi^{*}TN}\otimes_{M}\phi^{*}Z\right)\cdot_{\phi^{*}TN\otimes_{M}\phi^{*}T^{*}N}R^{\phi^{*}TN}:_{TM}\left(X\otimes_{M}Y\right)\\
={} & R^{\phi^{*}TN}\left(X,Y\right)\left(\phi^{*}Z\right)\\
={} & \nabla^{2}\phi^{*}Z:_{TM}\left(X\wedge_{M}Y\right)\\
={} & \phi^{*}\nabla^{2}Z:_{\phi^{*}TN}\left(\nnabla\phi\boxtimes_{M}\nnabla\phi\right):_{TM}\left(X\wedge_{M}Y\right)\\
 & +\phi^{*}\nabla Z\cdot_{\phi^{*}TN}\nabla\nnabla\phi:_{TM}\left(X\wedge_{M}Y\right) & \mbox{(by (\ref{prop:chain_rule_for_covariant_hessian}))}\\
={} & \phi^{*}\nabla^{2}Z:_{\phi^{*}TN}\left(\left(\nnabla\phi\cdot X\right)\wedge_{M}\left(\nnabla\phi\cdot Y\right)\right)+\phi^{*}\nabla Z\cdot_{\phi^{*}TN}0 & \mbox{(by symmetry of \ensuremath{\nabla\nnabla\phi})}\\
={} & \left(\phi^{*}\left(\left(\Id_{TN}\otimes_{N}Z\right)\cdot_{TN\otimes_{N}T^{*}N}R^{TN}\right)\right):_{\phi^{*}TN}\left(\left(\nnabla\phi\cdot X\right)\otimes_{M}\left(\nnabla\phi\cdot Y\right)\right)\\
={} & \left(\Id_{\phi^{*}TN}\otimes_{M}\phi^{*}Z\right)\cdot_{\phi^{*}TN\otimes_{M}\phi^{*}T^{*}N}\phi^{*}R^{TN}:_{\phi^{*}TN}\left(\nnabla\phi\boxtimes_{M}\nnabla\phi\right):_{TM}\left(X\otimes_{M}Y\right),
\end{align*}
and because $X,Y$ and $\phi^{*}Z$ are pointwise-arbitrary in their
respective spaces, this establishes the desired equality. 
\end{proof}
A common operation is to evaluate a covariant derivative along a single
tangent vector. One can express a single tangent vector as a section
of a particular pullback bundle, the map being the constant map evaluating
to the basepoint of the vector. This allows the richly-typed formalism
of pullback bundles to be used to evaluate derivatives at a point,
particularly noting that this safely deals with the overloading of
the natural pairing operator $\cdot$ (see Section \ref{sec:Strongly-Typed-Tensor-Field}).
\begin{prop}[Evaluation commutes with non-involved derivatives]
 \label{prop:evaluation_commutes_with_noninvolved_derivatives} Let
$A$ and $B$ be smooth manifolds and let $\sigma\in\Gamma\left(E\right)$
for some smooth bundle $E\to A\times B$ having a covariant derivative
$\nabla^{E}$. If $b\in B$ and the map $z\colon A\to A\times B,\, a\mapsto\left(a,b\right)$
represents evaluation at $b$, then
\[
z^{*}\left(\sigma_{,A}\right)=\left(z^{*}\sigma\right)_{,A},
\]
i.e. evaluation in $B$ commutes with a derivative along $A$.\end{prop}
\begin{proof}
Let $X\in\Gamma\left(TA\right)$, and let $p_{A}:=\pr_{1}^{A\times B}$
and $p_{B}:=\pr_{2}^{A\times B}$. Then
\begin{align*}
\left(z^{*}\sigma\right)_{,A}\cdot X & =\nabla^{z^{*}E}z^{*}\sigma\cdot X\\
 & =z^{*}\nabla^{E}\sigma\cdot_{z^{*}\left(TA\oplus TB\right)}\nnabla z\cdot_{TA}X\\
 & =z^{*}\left(\nabla^{E}\sigma\cdot_{TA\oplus TB}p_{A}^{*}X\right)\\
 & =z^{*}\left(\sigma_{,A}\cdot_{p_{A}^{*}TA}\nnabla p_{A}\cdot_{TA\oplus TB}p_{A}^{*}X\right) & \mbox{(since \ensuremath{\nnabla p_{B}\cdot_{TA\oplus TB}p_{A}^{*}X=0})}\\
 & =z^{*}\sigma_{,A}\cdot_{TA}X & \mbox{(since \ensuremath{z^{*}p_{A}^{*}X=\left(p_{A}\circ z\right)^{*}X=\Id_{A}^{*}X=X}),}
\end{align*}
and because $X$ is pointwise-arbitrary in $TM$, this implies that
$z^{*}\sigma_{,A}=\left(z^{*}\sigma\right)_{,A}$ as desired.\end{proof}
\begin{prop}
\label{prop:evaluation_of_trivialized_derivative} Let $A,B,C$ be
smooth manifolds, let $\psi\colon A\times B\to C$ be smooth, let
$p_{A}:=\pr_{1}^{A\times B}$ and $p_{B}:=\pr_{2}^{A\times B}$, and
let $X,Y\in\Gamma\left(TA\oplus TB\right)$. If $\nnabla p_{B}\cdot X=0$
and $\nnabla p_{A}\cdot Y=0$, then
\[
\psi_{,AB}:\left(\left(\nnabla p_{A}\cdot X\right)\otimes_{A\times B}\left(\nnabla p_{B}\cdot Y\right)\right)=\nabla_{Y}^{\psi^{*}TC}\nabla_{X}^{A\times B\to C}\psi.
\]
\end{prop}
\begin{proof}
The conditions $\nnabla p_{B}\cdot X=0$ and $\nnabla p_{A}\cdot Y=0$
imply that $\nabla_{Y}X=0$ in the product covariant derivative. Then
since $p_{A}\times_{A\times B}p_{B}=\Id_{A\times B}$, it follows
that 
\[
\nabla\nnabla p_{A}\oplus_{A\times B}\nabla\nnabla p_{B}=\nabla\left(\nnabla p_{A}\oplus_{A\times B}\nnabla p_{B}\right)=\nabla\nnabla\left(p_{A}\times_{A\times B}p_{B}\right)=\nabla\Id_{TA\oplus TB}=0,
\]
and therefore
\[
\nabla_{Y}\left(\nnabla p_{A}\cdot X\right)=\nabla_{Y}\nnabla p_{A}\cdot X+\nnabla p_{A}\cdot\nabla_{Y}X=0\cdot X+\nnabla p_{A}\cdot0=0.
\]
For the main calculation,
\begin{align*}
 & \psi_{,AB}:\left(\left(\nnabla p_{A}\cdot X\right)\otimes_{A\times B}\left(\nnabla p_{B}\cdot Y\right)\right)\\
={} & \left(\psi_{,AB}\cdot_{p_{B}^{*}TB}\nnabla p_{B}\cdot_{TA\oplus TB}Y\right)\cdot_{p_{A}^{*}TA}\nnabla p_{A}\cdot_{TA\oplus TB}X\\
={} & \left(\nabla_{Y}^{\psi\times_{A\times B}p_{A}}\psi_{,A}\right)\cdot_{p_{A}^{*}TA}\nnabla p_{A}\cdot_{TA\oplus TB}X & \mbox{(since \ensuremath{\nnabla p_{A}\cdot Y=0})}\\
={} & \nabla_{Y}^{\psi}\left(\psi_{,A}\cdot\nnabla p_{A}\cdot X\right)-\psi_{,A}\cdot\nabla_{Y}^{p_{A}}\left(\nnabla p_{A}\cdot X\right) & \mbox{(by reverse product rule)}\\
={} & \nabla_{Y}^{\psi^{*}TC}\nabla_{X}^{A\times B\to C}\psi & \mbox{(since \ensuremath{\nabla_{Y}\left(\nnabla p_{A}\cdot X\right)=0}),}
\end{align*}
as desired.
\end{proof}
\newpage{}

\part{Riemannian Calculus of Variations\label{part:Riemannian-Calculus-of-Variations}}

The use of the Calculus of Variations in the Riemannian setting to
develop the geodesic equations and to study harmonic maps is quite
well-established. A more general formulation is required for more
specific applications, such as continuum mechanics in Riemannian manifolds.
The tools developed in Part \ref{part:Mathematical-Setting} will
now be used to formulate the first and second variations and Euler-Lagrange
equations of an energy functional corresponding to a first-order Lagrangian.
In particular, the bundle decomposition discussed in Section \ref{sec:Decomposition-of-TE}
will be needed to employ the standard integration-by-parts trick seen
in the formulation of the analogous parts of the elementary Calculus
of Variations. The seemingly heavy and pedantic formalism built up
thus far will now show its usefulness.\\

In this part, let $\left(M,g\right)$ and $\left(S,h\right)$ be Riemannian
manifolds with $M$ compact. Calculations will be done formally in
the space $C^{\infty}\left(M,S\right)$, noting that its completion
under various norms will give various Sobolev spaces of maps from
$M$ to $S$, which are ultimately the spaces which must be considered
when finding critical points of the relevant energy functionals. See
\citep{Ebin&Marsden,Eliasson} for details on the analytical issues.
Let $dV_{g}$ denote the Riemannian volume form corresponding to metric
$g$, and let $d\overline{V}_{g}$ be the induced volume form on $\partial M$.
Let $\iota\colon\partial M\to M$ be the inclusion, and let $\nu\in\Gamma\left(\iota^{*}T^{*}M\right)$
be the unit normal covector field on $\partial M$. Let $E:=TS\otimes_{S\times M}T^{*}M$
and $\pi:=\pi_{S}^{TS}\otimes_{S\times M}\pi_{M}^{T^{*}M}$, making
$\pi\colon E\to S\times M$ a vector bundle. 

The energy functionals in this section will be assumed to have the
form
\begin{eqnarray*}
\mathcal{L}\colon C^{\infty}\left(M,S\right) & \to & \mathbb{R},\\
\phi & \mapsto & \int_{M}L\circ\nnabla\phi\, dV_{g},
\end{eqnarray*}
where $L\colon E\to\mathbb{R}$, referred to as the \textbf{Lagrangian}
of the functional, is smooth. Here, $\nnabla\phi$ could be understood
to take values either in $E=TS\otimes_{S\times M}T^{*}M$ or $\phi^{*}TS\otimes_{M}T^{*}M$.
In the former case, the composition $L\circ\nnabla\phi$ is literal,
while in the latter case, there is an implicit conversion from $\phi^{*}TS\otimes_{M}T^{*}M$
to $TS\otimes_{S\times M}T^{*}M$ via a fiber projection bundle morphism
(see (\ref{cor:pullback-fiber-projection})). Either way, $L\circ\nnabla\phi\colon M\to\mathbb{R}$.
Let $\nabla^{TS}$ and $\nabla^{TM}$ denote the respective Levi-Civita
connections, which induce a covariant derivative $\nabla^{E}$ on
$E$ (see (\ref{proposition_induced_covariant_derivatives_on_sum_and_product_bundles})).
Define the connection map $v\in\Gamma\left(\pi^{*}E\otimes_{E}T^{*}E\right)$
using $\nabla^{E}$ as in (\ref{eq:vertical_projection_from_covariant_derivative}).
For convenience, the $S\times M$ subscript will be suppressed on
the ``full'' tensor product defining $E$ from here forward.

\section{Critical Points and Variations}

One of the most pertinent properties of an energy functional is its
set of critical points. Often, the solution to a problem in physics
will take the form of minimizing a particular energy functional. Lagrangian
mechanics is the quintessintial example of this. This section will
deal with some of the main considerations regarding such critical
points.\\

Because the domain of a {[}real-valued{]} functional $\mathcal{L}$
may be a nonlinear space, the relevant first derivative is the {[}real-valued{]}
differential $d\mathcal{L}$, which is paired with the linearized
variation of a map $\phi\in C^{\infty}\left(M,S\right)$. In particular,
a \textbf{one-parameter variation} of $\phi$ is a smooth map $\Phi\colon M\times I\to S$,
where the $I$ component is the variational parameter. Letting $i$
denote the standard coordinate on $I$, the linearized variation is
then $\delta_{i}\Phi\colon M\to TS$, recalling that $\delta_{i}:=\frac{\partial}{\partial i}\mid_{i=0}$.
Because $\pi_{S}^{TS}\circ\delta_{i}\Phi=\phi$, it follows that $\delta_{i}\Phi\in\Gamma\left(\phi^{*}TS\right)$,
i.e. $\delta_{i}\Phi$ is a vector field along $\phi$. The object
$\delta_{i}\Phi$ will be called a \textbf{linearized variation}.
Call the elements of $\Gamma\left(\phi^{*}TS\right)$ \textbf{linear
variations}.
\begin{prop}[Each linear variation is a linearized variation]
 \label{prop:each_linear_variation_is_a_linearized_variation} Let
$\exp\colon U\to S$ denote the exponential map associated to $\nabla^{TS}$,
where $U\subseteq TS$ is a neighborhood of the zero bundle in $TS$
on which $\exp$ is defined, and let $\lambda\colon TS\times\mathbb{R}\to TS,\,\left(s,\epsilon\right)\mapsto\epsilon s$
denote the scalar multiplication structure on $TS$. If $A\in\Gamma\left(\phi^{*}TS\right)$
and if $\Phi\colon U\to S$ is defined by $\Phi:=\exp\circ\lambda\circ\left(A\times\Id_{I}\right)\mid_{U}$,
then $\delta_{i}\Phi=A$. In other words, every vector field over
$\phi$ is realized as the linearization of a one-parameter variation
of $\phi$. \end{prop}
\begin{proof}
The map $\Phi$ is well-defined and smooth by construction. Let $p\in M$.
Then
\begin{align*}
\left(\delta_{i}\Phi\right)\left(p\right) & =\delta_{i}\left(\Phi\left(p,i\right)\right)\\
 & =\delta_{i}\left(\exp\circ\lambda\circ\left(A\times\Id_{I}\right)\left(p,i\right)\right)\\
 & =\nnabla\exp\cdot\delta_{i}\left(\lambda\left(A\left(p\right),i\right)\right)\\
 & =\nnabla\exp\cdot\delta_{i}\left(iA\left(p\right)\right)\\
 & =\nnabla\exp\cdot\left(\iota_{VE}^{\pi^{*}E}\mid_{Z\left(\pi^{*}E\right)}\right)\cdot A\left(p\right)\\
 & =A\left(p\right),
\end{align*}
where $Z\left(\pi^{*}E\right)$ denotes the zero subbundle of $\pi^{*}E$.
The last equality follows from a naturality property of the exponential
map \citep[pg. 523]{IntroLee}.
\end{proof}
Thus each linear variation is a linearized variation, establishing
a natural identification of $T_{\phi}\left(C^{\infty}\left(M,S\right)\right)$
with $\Gamma\left(\phi^{*}TS\right)$, which will be useful when calculating
the differential of a functional on $C^{\infty}\left(M,S\right)$.
In fact, the exponential map construction in (\ref{prop:each_linear_variation_is_a_linearized_variation})
is a way to construct charts for the infinite dimensional manifold
$C^{\infty}\left(M,S\right)$ \citep[Theorem 5.2]{Eliasson}.

\section{First Variation\label{sec:First-Variation}}

This section is devoted to calculating the first variation of the
previously defined energy functional. Here is where the full richness
of the type system of the objects developed earlier in the paper will
really show their power (and arguably, necessity). While the type-specifying
notation may appear overly decorated and pedantic, subtle usage errors
can be detected and avoided by keeping track of the myriad of types
of the relevant objects through the sub/superscripts on covariant
derivatives and natural pairings; extremely complex constructions
can be made and navigated without much trouble. By contrast, performing
the ensuing calculations in coordinate trivializations would result
in an intractible proliferation of Christoffel symbols and indexed
expressions which would prove difficult to read and would be highly
prone to error.\\

Because the Lagrangian $L\colon E\to\mathbb{R}$ is defined on a vector
bundle $\pi\colon E\to S\times M$ over the product space $S\times M$,
the decomposition in (\ref{cor:horizontal_vertical_derivatives})
can be slightly refined. The projection $\pi$ can be decomposed into
the factors $\pi_{S}:=\pr_{S}^{S\times M}\circ\pi$ and $\pi_{M}:=\pr_{M}^{S\times M}\circ\pi$,
so that $\pi=\pi_{S}\times_{E}\pi_{M}$. Then $h=\nnabla\pi=\nnabla\pi_{S}\oplus_{E}\nnabla\pi_{M}$.
Let 
\[
\sigma:=\nnabla\pi_{S}\in\Gamma\left(\pi_{S}^{*}TS\otimes_{E}T^{*}E\right)\mbox{ and }\mu:=\nnabla\pi_{M}\in\Gamma\left(\pi_{M}^{*}TM\otimes_{E}T^{*}E\right).
\]
The letters \emph{sigma} and \emph{mu} have been chosen to reflect
the fact that $L_{,\sigma}\in\Gamma\left(\pi_{S}^{*}T^{*}S\right)$
and $L_{,\mu}\in\Gamma\left(\pi_{M}^{*}T^{*}M\right)$ give the ``$S$
component'' (spatial) and ``$M$ component'' (material) of the
derivative $\nabla^{E\to\mathbb{R}}L\in\Gamma\left(T^{*}E\right)$.
The connection map $v$ will be retained as is, giving $L_{,v}\in\Gamma\left(\pi^{*}E^{*}\right)$,
the ``$E$ component'' (fiber) of $\nabla^{E\to\mathbb{R}}L$. See
(\ref{remark:analogs-in-elementary-cov}) for a discussion of how
the quantities $L_{,\mu},L_{,\sigma},L_{,v}$ generalize the analogous
structures in the elementary treatment of the calculus of variations.

Because a one-parameter variation of $\phi\in C^{\infty}\left(M,S\right)$
has the form $\Phi\colon M\times I\to S$ but the energy functional
$\mathcal{L}$ involves only the $M$ derivative of its argument,
the partial tangent map must be used here. For the purposes of calculating
the first and second variations, $\mathcal{L}$ must be written as
\[
\mathcal{L}\left(\phi\right):=\int_{M}L\circ\phi_{,M}\, dV_{g}.
\]

\begin{thm}[First variation of $\mathcal{L}$]
 \label{thm:first_variation_of_L} Let $\mathcal{L}$, $L$, $\sigma$,
$\mu$, $v$ and $\nu$ all be defined as above. If $\phi\in C^{\infty}\left(M,S\right)$
and $A\in\Gamma\left(\phi^{*}TS\right)$, then
\[
d\mathcal{L}\left(\phi\right)\cdot A=\int_{M}A\cdot_{\phi^{*}T^{*}S}\left(\phi_{,M}^{*}L_{,\sigma}-\Div_{M}\left(\phi_{,M}^{*}L_{,v}\right)\right)\, dV_{g}+\int_{\partial M}A\cdot_{\phi^{*}T^{*}S}\phi_{,M}^{*}L_{,v}\cdot_{T^{*}M}\nu\, d\overline{V}_{g}.
\]

\end{thm}
The expression above is often called the \textbf{first variation}
of $\mathcal{L}$. A type analysis here gives $\phi_{,M}^{*}L_{,\sigma}\in\Gamma\left(\phi^{*}T^{*}S\right)$
and $\phi_{,M}^{*}L_{,v}\in\Gamma\left(\phi^{*}T^{*}S\otimes_{M}TM\right)$.
Recall that because the domain of $\phi$ is $M$, $\nnabla\phi\equiv\phi_{,M}$.
\begin{proof}
Supporting calculations will be made below in lemmas. Let $\Phi\colon M\times I\to S$
be as in (\ref{prop:each_linear_variation_is_a_linearized_variation}),
so that $\delta_{i}\Phi=A$. For tidiness, let $\mathbf{L}_{,\sigma}:=\phi_{,M}^{*}L_{,\sigma}$
and $\mathbf{L}_{,v}:=\phi_{,M}^{*}L_{,v}$. Then
\begin{align*}
d\mathcal{L}\left(\phi\right)\cdot A & =d\mathcal{L}\left(\phi\right)\cdot\delta_{i}\Phi\\
 & =\delta_{i}\left(\mathcal{L}\left(\Phi\right)\right)\\
 & =\int_{M}\delta_{i}\left(L\circ\Phi_{,M}\right)\, dV_{g}\\
 & =\int_{M}\mathbf{L}_{,\sigma}\cdot_{\phi^{*}TS}A+\mathbf{L}_{,v}\cdot_{\phi^{*}TS\otimes_{M}T^{*}M}\nabla^{\phi^{*}TS}A\, dV_{g} & \mbox{(by (\ref{lem:first-variation-calculation-1}))}\\
 & =\int_{M}A\cdot_{\phi^{*}T^{*}S}\left(\mathbf{L}_{,\sigma}-\Div_{M}\mathbf{L}_{,V}\right)+\Div_{M}\left(A\cdot_{\phi^{*}T^{*}S}\mathbf{L}_{,V}\right)\, dV_{g} & \mbox{(by (\ref{lem:first-variation-calculation-1}))}\\
 & =\int_{M}A\cdot_{\phi^{*}T^{*}S}\left(\mathbf{L}_{,\sigma}-\Div_{M}\mathbf{L}_{,V}\right)\, dV_{g}+\int_{\partial M}A\cdot_{\phi^{*}T^{*}S}\mathbf{L}_{,V}\cdot_{T^{*}M}\nu\, d\overline{V}_{g} & \mbox{(divergence theorem),}
\end{align*}
as desired. 

As for the types of\textbf{ }$\phi_{,M}^{*}L_{,\sigma}$ and $\phi_{,M}^{*}L_{,v}$,
the contravariance of bundle pullback allows significant simplification.
Because $L_{,\sigma}\in\Gamma\left(\pi_{S}^{*}T^{*}S\right)$ and
$L_{,v}\in\Gamma\left(\pi^{*}E^{*}\right)$,
\begin{align*}
\phi_{,M}^{*}L_{,\sigma} & \in\Gamma\left(\phi_{,M}^{*}\pi_{S}^{*}T^{*}S\right)=\Gamma\left(\left(\pi_{S}\circ\phi_{,M}\right)^{*}T^{*}S\right)=\Gamma\left(\phi^{*}T^{*}S\right)\mbox{ and}\\
\phi_{,M}^{*}L_{,v} & \in\Gamma\left(\phi_{,M}^{*}\pi^{*}E\right)=\Gamma\left(\left(\pi\circ\phi_{,M}\right)^{*}\left(T^{*}S\otimes TM\right)\right)=\Gamma\left(\phi^{*}T^{*}S\otimes_{M}TM\right).
\end{align*}
\\

The supporting calculations follow. Define $z\colon M\to M\times I,\, m\mapsto\left(m,0\right)$
for purposes of evaluation of $i=0$ via precomposition as in (\ref{prop:evaluation_commutes_with_noninvolved_derivatives}).
Then $\delta_{i}$ is a section of a pullback bundle; $\delta_{i}=z^{*}\partial_{i}\in\Gamma\left(z^{*}\left(TM\oplus TI\right)\right)$.
It should be noted that $\Phi\circ z=\phi$ by definition, and that
$z^{*}\Phi_{,M}=\left(z^{*}\Phi\right)_{,M}=\phi_{,M}$ by (\ref{prop:evaluation_commutes_with_noninvolved_derivatives}).\end{proof}
\begin{lem}
\label{lem:first-variation-calculation-1} Let $L$, $\Phi$, $A$,
$\sigma$, and $v$ be as in Theorem \ref{thm:first_variation_of_L}.
The variational derivative of $L\circ\Phi_{,M}$ decomposes in terms
of the partial covariant derivatives $\mathbf{L}_{,\sigma}$ and $\mathbf{L}_{,v}$
and the linearized variation $A$; 
\[
\delta_{i}\left(L\circ\Phi_{,M}\right)=\phi_{,M}^{*}L_{,\sigma}\cdot_{\phi^{*}TS}\delta_{i}\Phi+\phi_{,M}^{*}L_{,v}\cdot_{\left(\phi\times_{M}\Id_{M}\right)^{*}E}\nabla^{\phi}\delta_{i}\Phi.
\]
The integration-by-parts trick as in the derivation of the first variation
in elementary calculus of variations generalizes to the covariant
setting;

\[
\mathbf{L}_{,\sigma}\cdot_{\phi^{*}TS}A+\mathbf{L}_{,v}\cdot_{\phi^{*}TS\otimes_{M}T^{*}M}\nabla^{\phi}A=A\cdot_{\phi^{*}T^{*}S}\left(\mathbf{L}_{,\sigma}-\Div_{M}\mathbf{L}_{,v}\right)+\Div_{M}\left(A\cdot_{\phi^{*}T^{*}S}\mathbf{L}_{,v}\right).
\]
\end{lem}
\begin{proof}
A wonderful string of equalities follows.
\begin{align*}
 & \delta_{i}\left(L\circ\Phi_{,M}\right)\\
={} & z^{*}\nabla^{M\times I\to\mathbb{R}}\left(L\circ\Phi_{,M}\right)\cdot_{z^{*}\left(TM\oplus TI\right)}\delta_{i} & \mbox{ (here, \ensuremath{\delta_{i}=z^{*}\partial_{i}})}\\
={} & z^{*}\Phi_{,M}^{*}\nabla^{E\to\mathbb{R}}L\cdot_{z^{*}\Phi_{,M}^{*}TE}z^{*}\nnabla^{M\times I\to E}\Phi_{,M}\cdot_{z^{*}\left(TM\oplus TI\right)}\delta_{i} & \mbox{(chain rule)}\\
={} & \phi_{,M}^{*}\left(L_{,\sigma}\cdot_{\pi_{S}^{*}TS}\sigma+L_{,\mu}\cdot_{\pi_{M}^{*}TM}\mu+L_{,v}\cdot_{\pi^{*}E}v\right)\cdot_{\phi_{,M}^{*}TE}\delta_{i}\Phi_{,M} & \mbox{(by (\ref{prop:partial-covariant-derivatives}) and because \ensuremath{\Phi_{,M}\circ z=\phi_{,M}})}\\
={} & \phi_{,M}^{*}L_{,\sigma}\cdot_{\phi_{,M}^{*}\pi_{S}^{*}TS}\phi_{,M}^{*}\sigma\cdot_{\phi_{,M}^{*}TE}\delta_{i}\Phi_{,M}\\
 & +\phi_{,M}^{*}L_{,\mu}\cdot_{\phi_{,M}^{*}\pi_{M}^{*}TM}\phi_{,M}^{*}\mu\cdot_{\phi_{,M}^{*}TE}\delta_{i}\Phi_{,M}\\
 & +\phi_{,M}^{*}L_{,v}\cdot_{\phi_{,M}^{*}\pi^{*}E}\phi_{,M}^{*}v\cdot_{\phi_{,M}^{*}TE}\delta_{i}\Phi_{,M}\\
={} & \phi_{,M}^{*}L_{,\sigma}\cdot_{\phi^{*}TS}\delta_{i}\Phi+\phi_{,M}^{*}L_{,v}\cdot_{\left(\phi\times_{M}\Id_{M}\right)^{*}E}\nabla^{\phi}\delta_{i}\Phi & \mbox{(by (\ref{lem:first-variation-calculation-2}))}
\end{align*}
Note that by (\ref{prop:pullback_is_contravariant_functor}), $\Phi_{,M}^{*}\pi_{S}^{*}TS=\left(\pi_{S}\circ\Phi_{,M}\right)^{*}TS=\Phi^{*}TS$,
$\Phi_{,M}^{*}\pi_{M}^{*}TM=\left(\pi_{M}\circ\Phi_{,M}\right)^{*}TM=\left(\pr_{M}^{M\times I}\right)^{*}TM$
and $\Phi_{,M}^{*}\pi^{*}E=\left(\pi\circ\Phi_{,M}\right)^{*}E=\left(\Phi\times_{M\times I}\pr_{M}^{M\times I}\right)^{*}E$.
Replacing $\delta_{i}\Phi$ with $A$ gives
\[
\delta_{i}\left(L\circ\Phi_{,M}\right)=\mathbf{L}_{,\sigma}\cdot_{\phi^{*}TS}A+\mathbf{L}_{,v}\cdot_{\phi^{*}TS\otimes_{M}T^{*}M}\nabla^{\phi}A,
\]
establishing the first equality.

For the second, 
\begin{align*}
 & \mathbf{L}_{,\sigma}\cdot_{\phi^{*}TS}A+\mathbf{L}_{,v}\cdot_{\phi^{*}TS\otimes_{M}T^{*}M}\nabla^{\phi}A\\
={} & \mathbf{L}_{,\sigma}\cdot_{\phi^{*}TS}A+\tr_{T^{*}M}\left(\mathbf{L}_{,v}\cdot_{\phi^{*}TS}\nabla^{\phi}A\right) & \mbox{(tracing \ensuremath{TM}separately)}\\
={} & A\cdot_{\phi^{*}T^{*}S}\mathbf{L}_{,\sigma}+\tr_{T^{*}M}\left(\nabla\left(\mathbf{L}_{,v}\cdot_{\phi^{*}TS}A\right)-\left(\nabla^{\phi\times_{M}\Id_{M}}\mathbf{L}_{,v}\right)\cdot_{\phi^{*}TS}A\right) & \mbox{(reverse product rule)}\\
={} & A\cdot_{\phi^{*}T^{*}S}\mathbf{L}_{,\sigma}-A\cdot_{\phi^{*}T^{*}S}\tr_{T^{*}M}\nabla^{\phi\times_{M}\Id_{M}}\mathbf{L}_{,v}+\tr_{T^{*}M}\nabla\left(A\cdot_{\phi^{*}T^{*}S}\mathbf{L}_{,v}\right) & \mbox{(\ensuremath{\cdot_{\phi^{*}TS}}commutes with \ensuremath{\tr_{T^{*}M}})}\\
={} & A\cdot_{\phi^{*}T^{*}S}\left(\mathbf{L}_{,\sigma}-\Div_{M}\mathbf{L}_{,v}\right)+\Div_{M}\left(A\cdot_{\phi^{*}T^{*}S}\mathbf{L}_{,v}\right) & \mbox{(definition of \ensuremath{\Div_{M}}).}
\end{align*}
Note that $\mathbf{L}_{,v}\in\Gamma\left(\phi^{*}T^{*}S\otimes_{M}TM\right)$,
so $\Div_{M}\mathbf{L}_{,v}\in\Gamma\left(\phi^{*}T^{*}S\right)$
and $A\cdot_{\phi^{*}T^{*}S}\mathbf{L}_{,v}\in\Gamma\left(TM\right)$.\end{proof}
\begin{lem}
\label{lem:first-variation-calculation-2} The variation $\delta_{i}\Phi_{,M}$
decomposes as follows.
\begin{align*}
\phi_{,M}^{*}\sigma\cdot_{\phi_{,M}^{*}TE}\delta_{i}\Phi_{,M} & =\delta_{i}\Phi\in\Gamma\left(\phi^{*}TS\right),\\
\phi_{,M}^{*}\mu\cdot_{\phi_{,M}^{*}TE}\delta_{i}\Phi_{,M} & =0\in\Gamma\left(TM\right),\\
\phi_{,M}^{*}v\cdot_{\phi_{,M}^{*}TE}\delta_{i}\Phi_{,M} & =\nabla^{\phi^{*}TS}\delta_{i}\Phi\in\Gamma\left(\phi^{*}TS\otimes_{M}T^{*}M\right).
\end{align*}
\end{lem}
\begin{proof}
This calculation determines the $\sigma$ component of $\delta_{i}\Phi_{,M}$.
\begin{align*}
\phi_{,M}^{*}\sigma\cdot_{\phi_{,M}^{*}TE}\delta_{i}\Phi_{,M} & =\phi_{,M}^{*}\nnabla\pi_{S}\cdot_{\phi_{,M}^{*}TE}\delta_{i}\Phi_{,M}\\
 & =\delta_{i}\left(\pi_{S}\circ\Phi_{,M}\right)\\
 & =\delta_{i}\left(\pr_{S}^{S\times M}\circ\pi\circ\Phi_{,M}\right)\\
 & =\delta_{i}\Phi\in\Gamma\left(z^{*}\Phi^{*}TS\right)\cong\Gamma\left(\phi^{*}TS\right).
\end{align*}
This calculation determines the $\mu$ component of $\delta_{i}\Phi_{,M}$.
\begin{align*}
\phi_{,M}^{*}\mu\cdot_{\phi_{,M}^{*}TE}\delta_{i}\Phi_{,M} & =\phi_{,M}^{*}\nnabla\pi_{M}\cdot_{\phi_{,M}^{*}TE}\delta_{i}\Phi_{,M}\\
 & =\delta_{i}\left(\pi_{M}\circ\Phi_{,M}\right)\\
 & =\delta_{i}\left(\pr_{M}^{S\times M}\circ\pi\circ\Phi_{,M}\right)\\
 & =\delta_{i}\pr_{M}^{M\times I}\\
 & =0\in\Gamma\left(z^{*}\left(\pr_{M}^{M\times I}\right)^{*}TM\right)\cong\Gamma\left(TM\right).
\end{align*}
The last equality follows from the fact that $\pr_{M}^{M\times I}$
does not depend on the $i$ coordinate.

This calculation determines the $v$ component of $\delta_{i}\Phi_{,M}$.
Let $p_{M}:=\pr_{M}^{M\times I}$ and $p_{I}:=\pr_{I}^{M\times I}$.
The left-hand side of the third equality claimed in the lemma will
be examined before evaluating at $i=0$; 
\[
\Phi_{,M}^{*}v\cdot_{\Phi_{,M}^{*}TE}\partial_{i}\Phi_{,M}=\nabla_{p_{I}^{*}\partial_{i}}^{\left(\pi\circ\Phi_{,M}\right)^{*}E}\Phi_{,M}=\nabla_{p_{I}^{*}\partial_{i}}^{\left(\Phi\times_{M\times I}p_{M}\right)^{*}\left(TS\otimes T^{*}M\right)}\Phi_{,M}\in\Gamma\left(\Phi^{*}TS\otimes_{M\times I}p_{M}^{*}T^{*}M\right).
\]
Let $Y\in\Gamma\left(TM\right)$, noting that $p_{M}^{*}Y\in\Gamma\left(p_{M}^{*}TM\right)$.
Then
\begin{align*}
 & \left(\Phi_{,M}^{*}v\cdot_{\Phi_{,M}^{*}TE}\partial_{i}\Phi_{,M}\right)\cdot_{p_{M}^{*}TM}p_{M}^{*}Y\\
 & =\nabla_{p_{I}^{*}\partial_{i}}^{\Phi^{*}TS\otimes_{M\times I}p_{M}^{*}T^{*}M}\Phi_{,M}\cdot_{p_{M}^{*}TM}p_{M}^{*}Y\\
 & =\left(\Phi_{,MI}\cdot_{p_{I}^{*}TI}p_{I}^{*}\partial_{i}\right)\cdot_{p_{M}^{*}TM}p_{M}^{*}Y\\
 & =\left(\Phi_{,IM}\cdot_{p_{M}^{*}TM}p_{M}^{*}Y\right)\cdot_{p_{I}^{*}TI}p_{I}^{*}\partial_{i} & \mbox{(by (\ref{prop:symmetries-of-partial-covariant-derivatives}))}\\
 & =\left(\Phi_{,I}\cdot_{p_{I}^{*}TI}p_{I}^{*}\partial_{i}\right)_{,M}\cdot_{p_{M}^{*}TM}p_{M}^{*}Y-\Phi_{,I}\cdot_{p_{I}^{*}TI}\left(\left(p_{I}^{*}\partial_{i}\right)_{,M}\cdot_{p_{M}^{*}TM}p_{M}^{*}Y\right)\\
 & =\left(\partial_{i}\Phi\right)_{,M}\cdot_{p_{M}^{*}TM}p_{M}^{*}Y & \mbox{(since \ensuremath{p_{I}^{*}\partial_{i}}doesn't depend on \ensuremath{M}).}
\end{align*}
Recall that $\Id_{M}=p_{M}\circ z$ and that the pullback of bundles
is contravariant. Then evaluating at $i=0$ via pullback by $z$ renders
\begin{align*}
 & \left(\phi_{,M}^{*}v\cdot_{\phi_{,M}^{*}TE}\delta_{i}\Phi_{,M}\right)\cdot_{TM}Y\\
 & =\left(\left(\Phi_{,M}\circ z\right)^{*}v\cdot_{\left(\Phi_{,M}\circ z\right)^{*}TE}z^{*}\partial_{i}\Phi_{,M}\right)\cdot_{\left(p_{M}\circ z\right)^{*}TM}\left(p_{M}\circ z\right)^{*}Y\\
 & =\left(z^{*}\Phi_{,M}^{*}v\cdot_{z^{*}\Phi_{,M}^{*}TE}z^{*}\partial_{i}\Phi_{,M}\right)\cdot_{z^{*}p_{M}^{*}TM}z^{*}p_{M}^{*}Y\\
 & =z^{*}\left(\left(\Phi_{,M}^{*}v\cdot_{\Phi_{,M}^{*}TE}\partial_{i}\Phi_{,M}\right)\cdot_{p_{M}^{*}TM}p_{M}^{*}Y\right)\\
 & =z^{*}\left(\left(\partial_{i}\Phi\right)_{,M}\cdot_{p_{M}^{*}TM}p_{M}^{*}Y\right)\\
 & =z^{*}\left(\partial_{i}\Phi\right)_{,M}\cdot_{z^{*}p_{M}^{*}TM}z^{*}p_{M}^{*}Y\\
 & =\left(z^{*}\partial_{i}\Phi\right)_{,M}\cdot_{\left(p_{M}\circ z\right)^{*}TM}\left(p_{M}\circ z\right)^{*}Y & \mbox{(by (\ref{prop:evaluation_commutes_with_noninvolved_derivatives}))}\\
 & =\left(\delta_{i}\Phi\right)_{,M}\cdot_{TM}Y\\
 & =\nabla^{\phi^{*}TS}\delta_{i}\Phi\cdot_{TM}Y
\end{align*}
The last equality is because $\delta_{i}\Phi\in\Gamma\left(\phi^{*}TS\right)$,
which is a bundle over $M$, and therefore $\left(\delta_{i}\Phi\right)_{,M}$
is the total covariant derivative. Because $Y$ is pointwise-arbitrary
in $TM$, this implies that $\phi_{,M}^{*}\cdot_{\phi_{,M}^{*}TI}\delta_{i}\Phi_{,M}=\nabla^{\phi^{*}TS}\delta_{i}\Phi$,
i.e. the variational derivative $\delta_{i}$ commutes with the first
material derivative, just as in the analogous situation in elementary
calculus of variations.\end{proof}
\begin{cor}[Euler-Lagrange equations]
 \label{cor:euler_lagrange_equations} If $\phi\in C^{\infty}\left(M,S\right)$
is a critical point of $\mathcal{L}$ (i.e. if $d\mathcal{L}\left(\phi\right)\cdot A=0$
for all $A\in\Gamma\left(\phi^{*}TS\right)$), then
\begin{align*}
\phi_{,M}^{*}L_{,\sigma}-\Div_{M}\left(\phi_{,M}^{*}L_{,v}\right) & =0\mbox{ on }M,\\
\phi_{,M}^{*}L_{,v}\cdot_{TM}\nu & =0\mbox{ on }\partial M.
\end{align*}
These are called the \textbf{Euler-Lagrange equations} for the energy
functional $\mathcal{L}$. Recall that because the domain of $\phi$
is $M$, $\nnabla\phi\equiv\phi_{,M}$.\end{cor}
\begin{proof}
This follows trivially from (\ref{thm:first_variation_of_L}) and
the Fundamental Lemma of the Calculus of Variations \citep[pg. 16]{Giaquinta&Hildebrandt}.
\end{proof}
It should be noted that the boundary Euler-Lagrange equation is due
to the fact that the admissible variations are entirely unrestricted.
If, for example, the class of maps being considered had fixed boundary
data, then any variation would vanish at the boundary, and there would
be no boundary Euler-Lagrange equation; this is typically how geodesics
and harmonic maps are formulated.\\

\begin{rem}[Analogs in elementary calculus of variations]
 \label{remark:analogs-in-elementary-cov} The quantities $L_{,\mu},L_{,\sigma},L_{,v}$
generalize the quantities $\frac{\partial L}{\partial x},\frac{\partial L}{\partial z},\frac{\partial L}{\partial p}$
respectively of the elementary treatment of the calculus of variations
for energy functional
\[
\left(f\colon U\to\mathbb{R}^{n}\right)\mapsto\int_{U}L\left(x,f\left(x\right),Df\left(x\right)\right)\, dx,
\]
where $U\subset\mathbb{R}^{m}$ is compact and $U\times\mathbb{R}^{n}\times\mathbb{R}^{m\times n}\ni\left(x,z,p\right)\mapsto L\left(x,z,p\right)$
is the Lagrangian. Here, $\frac{\partial L}{\partial x}\colon U\times\mathbb{R}^{n}\times\mathbb{R}^{m\times n}\to\mathbb{R}^{m},$
$\frac{\partial L}{\partial z}\colon U\times\mathbb{R}^{n}\times\mathbb{R}^{m\times n}\to\mathbb{R}^{n},$
and $\frac{\partial L}{\partial p}\colon U\times\mathbb{R}^{n}\times\mathbb{R}^{m\times n}\to\mathbb{R}^{m\times n}$
decompose the total derivative $dL$ and are defined by the relation
\[
dL\left(x,z,p\right)\cdot\left(u,v,w\right)=\frac{\partial L}{\partial x}\left(x,z,p\right)\cdot u+\frac{\partial L}{\partial z}\left(x,z,p\right)\cdot v+\frac{\partial L}{\partial p}\left(x,z,p\right):w
\]
for $u\in\mathbb{R}^{m},$ $v\in\mathbb{R}^{n},$ and $w\in\mathbb{R}^{m\times n}$.
The Euler-Lagrange equation in this setting is
\[
\left(\frac{\partial L}{\partial z}-\Div_{U}\frac{\partial L}{\partial p}\right)\left(x,f\left(x\right),Df\left(x\right)\right)=0\mbox{ for }x\in U,
\]
noting that the left hand side of the equation takes values in $\mathbb{R}^{n}$.\\

\end{rem}
In most situations involving simpler calculations, it is desirable
and acceptable to dispense with the highly decorated notation and
use trimmed-town, context-dependent notation, leaving off type-specifying
sub/superscripts when clear from context. 
\begin{prop}[Conserved quantity]
 \label{prop:conserved-quantity} If $M$ is a real interval, $\phi\in C^{\infty}\left(M,S\right)$
satisfies the Euler-Lagrange equation, and $L_{,\mu}=0$, then 
\[
H:=\left(\nnabla\phi\right)^{*}L_{,v}\cdot_{\phi^{*}TS\otimes_{M}T^{*}M}\nnabla\phi-\left(\nnabla\phi\right)^{*}L\in C^{\infty}\left(M,\mathbb{R}\right)
\]
is constant. If $L$ is kinetic minus potential energy, then $H$
is kinetic plus potential energy (the total energy), and is referred
to as the Hamiltonian.\end{prop}
\begin{proof}
Let $t$ be the standard real coordinate. Note that because $M$ is
a real interval, $\nnabla\phi=\phi^{\prime}\otimes_{M}dt$. Terms
appearing in the derivative of $H$ can be simplified as follows.
Note the repeated $\nnabla$ derivatives; $\nnabla\phi\colon M\to\phi^{*}TS\otimes_{M}T^{*}M$
but $\nnabla\nnabla\phi\colon M\to\left(\nnabla\phi\right)^{*}T\left(\phi^{*}TS\otimes_{M}T^{*}M\right)$
lands in a higher tangent space. 
\begin{align*}
\left(\nnabla\phi\right)^{*}\sigma\cdot\nnabla\nnabla\phi\cdot\frac{d}{dt} & =\left(\nnabla\phi\right)^{*}\nnabla\pi_{S}\cdot\nnabla\nnabla\phi\cdot\frac{d}{dt}=\frac{d}{dt}\left(\pi_{S}\circ\nnabla\phi\right)=\phi^{\prime},\\
\left(\nnabla\phi\right)^{*}v\cdot\nnabla\nnabla\phi\cdot\frac{d}{dt} & =\left(\nnabla\phi\right)^{*}v\cdot\frac{d}{dt}\nnabla\phi=\nabla_{\frac{d}{dt}}\nnabla\phi,\\
\nabla_{\frac{d}{dt}}\left(\nnabla\phi\right)^{*}L & =\left(\nnabla\phi\right)^{*}\nabla L\cdot\nnabla\nnabla\phi\cdot\frac{d}{dt}\\
 & =\left(\nnabla\phi\right)^{*}L_{,\sigma}\cdot\phi^{\prime}+\left(\nnabla\phi\right)^{*}L_{,v}\cdot\nabla_{\frac{d}{dt}}\nnabla\phi,\\
\nabla_{\frac{d}{dt}}\left(\left(\nnabla\phi\right)^{*}L_{,v}:\nnabla\phi\right) & =\nabla_{\frac{d}{dt}}\left(\nnabla\phi\right)^{*}L_{,v}:\left(\phi^{\prime}\otimes_{M}dt\right)+\left(\nnabla\phi\right)^{*}L_{,v}:\nabla_{\frac{d}{dt}}\nnabla\phi\\
 & =\left(\nabla_{\frac{d}{dt}}\left(\nnabla\phi\right)^{*}L_{,v}\cdot dt\right)\cdot\phi^{\prime}+\left(\nnabla\phi\right)^{*}L_{,v}:\nabla_{\frac{d}{dt}}\nnabla\phi.
\end{align*}
Again, because $M$ is a real interval, the divergence is just the
derivative, so the Euler-Lagrange equation is
\begin{align*}
0 & =\left(\nnabla\phi\right)^{*}L_{,\sigma}-\Div_{M}\left(\left(\nnabla\phi\right)^{*}L_{,v}\right)\\
 & =\left(\nnabla\phi\right)^{*}L_{,\sigma}-\nabla\left(\nnabla\phi\right)^{*}L_{,v}:\left(dt\otimes_{M}\frac{d}{dt}\right)\\
 & =\left(\nnabla\phi\right)^{*}L_{,\sigma}-\nabla_{\frac{d}{dt}}\left(\nnabla\phi\right)^{*}L_{,v}\cdot dt,
\end{align*}
and therefore $\nabla_{\frac{d}{dt}}\left(\nnabla\phi\right)^{*}L_{,v}\cdot dt=\left(\nnabla\phi\right)^{*}L_{,\sigma}$.
Thus
\[
\nabla_{\frac{d}{dt}}H=\nabla_{\frac{d}{dt}}\left(\left(\nnabla\phi\right)^{*}L_{,v}:\nnabla\phi-\left(\nnabla\phi\right)^{*}L\right)=\left(\nabla_{\frac{d}{dt}}\left(\nnabla\phi\right)^{*}L_{,v}\cdot dt-\left(\nnabla\phi\right)^{*}L_{,\sigma}\right)\cdot\phi^{\prime}
\]
which is zero because $\phi$ satisfies the Euler-Lagrange equation.
This shows that $H$ is constant along solutions of the Euler-Lagrange
equation, and is therefore a conserved quantity. It should be noted
that this proof relies on the fact that the divergence takes a particularly
simple form when the domain $M$ is a real interval; the result does
not necessarily hold for a general choice of $M$. \end{proof}
\begin{example}[Harmonic maps]
 \label{example_harmonic_maps} Define a metric 
\[
k\in\Gamma\left(E^{*}\otimes_{S\times M}E^{*}\right)\cong\Gamma\left(\left(TS\otimes T^{*}M\right)\otimes_{S\times M}\left(TS\otimes T^{*}M\right)\right)
\]
in a manner analogous to that in (\ref{example_tensor_product_of_inner_product_spaces});
\[
k:=h\boxtimes g^{-1}.
\]
To clarify, $h\otimes g^{-1}\in\Gamma\left(\left(T^{*}S\otimes_{S}T^{*}S\right)\otimes\left(TM\otimes_{M}TM\right)\right)$,
so permuting the middle two components (as in the definition of $h\boxtimes g^{-1}$)
gives the correct type, including the necessary metric symmetry condition.
If $A\in E$, then $\left|A\right|_{k}^{2}$ is the quantity obtained
by raising/lowering the indices of $A$ and pairing it naturally with
$A$. A useful fact is that $\nabla k=0$; if $u\oplus v\in TS\oplus TM$,
then permutation commutativity (\ref{cor:permutation_tensor_fields_are_parallel})
and the product rule gives 
\[
\nabla_{u\oplus v}k=\nabla_{u\oplus v}\left(h\boxtimes g^{-1}\right)=\nabla_{u}h\boxtimes g^{-1}+h\boxtimes\nabla_{v}g^{-1},
\]
which equals zero because $h$ and $g^{-1}$ are parallel with respect
to $\nabla^{TS}$ and $\nabla^{T^{*}M}$ respectively.

With Lagrangian 
\[
L\colon E\to\mathbb{R},\, A\mapsto\frac{1}{2}\left|A\right|_{k}^{2}
\]
and energy functional 
\[
\mathcal{E}\left(\phi\right):=\int_{M}L\circ\nnabla\phi\, dV_{g}
\]
($\mathcal{E}\left(\phi\right)$ is called the \textbf{energy} of
$\phi$), the resulting Euler-Lagrange equations can be written down
after calculating $L_{,\sigma}$ and $L_{,v}$. It is worthwhile to
note that $L$ is a quadratic form $A\mapsto A:\frac{1}{2}k:A$ on
$E$, which will automatically imply that $L_{,v}\left(A\right)=A:k$.
However, the calculation showing this will be carried out for demonstration
purposes.

Let $A,B\in TS\otimes T^{*}M$. Then $\epsilon\mapsto A+\epsilon B$
is a vertical variation of $A$, since $h\left(\delta\left(A+\epsilon B\right)\right)=0$,
so
\begin{align*}
L_{,v}\left(A\right):B & =L_{,v}\left(A\right):v\cdot\delta_{\epsilon}\left(A+\epsilon B\right)\\
 & =\delta_{\epsilon}\left(L\left(A+\epsilon B\right)\right)\\
 & =\delta_{\epsilon}\left(\left(A+\epsilon B\right):\frac{1}{2}\left(\pi^{*}k\left(A+\epsilon B\right)\right):\left(A+\epsilon B\right)\right).
\end{align*}
The product rule gives three terms. The middle term is zero because
$\pi\left(A+\epsilon B\right)=\pi\left(A\right)$, and therefore does
not depend on $\epsilon$. The basepoint evaluation notation for $\pi^{*}k\left(A\right)$
will be suppressed for brevity (see Section \ref{sec:Strongly-Typed-Tensor-Field}).
Thus
\[
L_{,v}\left(A\right):B=B:\frac{1}{2}k:A+A:\frac{1}{2}k:B=A:k:B,
\]
where the last equality results from the symmetry of $k$. By the
nondegeneracy of the natural pairing on $TS\otimes T^{*}M$ (which
is denoted here by $:$), this implies that $L_{,v}\left(A\right)=A:k$.

To calculate $L_{,\sigma}$, it is sufficient (and can be easier)
to calculate $L_{,h}$, as $h=\nnabla\pi$, $\pi=\pi_{S}\times_{E}\pi_{M}$,
so $h=\sigma\oplus_{E}\mu$. Let $A\left(\epsilon\right)$ be a horizontal
curve in $E=TS\otimes T^{*}M$; this means that $v\cdot\frac{d}{d\epsilon}A=0$.
Recall that $v\cdot\frac{d}{d\epsilon}A$ is defined by $\nabla_{\frac{d}{d\epsilon}}^{\left(\pi\circ A\right)^{*}E}A$.
Then
\begin{align*}
L_{,h}\left(A\right)\cdot_{\pi^{*}\left(TS\oplus TM\right)}h\cdot_{TE}\delta_{\epsilon}A & =\left(L_{,h}\cdot_{\pi^{*}\left(TS\oplus TM\right)}h+L_{,v}\cdot_{\pi^{*}E}v\right)\cdot_{TE}\delta_{\epsilon}A\\
 & =\nabla L\cdot_{TE}\delta_{\epsilon}A\\
 & =\delta_{\epsilon}\left(L\circ A\right)\\
 & =\delta_{\epsilon}\left(A:\frac{1}{2}A^{*}\pi^{*}k:A\right).
\end{align*}
As before, the product rule gives three terms. Using the contravariance
of bundle pullback, the middle term is 
\[
\frac{1}{2}\nabla_{\delta_{\epsilon}}^{\left(\pi\circ A\right)^{*}\left(E^{*}\otimes_{S\times M}E^{*}\right)}\left(\pi\circ A\right)^{*}k=\frac{1}{2}\left(\pi\circ A\right)^{*}\nabla^{E^{*}\otimes_{S\times M}E^{*}}k\cdot\delta_{\epsilon}\left(\pi\circ A\right),
\]
which equals zero because $\nabla k=0$. Thus
\[
L_{,h}\left(A\right)\cdot h\cdot\delta_{\epsilon}A=\nabla_{\delta_{\epsilon}}A:\frac{1}{2}k:A+A:\frac{1}{2}k:\nabla_{\delta_{\epsilon}}A,
\]
which equals zero because $\nabla_{\delta_{\epsilon}}A=v\cdot\delta_{\epsilon}A=0$.
The quantity $h\cdot\delta_{\epsilon}A$ can take any value in $\pi^{*}\left(TS\oplus TM\right)$,
showing that $L_{,h}=0$. Finally, $h=\sigma\oplus_{E}\mu$ implies
that $L_{,\sigma}=0$ and $L_{,\mu}=0$. This can be understood from
the fact that $L$ depends only on the fiber values of $A$, and has
no explicit dependence on the basepoint; this relies crucially on
the fact that $\nabla k=0$.

Finally, the Euler-Lagrange equations can be written down. Recalling
that the natural trace of a tensor (used in the divergence term in
the Euler-Lagrange equation) is contraction with the appropriate identity
tensor, let $\left(e_{i}\right)$ be a local frame for $TM$ and let
$\left(e^{i}\right)$ be its dual coframe, so that $e_{i}\otimes_{M}e^{i}$
is a local expression%
\footnote{It should be noted that while $\Id_{TM}$ is being written as the
local expression $e_{i}\otimes_{M}e^{i}$, no inherently local property
is being used; this tensor decomposition is only used so that the
product rule can be used in the following calculations in a clear
way.%
} for $\Id_{TM}\in\Gamma\left(TM\otimes_{M}T^{*}M\right)$. The type-subscripted
notation will be minimized except to help clarify. On $M$:
\begin{align*}
0 & =\left(\nnabla\phi\right)^{*}L_{,\sigma}-\Div_{M}\left(\left(\nnabla\phi\right)^{*}L_{,v}\right)\\
 & =-\tr\nabla\left(\nnabla\phi:k\right)\\
 & =-\nabla_{e_{i}}\left(\nnabla\phi:k\right)\cdot_{T^{*}M}e^{i}\\
 & =-\nabla_{e_{i}}\nnabla\phi:k\cdot e^{i}-\nnabla\phi:\nabla_{e_{i}}k\cdot e^{i}.
\end{align*}
The second term vanishes because $\nabla k=0$. Unraveling the definition
of $k$ gives $\nabla_{e_{i}}\nnabla\phi:k=\phi^{*}h\cdot\nabla_{e_{i}}\nnabla\phi\cdot g^{-1}$.
Contracting both sides of the above equation with $-\phi^{*}h^{-1}$
gives
\[
0=\nabla_{e_{i}}\nnabla\phi\cdot g^{-1}\cdot e^{i}=\nabla_{e_{i}}\nnabla\phi\cdot e_{i}=\tr_{g}\nabla^{2}\phi\in\Gamma\left(\phi^{*}TS\right).
\]
The quantity $\tr_{g}\nabla^{2}\phi$ is the $g$-trace of the covariant
Hessian of $\phi$ and can rightfully be called the \textbf{covariant
Laplacian} of $\phi$ and denoted by $\Delta_{g}\phi$ (this is also
referred to as the \textbf{tension field} of $\phi$ in other literature
\citep[pg. 13]{Xin}, which is denoted $\tau\left(\phi\right)$).
Note that $\Delta_{g}\phi$ is a vector field along $\phi$. This
makes sense because $\phi$ is not necessarily a scalar function;
it takes values in $S$. In the case $S=\mathbb{R}$, $\Delta_{g}\phi$
is the ordinary covariant Laplacian on scalar functions. 

A \textbf{harmonic map} is defined as a critical point of the energy
functional $\mathcal{E}\left(\phi\right):=\int_{M}\frac{1}{2}\left|\nnabla\phi\right|_{k}^{2}\, dV_{M}$.
Assuming a fixed boundary (so that the variations vanish on the boundary)
eliminates the boundary Euler-Lagrange equation, the remaining equation
is
\[
\Delta_{g}\phi=0\mbox{ on the interior of }M,
\]
which is the generalization of Laplace's equation. Satisfying Laplace's
equation is a sufficient condition for a map to be a critical point
of the energy functional. There is an abundance of literature concerning
harmonic maps and the analysis thereof \citep{Eells&LeMaire,Giaquinta&Hildebrandt,Nishikawa,Xin}.
\end{example}

\begin{example}[The geodesic equation]
 \label{example:geodesic_equation} A fundamental problem in differential
geometry is determining length-minimizing curves between given points.
If $M$ is a bounded, real interval, and $t$ denotes the standard
real coordinate, then the length functional on curves $\phi\colon M\to S$
is $\mathcal{L}\left(\phi\right):=\int_{M}\left|\phi^{\prime}\right|_{g}\, dt$.
A topological metric $d\colon M\times M\to\mathbb{R}$ on $M$ can
be defined as 
\[
d\left(p,q\right):=\inf\left\{ \mathcal{L}\left(\phi\right)\mid\phi\mbox{ joins \ensuremath{p}to \ensuremath{q}}\right\} .
\]
It can be shown that the length functional $\mathcal{L}\left(\phi\right):=\int_{M}\left|\phi^{\prime}\right|_{h}\, dt$
and the energy functional $\mathcal{E}\left(\phi\right):=\int_{M}\frac{1}{2}\left|\phi^{\prime}\right|_{h}^{2}\, dt$
have identical minimizers. Note that $\phi^{\prime}\in\Gamma\left(\phi^{*}TS\right)$.
It is therefore sufficient to consider the analytically preferable
energy functional.

In this case, the metric $g$ on $M$ is just scalar multiplication
on $\mathbb{R}$. Because $M$ is one-dimensional and $t$ is the
standard real coordinate, $\frac{d}{dt}$ is a global, parallel orthonormal
frame for $TM$, and the $g$-trace of $\nabla^{2}\phi$ (i.e. $\Delta_{g}\phi$)
has a single term. The Euler-Lagrange equation, on the interior of
$M$, is
\[
0=\Delta_{g}\phi=\tr_{g}\nabla^{2}\phi=\nabla\nnabla\phi:\left(\frac{d}{dt},\frac{d}{dt}\right)=\nabla_{\frac{d}{dt}}\nnabla\phi\cdot\frac{d}{dt}=\nabla_{\frac{d}{dt}}\left(\nnabla\phi\cdot\frac{d}{dt}\right)-\nnabla\phi\cdot\nabla_{\frac{d}{dt}}\frac{d}{dt}.
\]
But $\nnabla\phi\cdot\frac{d}{dt}=\phi^{\prime}$ and $\nabla_{\frac{d}{dt}}\frac{d}{dt}=0$,
giving the \textbf{geodesic equation}
\[
\nabla_{\frac{d}{dt}}^{\phi^{*}TS}\phi^{\prime}=0\mbox{ on the interior of }M.
\]
This is the covariant way to state that the acceleration of $\phi$
is identically zero. The geodesic equation is commonly notated as
$0=\nabla_{\phi^{\prime}}\phi^{\prime}$, though such notation is
inaccurate because $\phi^{\prime}$ is not a vector field on $S$,
but a vector field along $\phi$, and therefore use of the pullback
covariant derivative $\nabla^{\phi^{*}TS}$ is correct (see (\ref{remark:pullback_derivative_captures_fiber_variation})).

While formulated using fixed boundary conditions ($\phi$ has $p$
and $q$ as its endpoints), the geodesic equation is a second order
ODE for which initial tangent vector conditions are sufficient to
uniquely determine a solution.
\end{example}

\section{\label{sec:Second-Variation}Second Variation}

A further consideration after finding critical points of the energy
functional $\mathcal{L}$ is determining which critical points are
extrema. This will involve calculating the second derivative of $\mathcal{L}$.
Let $C:=C^{\infty}\left(M,S\right)$, noting that $T_{\phi}C\cong\Gamma\left(\phi^{*}TS\right)$
for $\phi\in C$. The first derivative of $\mathcal{L}$ is $\nabla^{C\to\mathbb{R}}\mathcal{L}:=d\mathcal{L}$,
as seen in the previous section. The second derivative is the covariant
Hessian $\nabla^{T^{*}C}\nabla^{C\to\mathbb{R}}\mathcal{L}$, where
the covariant derivative $\nabla^{T^{*}C}$ is induced by $\nabla^{TS}$
\citep[Theorem 5.4]{Eliasson}. 

For the remainder of this section, let $I,J\subseteq\mathbb{R}$ be
neighborhoods of zero, let $i$ and $j$ be their respective standard
coordinates, and extend the existing $\delta$-style derivative-at-a-point
notation by defining $\delta_{i}:=\frac{\partial}{\partial i}\mid_{i=j=0}$,
$\delta_{j}:=\frac{\partial}{\partial j}\mid_{i=j=0}$, and evaluation
map $z\colon M\to M\times I\times J,\, m\mapsto\left(m,0,0\right)$.
Then $\delta_{i}=z^{*}\partial_{i}$ and $\delta_{j}=z^{*}\partial_{j}$;
these will be used as in the calculation of the first variation. 
\begin{thm}[Second variation of $\mathcal{L}$]
 \label{thm:second_variation_of_L} Let $\mathcal{L}$, $L$, $\sigma$,
$\mu$, $v$ and $\nu$ all be defined as above. If $\phi\in C^{\infty}\left(M,S\right)$
is a critical point of $\mathcal{L}$ and $A,B\in T_{\phi}C\cong\Gamma\left(\phi^{*}TS\right)$,
then the covariant Hessian of $\mathcal{L}$ is 
\begin{align*}
 & \nabla^{2}\mathcal{L}\left(\phi\right):_{T_{\phi}C}\left(A\otimes B\right)\\
={} & \int_{M}A\cdot_{\phi^{*}T^{*}S}\phi_{,M}^{*}L_{,\sigma\sigma}\cdot_{\phi^{*}TS}B+A\cdot_{\phi^{*}T^{*}S}\phi_{,M}^{*}L_{,\sigma v}\cdot_{\phi^{*}TS\otimes_{M}T^{*}M}\nabla^{\phi^{*}TS}B\\
 & +\nabla^{\phi^{*}TS}A\cdot_{\phi^{*}T^{*}S\otimes_{M}TM}\phi_{,M}^{*}L_{,v\sigma}\cdot_{\phi^{*}TS}B+\nabla^{\phi^{*}TS}A\cdot_{\phi^{*}T^{*}S\otimes_{M}TM}\phi_{,M}^{*}L_{,vv}\cdot_{\phi^{*}TS\otimes_{M}T^{*}M}\nabla^{\phi^{*}TS}B\\
 & -A\cdot_{\phi^{*}T^{*}S}\left(\phi_{,M}^{*}L_{,v}\cdot_{\phi^{*}TS\otimes_{M}T^{*}M}\left(\phi^{*}R^{TS}\cdot_{\phi^{*}TS}\phi_{,M}\right)\right)\cdot_{\phi^{*}TS}B\, dV_{g}.
\end{align*}
This is often called the \textbf{second variation} of $\mathcal{L}$.
Here, $R^{TS}\in\Gamma\left(TS\otimes_{S}T^{*}S\otimes_{S}T^{*}S\otimes_{S}T^{*}S\right)$
denotes the Riemannian curvature endomorphism tensor for the Levi-Civita
connection on $TS$.\end{thm}
\begin{proof}
Let $\Phi\colon M\times I\times J\to S$ be a two-parameter variation
such that $\delta_{i}\Phi=A$ and $\delta_{j}\Phi=B$ (e.g. $\Phi\left(m,i,j\right):=\exp\left(iA\left(m\right)+jB\left(m\right)\right)$).
The variation $\Phi$ can be naturally identified with a variation
$\overline{\Phi}\colon I\times J\to C,\,\left(i,j\right)\mapsto\left(m\mapsto\Phi\left(m,i,j\right)\right)$
which is more conducive to the use of $C$ as a manifold. The tensor
products in the generally infinite-dimensional $TC$ are taken formally.
Let $\overline{z}:=\left(0,0\right)\in I\times J$.

By (\ref{prop:chain_rule_for_covariant_hessian}), taking the algebra
formally in the case of infinite-rank vector bundles,
\[
\nabla^{2}\left(\mathcal{L}\circ\overline{\Phi}\right)=\overline{\Phi}^{*}\nabla^{2}\mathcal{L}:_{\overline{\Phi}^{*}TC}\left(\nnabla\overline{\Phi}\boxtimes_{I\times J}\nnabla\overline{\Phi}\right)+\overline{\Phi}^{*}\nabla\mathcal{L}\cdot_{\overline{\Phi}^{*}TC}\nabla\nnabla\overline{\Phi},
\]
so
\begin{align*}
 & \left(\nabla^{2}\mathcal{L}\circ_{C}\phi\right):_{T_{\phi}C}\left(A\otimes B\right)\\
={} & \left(\nabla^{2}\mathcal{L}\circ_{C}\phi\right):_{T_{\phi}C}\left(\delta_{i}\overline{\Phi}\otimes\delta_{j}\overline{\Phi}\right)+\left(\nabla\mathcal{L}\circ_{C}\phi\right)\cdot_{T_{\phi}C}\nabla_{\delta_{j}}^{\overline{\Phi}^{*}TC}\partial_{i}\overline{\Phi} & \mbox{(since \ensuremath{\nabla\mathcal{L}\circ_{C}\phi=0})}\\
={} & \left(\nabla^{2}\mathcal{L}\circ_{C}\overline{\Phi}\circ_{I\times J}\overline{z}\right):_{\overline{z}^{*}\overline{\Phi}^{*}TC}\left(\delta_{i}\overline{\Phi}\otimes\delta_{j}\overline{\Phi}\right)+\left(\nabla\mathcal{L}\circ_{C}\overline{\Phi}\circ_{I\times J}\overline{z}\right)\cdot_{\overline{z}^{*}\overline{\Phi}^{*}TC}\nabla_{\delta_{j}}^{\overline{\Phi}^{*}TC}\partial_{i}\overline{\Phi}\\
={} & \nabla^{2}\left(\mathcal{L}\circ_{C}\overline{\Phi}\right):_{TI\oplus TJ}\left(\delta_{i}\otimes_{I\times J}\delta_{j}\right) & \mbox{(by above)}\\
={} & \delta_{j}\partial_{i}\left(\mathcal{L}\circ_{C}\overline{\Phi}\right)\\
={} & \int_{M}\delta_{j}\partial_{i}\left(L\circ\Phi_{,M}\right)\, dV_{g}\\
={} & \int_{M}\nabla^{2}\left(L\circ\Phi_{,M}\right):_{TM\oplus TI\oplus TJ}\left(\delta_{i}\otimes_{M\times I\times J}\delta_{j}\right)\, dV_{g}\\
={} & \int_{M}A\cdot_{\phi^{*}T^{*}S}\phi_{,M}^{*}L_{,\sigma\sigma}\cdot_{\phi^{*}TS}B+A\cdot_{\phi^{*}T^{*}S}\phi_{,M}^{*}L_{,\sigma v}\cdot_{\phi^{*}TS\otimes_{M}T^{*}M}\nabla^{\phi^{*}TS}B\\
 & +\nabla^{\phi^{*}TS}A\cdot_{\phi^{*}T^{*}S\otimes_{M}TM}\phi_{,M}^{*}L_{,v\sigma}\cdot_{\phi^{*}TS}B\\
 & +\nabla^{\phi^{*}TS}A\cdot_{\phi^{*}T^{*}S\otimes_{M}TM}\phi_{,M}^{*}L_{,vv}\cdot_{\phi^{*}TS\otimes_{M}T^{*}M}\nabla^{\phi^{*}TS}B\\
 & -A\cdot_{\phi^{*}T^{*}S}\left(\phi_{,M}^{*}L_{,v}\cdot_{\phi^{*}TS\otimes_{M}T^{*}M}\left(\phi^{*}R^{TS}\cdot_{\phi^{*}TS}\phi_{,M}\right)\right)\cdot_{\phi^{*}TS}B\, dV_{g} & \mbox{(by Calculation (1)).}
\end{align*}
Supporting calculations follow.\\

Calculation (1): Abbreviate $\phi_{,M}^{*}L_{,xy}$ by $\mathbf{L}_{,xy}$.
By (\ref{prop:chain_rule_for_covariant_hessian}),
\begin{align*}
 & \nabla^{2}\left(L\circ\Phi_{,M}\right):_{TM\oplus TI\oplus TJ}\left(\delta_{i}\otimes_{M\times I\times J}\delta_{j}\right)\\
={} & \left(\left[\Phi_{,M}^{*}\nabla^{2}L:_{\Phi_{,M}^{*}TE}\left(\nnabla\Phi_{,M}\boxtimes_{M\times I\times J}\nnabla\Phi_{,M}\right)+\Phi_{,M}^{*}\nabla L\cdot_{\Phi_{,M}^{*}TE}\nabla\nnabla\Phi_{,M}\right]\circ z\right):_{z^{*}\left(TM\oplus TI\oplus TJ\right)}\left(\delta_{i}\otimes_{M}\delta_{j}\right)\\
={} & z^{*}\Phi_{,M}^{*}\nabla^{2}L:_{z^{*}\Phi_{,M}^{*}TE}\left(\delta_{i}\Phi_{,M}\otimes_{M}\delta_{j}\Phi_{,M}\right)+z^{*}\Phi_{,M}^{*}\nabla L\cdot_{z^{*}\Phi_{,M}^{*}TE}\nabla_{\delta_{j}}^{\Phi_{,M}^{*}TE}\partial_{i}\Phi_{,M}\\
 & \mbox{(by Calculation (2))}\\
={} & \mathbf{L}_{,\sigma\sigma}:_{\phi_{,M}^{*}\pi_{S}^{*}TS}\left(\delta_{i}\Phi\otimes_{M}\delta_{j}\Phi\right)+\mathbf{L}_{,\sigma v}\cdot_{\phi_{,M}^{*}\pi_{S}^{*}TS\otimes_{M}\phi_{,M}^{*}\pi^{*}E}\left(\delta_{i}\Phi\otimes_{M}\nabla^{\phi^{*}TS}\delta_{j}\Phi\right)\\
 & +\mathbf{L}_{,v\sigma}\cdot_{\phi_{,M}^{*}\pi^{*}E\otimes_{M}\phi_{,M}^{*}\pi_{S}^{*}TS}\left(\nabla^{\phi^{*}TS}\delta_{i}\Phi\otimes_{M}\delta_{j}\Phi\right)+\mathbf{L}_{,vv}:_{\phi_{,M}^{*}\pi^{*}E}\left(\nabla^{\phi^{*}TS}\delta_{i}\Phi\otimes_{M}\nabla^{\phi^{*}TS}\delta_{j}\Phi\right)\\
 & +\mathbf{L}_{,\sigma}\cdot_{\phi_{,M}^{*}\pi_{S}^{*}TS}\nabla_{\delta_{j}}^{\Phi^{*}TS}\partial_{i}\Phi+\mathbf{L}_{,v}\cdot_{\phi_{,M}^{*}\pi^{*}E}\nabla^{\phi^{*}TS}\nabla_{\delta_{j}}^{\Phi^{*}TS}\partial_{i}\Phi\\
 & +\mathbf{L}_{,v}\cdot_{\phi_{,M}^{*}\pi^{*}E}\left(\left(\Id_{\phi^{*}TS}\otimes_{M}\delta_{i}\Phi\right)\cdot_{\phi^{*}TS\otimes_{M}\phi^{*}T^{*}S}\left(\phi^{*}R^{TS}\cdot_{\phi^{*}TS}\delta_{j}\Phi\right)\cdot_{\phi^{*}TS}\phi_{,M}\right)\\
 & \mbox{(by Calculation (3)).}
\end{align*}
Note that $\nabla_{\delta_{j}}^{\Phi^{*}TS}\partial_{i}\Phi\in\Gamma\left(\phi^{*}TS\right)$,
and since $\phi$ is a critical point of $\mathcal{L}$,
\[
\int_{M}\mathbf{L}_{,\sigma}\cdot_{\phi_{,M}^{*}\pi_{S}^{*}TS}\nabla_{\delta_{j}}^{\Phi^{*}TS}\partial_{i}\Phi+\mathbf{L}_{,v}\cdot_{\phi_{,M}^{*}\pi^{*}E}\nabla^{\phi^{*}TS}\nabla_{\delta_{j}}^{\Phi^{*}TS}\partial_{i}\Phi\, dV_{g}=0.
\]
Thus
\begin{align*}
 & \int_{M}\nabla^{2}\left(L\circ\Phi_{,M}\right):_{TM\oplus TI\oplus TJ}\left(\delta_{i}\otimes_{M\times I\times J}\delta_{j}\right)\, dV_{g}\\
={} & \int_{M}\mathbf{L}_{,\sigma\sigma}:_{\phi_{,M}^{*}\pi_{S}^{*}TS}\left(\delta_{i}\Phi\otimes_{M}\delta_{j}\Phi\right)+\mathbf{L}_{,\sigma v}\cdot_{\phi_{,M}^{*}\pi_{S}^{*}TS\otimes_{M}\phi_{,M}^{*}\pi^{*}E}\left(\delta_{i}\Phi\otimes_{M}\nabla^{\phi^{*}TS}\delta_{j}\Phi\right)\\
 & +\mathbf{L}_{,v\sigma}\cdot_{\phi_{,M}^{*}\pi^{*}E\otimes_{M}\phi_{,M}^{*}\pi_{S}^{*}TS}\left(\nabla^{\phi^{*}TS}\delta_{i}\Phi\otimes_{M}\delta_{j}\Phi\right)+\mathbf{L}_{,vv}:_{\phi_{,M}^{*}\pi^{*}E}\left(\nabla^{\phi^{*}TS}\delta_{i}\Phi\otimes_{M}\nabla^{\phi^{*}TS}\delta_{j}\Phi\right)\\
 & +\delta_{i}\Phi\cdot_{\phi^{*}T^{*}S}\left(\mathbf{L}_{,v}\cdot_{\phi_{,M}^{*}\pi^{*}E}\left(\left(\phi^{*}R^{TS}\cdot_{\phi^{*}TS}\delta_{j}\Phi\right)\cdot_{\phi^{*}TS}\phi_{,M}\right)\right)\, dV_{g}\\
={} & \int_{M}A\cdot_{\phi^{*}T^{*}S}\mathbf{L}_{,\sigma\sigma}\cdot_{\phi^{*}TS}B+A\cdot_{\phi^{*}T^{*}S}\mathbf{L}_{,\sigma v}\cdot_{\phi^{*}TS\otimes_{M}T^{*}M}\nabla^{\phi^{*}TS}B\\
 & +\nabla^{\phi^{*}TS}A\cdot_{\phi^{*}T^{*}S\otimes_{M}TM}\mathbf{L}_{,v\sigma}\cdot_{\phi^{*}TS}B+\nabla^{\phi^{*}TS}A\cdot_{\phi^{*}T^{*}S\otimes_{M}TM}\mathbf{L}_{,vv}\cdot_{\phi^{*}TS\otimes_{M}T^{*}M}\nabla^{\phi^{*}TS}B\\
 & -A\cdot_{\phi^{*}T^{*}S}\left(\mathbf{L}_{,v}\cdot_{\phi^{*}TS\otimes_{M}T^{*}M}\left(\phi^{*}R^{TS}\cdot_{\phi^{*}TS}\phi_{,M}\right)\right)\cdot_{\phi^{*}TS}B\, dV_{g}\\
 & \mbox{ (by antisymmetry of curvature tensor).}
\end{align*}
\\

Calculation (2):
\begin{align*}
 & z^{*}\nabla\nnabla\Phi_{,M}:_{z^{*}\left(TM\oplus TI\oplus TJ\right)}\left(\delta_{i}\otimes_{M}\delta_{j}\right)\\
={} & z^{*}\nabla^{\Phi_{,M}^{*}TE\otimes_{M\times I\times J}\left(T^{*}M\oplus T^{*}I\oplus T^{*}J\right)}\nnabla\Phi_{,M}:_{z^{*}\left(TM\oplus TI\oplus TJ\right)}z^{*}\left(\partial_{i}\otimes_{M\times I\times J}\partial_{j}\right)\\
={} & z^{*}\left(\nabla^{\Phi_{,M}^{*}TE\otimes_{M\times I\times J}\left(T^{*}M\oplus T^{*}I\oplus T^{*}J\right)}\nnabla\Phi_{,M}:_{TM\oplus TI\oplus TJ}\left(\partial_{i}\otimes_{M\times I\times J}\partial_{j}\right)\right)\\
={} & z^{*}\left(\nabla_{\partial_{j}}^{\Phi_{,M}^{*}TE\otimes_{M\times I\times J}\left(T^{*}M\oplus T^{*}I\oplus T^{*}J\right)}\nnabla\Phi_{,M}\cdot_{TM\oplus TI\oplus TJ}\partial_{i}\right)\\
={} & z^{*}\nabla_{\partial_{j}}^{\Phi_{,M}^{*}TE}\left(\nnabla\Phi_{,M}\cdot_{TM\oplus TI\oplus TJ}\partial_{i}\right)\mbox{ (since \ensuremath{\nabla_{\partial_{j}}^{TM\oplus TI\oplus TJ}\partial_{i}=0})}\\
={} & \nabla_{\delta_{j}}^{\Phi_{,M}^{*}TE}\partial_{i}\Phi_{,M}.
\end{align*}
\\

Calculation (3): As calculated in the proof of (\ref{thm:first_variation_of_L}),
\begin{align*}
\phi_{,M}^{*}\sigma\cdot_{\phi_{,M}^{*}TE}\delta_{i}\Phi_{,M} & =\delta_{i}\Phi\in\Gamma\left(\phi^{*}TS\right),\\
\phi_{,M}^{*}\mu\cdot_{\phi_{,M}^{*}TE}\delta_{i}\Phi_{,M} & =0\in\Gamma\left(TM\right),\\
\phi_{,M}^{*}v\cdot_{\phi_{,M}^{*}TE}\delta_{i}\Phi_{,M} & =\nabla^{\phi^{*}TS}\delta_{i}\Phi\in\Gamma\left(\phi^{*}TS\otimes_{M}T^{*}M\right).
\end{align*}
Furthermore, letting $P:=\pr_{M}^{M\times I\times J}$ for brevity
and noting that $P\circ z=\Id_{M}$,
\begin{align*}
 & \phi_{,M}^{*}\sigma\cdot_{\phi_{,M}^{*}TE}\nabla_{\delta_{j}}^{\Phi_{,M}^{*}TE}\partial_{i}\Phi_{,M}\\
={} & z^{*}\nabla_{\partial_{j}}^{\Phi_{,M}^{*}\pi_{S}^{*}TS}\left(\Phi_{,M}^{*}\sigma\cdot_{\Phi_{,M}^{*}TE}\partial_{i}\Phi_{,M}\right) & \mbox{(since \ensuremath{\nabla\sigma=0})}\\
={} & z^{*}\nabla_{\partial_{j}}^{\left(\pi_{S}\circ\Phi_{,M}\right)^{*}TS}\partial_{i}\Phi & \mbox{(using calculation from (\ref{thm:first_variation_of_L}))}\\
={} & \nabla_{\delta_{j}}^{\Phi^{*}TS}\partial_{i}\Phi\in\Gamma\left(z^{*}\Phi^{*}TS\right)\cong\Gamma\left(\phi^{*}TS\right),
\end{align*}
\begin{align*}
 & \phi_{,M}^{*}\mu\cdot_{\phi_{,M}^{*}TE}\nabla_{\delta_{j}}^{\Phi_{,M}^{*}TE}\partial_{i}\Phi_{,M}\\
={} & z^{*}\nabla_{\partial_{j}}^{\Phi_{,M}^{*}\pi_{M}^{*}TM}\left(\Phi_{,M}^{*}\mu\cdot_{\Phi_{,M}^{*}TE}\partial_{i}\Phi_{,M}\right) & \mbox{(since \ensuremath{\nabla\mu=0})}\\
={} & z^{*}\nabla_{\partial_{j}}^{\left(\pi_{M}\circ\Phi_{,M}\right)^{*}TM}0 & \mbox{(using calculation from (\ref{thm:first_variation_of_L}))}\\
={} & 0\in\Gamma\left(z^{*}\left(\pi_{M}\circ\Phi_{,M}\right)^{*}TM\right)\cong\Gamma\left(z^{*}P^{*}TM\right)\cong\Gamma\left(TM\right),
\end{align*}
\begin{align*}
 & \phi_{,M}^{*}v\cdot_{\phi_{,M}^{*}TE}\nabla_{\delta_{j}}^{\Phi_{,M}^{*}TE}\partial_{i}\Phi_{,M}\\
={} & z^{*}\nabla_{\partial_{j}}^{\Phi_{,M}^{*}\pi^{*}E}\left(\Phi_{,M}^{*}v\cdot_{\Phi_{,M}^{*}TE}\partial_{i}\Phi_{,M}\right) & \mbox{(since \ensuremath{\nabla v=0})}\\
={} & z^{*}\nabla_{\partial_{j}}^{\left(\pi\circ\Phi_{,M}\right)^{*}E}\left(\partial_{i}\Phi\right)_{,M} & \mbox{(using calculation from (\ref{thm:first_variation_of_L})).}
\end{align*}
Note that 
\[
\phi_{,M}^{*}v\in\Gamma\left(\phi_{,M}^{*}\pi^{*}E\otimes_{M}\phi_{,M}^{*}T^{*}E\right)\cong\Gamma\left(\left(\phi\times_{M}\Id_{M}\right)^{*}E\otimes_{M}\phi_{,M}^{*}T^{*}E\right),
\]
and therefore 
\[
\phi_{,M}^{*}v\cdot_{\phi_{,M}^{*}TE}\nabla_{\delta_{j}}^{\Phi_{,M}^{*}TE}\partial_{i}\Phi_{,M}\in\Gamma\left(\left(\phi\times_{M}\Id_{M}\right)^{*}E\right)\cong\Gamma\left(\phi^{*}TS\otimes_{M}T^{*}M\right),
\]
so it suffices to examine its natural pairing with $TM$ elements.
Let $X\in\Gamma\left(TM\right)$, noting that $X=\Id_{M}^{*}X=z^{*}P^{*}X$
and that $P^{*}X=TP\cdot\left(X\oplus0_{TI}\oplus0_{TJ}\right)\in\Gamma\left(P^{*}TM\right)$.
Then
\begin{align*}
 & \left(\phi_{,M}^{*}v\cdot_{\phi_{,M}^{*}TE}\nabla_{\delta_{j}}^{\Phi_{,M}^{*}TE}\partial_{i}\Phi_{,M}\right)\cdot_{TM}X\\
={} & z^{*}\nabla_{\partial_{j}}^{\Phi^{*}TS\otimes_{M\times I\times J}P^{*}T^{*}M}\left(\partial_{i}\Phi\right)_{,M}\cdot_{z^{*}P^{*}TM}z^{*}P^{*}X\\
={} & z^{*}\nabla_{\partial_{j}}^{\Phi^{*}TS}\left(\left(\partial_{i}\Phi\right)_{,M}\cdot_{P^{*}TM}\nnabla P\cdot_{TM\oplus TI\oplus TJ}\left(X\oplus0_{TI}\oplus0_{TJ}\right)\right)\\
 & -z^{*}\left(\left(\partial_{i}\Phi\right)_{,M}\cdot\nabla_{\partial_{j}}^{P^{*}TM}\left(\nnabla P\cdot_{TM\oplus TI\oplus TJ}\cdot\left(X\oplus0_{TI}\oplus0_{TJ}\right)\right)\right)\\
={} & z^{*}\left(\nabla_{\partial_{j}}^{\Phi^{*}TS}\nabla_{X\oplus0_{TI}\oplus0_{TJ}}^{\Phi^{*}TS}\partial_{i}\Phi-\left(\partial_{i}\Phi\right)_{,M}\cdot0_{P^{*}TM}\right)\\
={} & z^{*}\left(\nabla_{X\oplus0_{TI}\oplus0_{TJ}}^{\Phi^{*}TS}\nabla_{\partial_{j}}^{\Phi^{*}TS}\partial_{i}\Phi+\nabla_{\left[\partial_{j},X\oplus0_{TI}\oplus0_{TJ}\right]}^{\Phi^{*}TS}\partial_{i}\Phi-R^{\Phi^{*}TS}\left(\partial_{j},X\oplus0_{TI}\oplus0_{TJ}\right)\partial_{i}\Phi\right)\\
={} & z^{*}\left(\nabla_{X\oplus0_{TI}\oplus0_{TJ}}^{\Phi^{*}TS}\nabla_{\partial_{j}}^{\Phi^{*}TS}\partial_{i}\Phi+\nabla_{0}^{\Phi^{*}TS}\partial_{i}\Phi\right)\\
 & -z^{*}\left(\left(\Id_{\Phi^{*}TS}\otimes_{M\times I\times J}\partial_{i}\Phi\right)\cdot_{\Phi^{*}TS\otimes_{M\times I\times J}\Phi^{*}T^{*}S}R^{\Phi^{*}TS}:_{TM\oplus TI\oplus TJ}\left(\partial_{j}\otimes_{M\times I\times J}\left(X\oplus0_{TI}\oplus0_{TJ}\right)\right)\right)\\
={} & \left[\nabla^{\phi^{*}TS}\nabla_{\delta_{j}}^{\Phi^{*}TS}\partial_{i}\Phi+\left(\Id_{\phi^{*}TS}\otimes_{M}\delta_{i}\Phi\right)\cdot_{\phi^{*}TS\otimes_{M}\phi^{*}T^{*}S}\left(\phi^{*}R^{TS}\cdot_{\phi^{*}TS}\delta_{j}\Phi\right)\cdot_{\phi^{*}TS}\phi_{,M}\right]\cdot_{TM}X,
\end{align*}
where the last equality follows from Calculations (4) and (5). Because
$X$ is pointwise-arbitrary in $TM$, this shows that
\[
\phi_{,M}^{*}v\cdot_{\phi_{,M}^{*}TE}\nabla_{\delta_{j}}^{\Phi_{,M}^{*}TE}\partial_{i}\Phi_{,M}=\left(\nabla_{\delta_{j}}^{\Phi^{*}TS}\partial_{i}\Phi\right)_{,M}+\left(\Id_{\phi^{*}TS}\otimes_{M}\delta_{i}\Phi\right)\cdot_{\phi^{*}TS\otimes_{M}\phi^{*}T^{*}S}\left(\phi^{*}R^{TS}\cdot_{\phi^{*}TS}\delta_{j}\Phi\right)\cdot_{\phi^{*}TS}\phi_{,M}.
\]
\\

Calculation (4):
\begin{align*}
 & z^{*}\left(\nabla_{X\oplus0_{TI}\oplus0_{TJ}}^{\Phi^{*}TS}\nabla_{\partial_{j}}^{\Phi^{*}TS}\partial_{i}\Phi+\nabla_{0}^{\Phi^{*}TS}\partial_{i}\Phi\right)\\
={} & z^{*}\left(\nabla_{\partial_{j}}^{\Phi^{*}TS}\partial_{i}\Phi\right)_{,M}\cdot_{z^{*}P^{*}TM}z^{*}P^{*}X\\
={} & \left(\nabla_{\delta_{j}}^{\Phi^{*}TS}\partial_{i}\Phi\right)_{,M}\cdot_{TM}X\mbox{ (by (\ref{prop:evaluation_commutes_with_noninvolved_derivatives}))}\\
={} & \nabla^{\phi^{*}TS}\nabla_{\delta_{j}}^{\Phi^{*}TS}\partial_{i}\Phi\cdot_{TM}X\mbox{ (because \ensuremath{\nabla_{\delta_{j}}^{\Phi^{*}TS}\partial_{i}\Phi\in\Gamma\left(z^{*}\Phi^{*}TS\right)\cong\Gamma\left(\phi^{*}TS\right)}).}
\end{align*}
\\

Calculation (5):
\begin{align*}
 & -z^{*}\left(R^{\Phi^{*}TS}:_{TM\oplus TI\oplus TJ}\left(\partial_{j}\otimes_{M\times I\times J}\left(X\oplus0_{TI}\oplus0_{TJ}\right)\right)\right)\\
={} & z^{*}\left(R^{\Phi^{*}TS}:_{TM\oplus TI\oplus TJ}\left(\left(X\oplus0_{TI}\oplus0_{TJ}\right)\otimes_{M\times I\times J}\partial_{j}\right)\right) & \mbox{(antisymmetry of \ensuremath{R^{\Phi^{*}TS}})}\\
={} & z^{*}\left(\Phi^{*}R^{TS}:_{\Phi^{*}TS}\left(\nnabla\Phi\boxtimes_{M\times I\times J}\nnabla\Phi\right):_{TM\oplus TI\oplus TJ}\left(\left(X\oplus0_{TI}\oplus0_{TJ}\right)\otimes_{M\times I\times J}\partial_{j}\right)\right) & \mbox{(by (\ref{prop:pullback_curvature_endomorphism}))}\\
={} & z^{*}\left(\Phi^{*}R^{TS}:_{\Phi^{*}TS}\left(\left(\Phi_{,M}\cdot_{P^{*}TM}P^{*}X\right)\otimes_{M\times I\times J}\partial_{j}\Phi\right)\right)\\
={} & z^{*}\left(\left(\Phi^{*}R^{TS}\cdot_{\Phi^{*}TS}\partial_{j}\Phi\right)\cdot_{\Phi^{*}TS}\Phi_{,M}\cdot_{P^{*}TM}P^{*}X\right)\\
={} & \left(z^{*}\Phi^{*}R^{TS}\cdot_{z^{*}\Phi^{*}TS}z^{*}\partial_{j}\Phi\right)\cdot_{z^{*}\Phi^{*}TS}z^{*}\Phi_{,M}\cdot_{z^{*}P^{*}TM}z^{*}P^{*}X\\
={} & \left(\phi^{*}R^{TS}\cdot_{\phi^{*}TS}\delta_{j}\Phi\right)\cdot_{\phi^{*}TS}\phi_{,M}\cdot_{TM}X.
\end{align*}
\end{proof}
\begin{thm}[Second variation of $\mathcal{L}$ (alternate form)]
 \label{thm:second_variation_of_L_alternate} Let $\mathcal{L}$,
$L$, $\sigma$, $\mu$, $v$ and $\nu$ all be defined as above.
If $\phi\in C^{\infty}\left(M,S\right)$ is a critical point of $\mathcal{L}$
and $A,B\in\Gamma\left(\phi^{*}TS\right)$, then
\begin{align*}
 & \nabla^{2}\mathcal{L}\left(\phi\right):_{T_{\phi}C}\left(A\otimes B\right)\\
={} & \int_{M}A\cdot_{\phi^{*}T^{*}S}\mathbf{L}_{,\sigma\sigma}\cdot_{\phi^{*}TS}B+A\cdot_{\phi^{*}T^{*}S}\mathbf{L}_{,\sigma v}\cdot_{\phi^{*}TS\otimes_{M}T^{*}M}\nabla^{\phi^{*}TS}B\\
 & -A\cdot_{\phi^{*}T^{*}S}\Div_{M}\mathbf{L}_{,v\sigma}\cdot_{\phi^{*}TS}B-A\cdot_{\phi^{*}T^{*}S}\mathbf{L}_{,v\sigma}\cdot_{T^{*}M\otimes_{M}\phi^{*}TS}\left(\nabla^{\phi^{*}TS}B\right)^{\left(1\,2\right)}\\
 & -A\cdot_{\phi^{*}T^{*}S}\Div_{M}\mathbf{L}_{,vv}\cdot_{\phi^{*}TS\otimes_{M}T^{*}M}\nabla^{\phi^{*}TS}B\\
 & -A\cdot_{\phi^{*}T^{*}S}\mathbf{L}_{,vv}\cdot_{T^{*}M\otimes_{M}\phi^{*}TS\otimes_{M}T^{*}M}\left(\nabla^{\phi^{*}TS\otimes_{M}T^{*}M}\nabla^{\phi^{*}TS}B\right)^{\left(1\,2\,3\right)}\\
 & -A\cdot_{\phi^{*}T^{*}S}\left(\mathbf{L}_{,v}\cdot_{\phi^{*}TS\otimes_{M}T^{*}M}\left(\phi^{*}R^{TS}\cdot_{\phi^{*}TS}\phi_{,M}\right)\right)\cdot_{\phi^{*}TS}B\, dV_{g}\\
 & +\int_{\partial M}\left(A\cdot_{\phi^{*}T^{*}S}\mathbf{L}_{,v\sigma}\cdot_{\phi^{*}TS}B\right)\cdot_{T^{*}M}\nu+\left(A\cdot_{\phi^{*}T^{*}S}\mathbf{L}_{,vv}\cdot_{\phi^{*}TS\otimes_{M}T^{*}M}\nabla^{\phi^{*}TS}B\right)\cdot_{T^{*}M}\nu\, d\overline{V}_{g}
\end{align*}
\end{thm}
\begin{proof}
This result follows essentially from (\ref{thm:second_variation_of_L})
via several instances of integration by parts to express the integrand(s)
entirely in terms of $A$ and not its covariant derivatives. Abbreviate
$\phi_{,M}^{*}L_{,xy}$ by $\mathbf{L}_{,xy}$. Then, integrating
by parts allows the covariant derivatives of $A$ to be flipped across
the natural pairings over $\phi^{*}TS$.
\begin{align*}
 & \int_{M}\nabla^{\phi^{*}TS}A\cdot_{\phi^{*}T^{*}S\otimes_{M}TM}\mathbf{L}_{,v\sigma}\cdot_{\phi^{*}TS}B\, dV_{g}\\
={} & \int_{M}\tr_{TM}\left(\left(\nabla^{\phi^{*}TS}A\right)^{\left(1\,2\right)}\cdot_{\phi^{*}TS}\mathbf{L}_{,v\sigma}\cdot_{\phi^{*}TS}B\right)\, dV_{g} & \mbox{(\ensuremath{TM}trace is taken separately)}\\
={} & \int_{M}\tr_{TM}\left(\nabla^{TM}\left(A\cdot_{\phi^{*}T^{*}S}\mathbf{L}_{,v\sigma}\cdot_{\phi^{*}TS}B\right)\right)\\
 & -\tr_{TM}\left(A\cdot_{\phi^{*}T^{*}S}\nabla^{\phi^{*}T^{*}S\otimes_{M}TM\otimes_{M}\phi^{*}TS}\mathbf{L}_{,v\sigma}\cdot_{\phi^{*}TS}B\right)\\
 & -\tr_{TM}\left(A\cdot_{\phi^{*}T^{*}S}\mathbf{L}_{,v\sigma}\cdot_{\phi^{*}TS}\nabla^{\phi^{*}TS}B\right)\, dV_{g} & \mbox{(reverse product rule)}\\
={} & \int_{M}-A\cdot_{\phi^{*}T^{*}S}\Div_{M}\mathbf{L}_{,v\sigma}\cdot_{\phi^{*}TS}B & \mbox{(definition of divergence)}\\
 & -A\cdot_{\phi^{*}T^{*}S}\mathbf{L}_{,v\sigma}\cdot_{T^{*}M\otimes_{M}\phi^{*}TS}\left(\nabla^{\phi^{*}TS}B\right)^{\left(1\,2\right)}\, dV_{g}\\
 & +\int_{\partial M}\left(A\cdot_{\phi^{*}T^{*}S}\mathbf{L}_{,v\sigma}\cdot_{\phi^{*}TS}B\right)\cdot_{T^{*}M}\nu\, d\overline{V}_{g} & \mbox{(divergence theorem).}
\end{align*}
Similiarly, 
\begin{align*}
 & \int_{M}\nabla^{\phi^{*}TS}A\cdot_{\phi^{*}T^{*}S\otimes_{M}TM}\mathbf{L}_{,vv}\cdot_{\phi^{*}TS\otimes_{M}T^{*}M}\nabla^{\phi^{*}TS}B\, dV_{g}\\
={} & \int_{M}\tr_{TM}\left(\left(\nabla^{\phi^{*}TS}A\right)^{\left(1\,2\right)}\cdot_{\phi^{*}T^{*}S}\mathbf{L}_{,vv}\cdot_{\phi^{*}TS\otimes_{M}T^{*}M}\nabla^{\phi^{*}TS}B\right)\, dV_{g}\\
={} & \int_{M}\tr_{TM}\left(\nabla^{TM}\left(A\cdot_{\phi^{*}T^{*}S}\mathbf{L}_{,vv}\cdot_{\phi^{*}TS\otimes_{M}T^{*}M}\nabla^{\phi^{*}TS}B\right)\right)\\
 & -\tr_{TM}\left(A\cdot_{\phi^{*}T^{*}S}\nabla^{\phi^{*}T^{*}S\otimes_{M}TM\otimes_{M}\phi^{*}T^{*}S\otimes_{M}TM}\mathbf{L}_{,vv}\cdot_{\phi^{*}TS\otimes_{M}T^{*}M}\nabla^{\phi^{*}TS}B\right)\\
 & -\tr_{TM}\left(A\cdot_{\phi^{*}T^{*}S}\mathbf{L}_{,vv}\cdot_{\phi^{*}TS\otimes_{M}T^{*}M}\nabla^{\phi^{*}TS\otimes_{M}T^{*}M}\nabla^{\phi^{*}TS}B\right)\, dV_{g}\\
={} & \int_{M}-A\cdot_{\phi^{*}T^{*}S}\Div_{M}\mathbf{L}_{,vv}\cdot_{\phi^{*}TS\otimes_{M}T^{*}M}\nabla^{\phi^{*}TS}B\\
 & -A\cdot_{\phi^{*}T^{*}S}\mathbf{L}_{,vv}\cdot_{T^{*}M\otimes_{M}\phi^{*}TS\otimes_{M}T^{*}M}\left(\nabla^{\phi^{*}TS\otimes_{M}T^{*}M}\nabla^{\phi^{*}TS}B\right)^{\left(1\,2\,3\right)}\, dV_{g}\\
 & +\int_{\partial M}\left(A\cdot_{\phi^{*}T^{*}S}\mathbf{L}_{,vv}\cdot_{\phi^{*}TS\otimes_{M}T^{*}M}\nabla^{\phi^{*}TS}B\right)\cdot_{T^{*}M}\nu\, d\overline{V}_{g}.
\end{align*}
Together with (\ref{thm:second_variation_of_L}), this gives the desired
result.
\end{proof}

\section{Questions and Future Work}

This paper is a first pass at the development of a strongly-typed
tensor calculus formalism. The details of its workings are by no means
complete or fully polished, and its landscape is riddled with many
tempting rabbit holes which would certainly produce useful results
upon exploration, but which were out of the scope of a first exposition.
Here is a list of some topics which the author considers worthwhile
to pursue, and which will likely be the subject of his future work.
Hopefully some of these topics will be inspiring to other mathematicians,
and ideally will start a conversation on the subject.
\begin{itemize}
\item There refinements to be made to the type system used in this paper
in order to achieve better error-checking and possibly more insight
into the relevant objects. There are still implicit type identifications
being done (mostly the canonical identifications between different
pullback bundles).
\item The calculations done in this paper are not in an optimally polished
and refined state. With experience, certain common operations can
be identified, abstract computational rules generated for these operations,
and the relevant calculations simplified.
\item The language of Category Theory can be used to address the implicit/explicit
handling of natural type identifications, for example, the identification
used in showing the contravariance of bundle pullback; $\psi^{*}\phi^{*}F\cong\left(\phi\circ\psi\right)^{*}F$.
\item The details of the particular implementation of the pullback bundle
$\phi^{*}F$ as a submanifold of the direct product $M\times F$ are
used in this paper, but there is no reason to ``open up the box''
like this. For most purposes, the categorical definition of pullback
bundle suffices; the pullback bundle can be worked exclusively using
its projection maps $\pi_{M}^{\phi^{*}F}$ and $\rho_{F}^{\phi^{*}F}$.
In the author's experience (which occurred too late to be incorporated
into this paper), using this abstract interface cleans up calculations
involving pullback bundles significantly.
\item The type system used for any particular problem or calculation can
be enriched or simplified to adjust to the level of detail appropriate
for the situation. For example, if $\gamma\in C^{\infty}\left(\mathbb{R},M\right)$,
then $\nnabla\gamma\in\Gamma\left(\gamma^{*}TM\otimes_{\mathbb{R}}T^{*}\mathbb{R}\right)$,
but if $t$ is the standard coordinate on $\mathbb{R}$, then $\nnabla\gamma=\gamma^{\prime}\otimes_{\mathbb{R}}dt$,
where $\gamma^{\prime}\in\Gamma\left(\gamma^{*}TM\right)$ is given
by $\nnabla\gamma\cdot\frac{d}{dt}$. This ``primed'' derivative
has a simpler type than the total derivative, and would presumably
lead to simplier calculations (e.g. in (\ref{prop:conserved-quantity}).
This ``primed'' derivative could also be used in the derivation
of the first and second variations. While this would simplify the
type system, it would diversify the notation and make the computational
system less regularized. However, some situations may benefit overall
from this.
\item The notion of strong typing comes from computer programming languages.
The human-driven type-checking which is facilitated by the pedantically
decorated notation in this paper can be done by computer by implementing
the objects and operations of this tensor calculus formalism in a
strongly typed language such as Haskell. This would be a step toward
automated calculation checking, and could be considered a step toward
automated proof checking from the top down (as opposed to from the
bottom up, using a system such as the Coq Proof Assistant).
\item Is there some sort of completeness result about the calculational
tools and type system in this paper? In other words, is it possible
to accomplish ``everything'' in a global, coordinate-free way using
a certain set of tools, such as pullback bundles, covariant derivatives,
chain rules, permutations, evaluation-by-pullback?
\item The alternate form of the second variation (see (\ref{thm:second_variation_of_L_alternate}))
can be used to form a generalized Jacobi field equation for a particular
energy functional. Analysis of this equation and its solutions may
give insights analogous to the standard (geodesic-based) Jacobi field
equation.
\end{itemize}

\section*{Acknowledgements}

I would like to express my gratitude to the ARCS (Achievement Rewards
for College Scientists) Foundation for their having awarded me a 2011-2012
ARCS Fellowship, and for their generous efforts to promote excellence
in young scientists. I would like to thank my advisor Debra Lewis
for trusting in my abilities and providing me with the freedom in
which the creative endeavor that this paper required could flourish.
I would like to thank David DeConde for the invaluable conversations
at the Octagon in which imagination, creativity, and exploration were
gladly fostered. Thanks to Chris Shelley for showing me how to create
the tensor diagrams using Tikz. Finally, I would like to thank both
Debra and David for their help in editing this paper.

\nocite{Dodson&Radivoiovici,Ebin&Marsden,Palais,Xin}

\bibliographystyle{plain}
\bibliography{/Users/vdods/files/git/documents/school/references}

\end{document}